\documentclass[10pt,reqno]{amsart}
\usepackage{microtype}
\usepackage{mathtools}\mathtoolsset{showonlyrefs} %
\usepackage{tikz}
\usetikzlibrary{arrows.meta}
\usepackage[a4paper,margin=2.2cm]{geometry}
\usepackage[T1]{fontenc}
\usepackage{amsthm}
\usepackage{fourier,amsfonts,amssymb}
\usepackage{xcolor}
\usepackage[authoryear]{natbib}
\usepackage[hidelinks]{hyperref}
\usepackage{enumitem}
\allowdisplaybreaks%
\numberwithin{equation}{section}
\numberwithin{figure}{section}
\theoremstyle{plain}
\newtheorem{theorem}{Theorem}[section]

\newtheorem{lemma}[theorem]{Lemma}
\newtheorem{corollary}[theorem]{Corollary}
\newtheorem{openproblem}[theorem]{Open problem}
\newtheorem{definition}[theorem]{Definition}
\newtheorem{remark}[theorem]{Remark}

\title{Singular central limit theorems for the spherical ensemble and beyond}
\date{Spring, 2026}
\author{Djalil Chafaï}
\address{DMA, École normale supérieure, CEREMADE, Université Paris-Dauphine, PSL, CNRS, France}
\email{djalil.chafai@ens.psl.eu}
\author{David García-Zelada} 
\address{LPSM, Sorbonne Université, Paris, France}
\email{david.garcia-zelada@sorbonne-universite.fr}
\author{Yuan Yuan Xu}
\address{Academy of Mathematics and Systems Science,
Chinese Academy of Sciences, Beijing, China}
\email{yyxu2023@amss.ac.cn}

\newcommand{\R}{{\mathbb{R}}}
\newcommand{\N}{{\mathbb{N}}}

\newcommand{\C}{{\mathbb{C}}}
\newcommand{\A}{{\tilde{A}}}
\newcommand{\B}{{\tilde{B}}}
\newcommand{\ii}{\mathrm{i}}

\newcommand{\wt}{\widetilde}

\newcommand{\<}{\langle}
\renewcommand{\>}{\rangle}

\newcommand{\AG}{A^{\mathrm{Gin}}}
\newcommand{\BG}{B^{\mathrm{Gin}}}
\newcommand{\MG}{M^{\mathrm{Gin}}}
\newcommand{\E}{\mathbb{E}}
\renewcommand{\P}{{\mathbb P}}
\newcommand{\dd}{\operatorname{d}\!{}}
\newcommand{\ie}{\emph{i.e., }}
\newcommand{\eg}{\emph{e.g., }}
\newcommand{\cf}{\emph{cf., }}
\renewcommand{\Re}{\operatorname{Re}}
\renewcommand{\Im}{\operatorname{Im}}
\newcommand{\II}{\mathcal{L}_{i,t}}

\begin{document}

\begin{abstract}
We study the fluctuations of logarithmic Green singularities in the spherical ensemble, viewed as a random discretization of the two-sphere. Smooth observables exhibit the usual Sobolev or Gaussian free field fluctuations, whereas logarithmic singularities live on a larger logarithmic scale and asymptotically decouple in high-dimension, producing an explicit white-noise limit. The result gives precise asymptotics for logarithmic potentials and characteristic polynomials, with constants expressed through chordal geometry on the sphere.
\end{abstract}

\maketitle

{\footnotesize\tableofcontents}

\section{Introduction}

 The spherical ensemble is a unitary invariant random matrix model obtained
 by taking the ratio of two independent complex Ginibre matrices. Its
 spectrum is a heavy-tailed determinantal Coulomb gas on the complex plane.
 Its equilibrium measure is the complex Cauchy distribution. This planar gas
 is the image by the stereographic projection of a chordal gas on the
 two-sphere. Our first main result is a Central Limit Theorem (CLT) for linear
 statistics for test functions with an arbitrary finite number of
 logarithmic singularities. The normalization is at the order of the square root of the 
 logarithm of the dimension.  The contributions of the singularities are asymptotically 
 independent, and the asymptotic variance depends only on the
 weights of the singularities. We deduce a multipoint CLT for the logarithm
 of the modulus of the characteristic polynomial, which is the logarithmic
 potential of the empirical spectral distribution. We also deduce a
 multipoint CLT for the linear statistics tested against the Green function of the
 two-sphere. As a second main result, we compute covariance asymptotics showing that the 
 associated unnormalised fields are log-correlated. The approach relies on an interplay
 between probabilistic, analytic, and geometric arguments. Furthermore, 
 we also establish, under suitable assumptions, a matrix-universality extension
 when the Ginibre matrices are replaced by Girko
 matrices. Finally, we discuss possible geometric universality extensions in which the
 sphere is replaced by a more general compact Riemann surface, as well as
 natural open problems related to Fisher--Hartwig asymptotics and Gaussian
 multiplicative chaos.

In short, this work is about the high-dimensional spectral analysis of random matrices from the spherical ensemble, a famous model of mathematical physics. We obtain CLTs for test functions with logarithmic singularities. This covers in particular the asymptotic analysis of the logarithm of the modulus of the characteristic polynomial.

The analysis of the characteristic polynomial of random matrices goes back at least to \cite{zbMATH01631908}. For non-Hermitian random matrix models and two-dimensional Coulomb gases, we refer for instance to \cite{zbMATH01921763}, \cite{zbMATH07115604}, \cite{MR3946715}, \cite{zbMATH08053185}, \cite{bourgadeetal}.

\subsection{Spherical model}

Let $A$ and $B$ be independent random $n\times n$ matrices with iid entries
$\mathcal{N}_{\mathbb{C}}(0,1)$. Then $A$ and $B$ are almost surely invertible
and the joint law of the eigenvalues of $AB^{-1}$ has Lebesgue density
proportional to
\begin{equation}\label{eq:gas:C}
  (z_1,\ldots,z_n)\in\mathbb{C}^n
  \mapsto\prod_{k=1}^n(1+|z_k|^2)^{-(n+1)}\prod_{i<j}|z_i-z_j|^2
  =\mathrm{e}^{-2(n+1)\sum_{k=1}^nQ(z_k)}\prod_{i<j}|z_i-z_j|^2
\end{equation}
where the confining potential $Q(z)=\frac{1}{2}\log(1+|z|^2)$ is logarithmic at infinity. This Coulomb gas on $\mathbb{C}$ has equilibrium measure given by the complex Cauchy distribution\footnote{It is also the law of the ratio of two independent $\mathcal{N}_{\mathbb{C}}(0,1)$ random variables (just take $n=1$).}
\begin{equation}
  \mu(\mathrm{d}z)
  =\frac{\mathrm{d}^2z}{\pi(1+|z|^2)^2}
  =\frac{\Delta Q(z)}{2\pi}\mathrm{d}^2z.
\end{equation}
This model goes back to \cite{Krishnapur2006,Krishnapur2009}. It has a natural analog in free probability, see \cite{zbMATH05361525}, and was the object of further works in random matrix theory such as \cite{MR2540391}, \cite{MR2346510}, \cite{MR2772389}, and \cite{MR2926763}, and more recently \cite{byun2025propertiesonecomponentcoulombgas}. 

The probability distribution $\mu$ is also the image, by the stereographic projection $T$ of the uniform distribution $\nu$ on the unit sphere
$\mathbb{S}^2=\{x\in\mathbb{R}^3:x_1^2+x_2^2+x_3^2=1\}$. In other words,
\begin{equation}
  \mu=\nu\circ T^{-1}.
\end{equation}
A precise definition of $T$ is given in \eqref{eq:T}. The planar Coulomb gas \eqref{eq:gas:C} is also the image, by the stereographic projection $T$, of the gas on the unit sphere $\mathbb{S}^2$
with density with respect to the uniform measure proportional to
\begin{equation}\label{eq:gas:S}
  (x_1,\ldots,x_n)\in(\mathbb{S}^2)^n
  \mapsto\prod_{i<j}{\|x_i-x_j\|}_{\mathbb{R}^3}^2.
\end{equation}
This gas is invariant under 
the isometries of $\mathbb S^2$ or, equivalently, under the action of the orthogonal group
$\mathrm{O}(3)$ on $\mathbb R^3$
restricted to $\mathbb S^2$, and its equilibrium measure is $\nu$. The random matrix model
$AB^{-1}$ is known as the spherical ensemble. For all these aspects, we refer
to \cite[Chapter 4]{MR2552864}, \cite{MR2641363} and references therein. The spherical ensemble can be seen as a two-dimensional analogue of the circular unitary ensemble (CUE). In particular, it does not have boundary. Note however that for CUE, the random matrix model is on the circle side, while for the spherical model, the random matrix model is on the plane side rather than on the sphere side.

Our main object of interest is the characteristic polynomial of $AB^{-1}$ at point $z\in\mathbb{C}$, that we denote
\begin{equation}
  P_n(z)=\det(zI-AB^{-1}).
\end{equation}
Let $Z_n=(Z_{n,1},\ldots,Z_{n,n})$ be a random vector of $\mathbb{C}^n$
distributed according to the gas \eqref{eq:gas:C}, in other words distributed as the eigenvalues of $AB^{-1}$. For all measurable\footnote{Here $f$ is not necessarily defined on the whole complex plane but on a subset of full (Lebesgue) measure.}
$f:\mathbb{C}\to\mathbb{R}$, we consider the linear statistics
\begin{equation}
  L_n(f)=\sum_{k=1}^n f(Z_{n,k})=
  \int_{\mathbb C} f\mathrm{d}L_n
  \quad\text{where}\quad
  L_n=\sum_{k=1}^n\delta_{Z_{n,k}}.
\end{equation}
The characteristic polynomial $P_n$ of $AB^{-1}$ and the logarithmic potential of
$\eta_n$ are related via
\begin{equation}
  \log|P_n(z)|
  =L_n(f_z)
  =(\log\left|\cdot\right|*L_n)(z)
  \quad\text{where}\quad
  f_z(w)=\log|w-z|.
\end{equation}
The case $z=0$ is easy to study, and regarding the high-dimensional asymptotic analysis, we have
\begin{equation}\label{eq:logpot0:CLT}
  \mathbb{E}(L_n(f_0))=0,
  \quad
  \mathrm{Var}(L_n(f_0))=\frac{H_n}{2}+\frac{n}{2}\Bigl(\frac{\pi^2}{6}-H_n^{(2)}\Bigr)=\frac{H_n+1}{2}+o(1), %
  \quad\text{and}\quad
  \frac{L_n(f_0)}{\sqrt{\frac{1}{2}\log(n)}}
  \xrightarrow[n\to\infty]{\mathrm{d}}
  \mathcal{N}(0,1),
\end{equation}
where 
\begin{equation}
    H_n=\sum_{k=1}^n\frac{1}{k}
    \quad\text{and}\quad
    H_n^{(2)}=\sum_{k=1}^n\frac{1}{k^2} 
\end{equation}
are the first and second order $n$-th harmonic numbers. Indeed, 
\begin{equation}
    L_n(f_0)=\log|\det(AB^{-1})|=\log|\det(A)|-\log|\det(B)|.
\end{equation}
Since $A$ and $B$ are iid, $L_n(f_0)$ is centered, and \eqref{eq:logpot0:CLT} can be
deduced from the analysis of $\log|\det(A)|$, which in turn follows from the reduction to the sum of independent random variables. More precisely, by using the Hermitization $|\det(A)|^2=\det(AA^*)$ into the Wishart or Laguerre matrix $AA^*$, and the Cholesky--Bartlett factorization, we get that 
\begin{equation}\label{eq:log|det|Gamma}
|\det(A)|^2=\det(AA^*) \overset{\mathrm{d}}{=}\prod_{k=1}^n\Gamma_k
\quad\text{where $\Gamma_1,\ldots,\Gamma_n$ are independent, $\Gamma_k\sim\mathrm{Gamma}(k,1)$.}
\end{equation}
This gives the formula for $\mathrm{Var}(L_n(f_0))$. Moreover it allows to use the Lindeberg CLT via moments, or the Fréchet--Shohat CLT via cumulants to get \eqref{eq:logpot0:CLT}. We refer to \cite[Prop.~2.1 and Th.~3.3]{zbMATH05547558}. 
There are multiple other ways to prove \eqref{eq:logpot0:CLT}, for
instance by using the determinantal structure via the Kostlan observation for $A$ or for $AB^{-1}$ to get a reduction to a sum of independent random variables. 
More generally, we refer for instance to \cite{zbMATH01191059}, \cite{zbMATH06265853}, and \cite{MR3352055} for more on the subject of the log-determinant of random matrices.

Beyond \eqref{eq:logpot0:CLT}, our goal is to establish the CLT for $L_n(f_z)$ for $z\neq0$, and more generally the CLT for $L_n(f)$ for an arbitrary test function $f$ with a finite number of logarithmic singularities. This requires new methods and ideas.

The CLT \eqref{eq:logpot0:CLT} does not follow immediately from \cite{PhysRevLett.75.69,MR1894104} since the test function is not bounded, or from \cite{MR2346510,MR4721731,zbMATH06113080,bourgoin} since the test function $f_0$ is not in $H^1(\overline{\mathbb{C}})$ (singularity\footnote{Indeed $z\mapsto|\nabla\log|z||=\frac{1}{|z|}$ is locally Lebesgue integrable at $0$, but not its square!}), or from the main result of \cite{MR3788208,zbMATH07654445,zbMATH08155333,bourgadeetal} since $\mu$ is not compactly supported. We also refer to \cite{MR2361453,zbMATH07431029,MR2817648,MR4552960,zbMATH08155333} for more results on CLTs for linear statistics of two-dimensional Coulomb gases, and to \cite{MR2552864}, \cite{zbMATH06340358}, \cite{MR4408512} for some links with the zeros of Gaussian analytic functions. 

\subsection{Notation} 
\label{sub:notation} We typically use $x,y$ for points on $\mathbb{S}^2$,
$p,q$ for points on $\overline{\mathbb{C}}$, and $z,w$ for points on $\mathbb{C}$. For $n$--dependent positive
quantities $a_n,b_n$ we use the notation $a_n\ll b_n$ to denote that $\lim_{n\to\infty} (a_n/b_n)=0$. 

\subsection{CLTs for logarithmic singularities}
To explore the phenomenon behind the CLT for $L_n(f_0)$ in \eqref{eq:logpot0:CLT} and target a CLT for the logarithmic potential $L_n(f_z)$, we introduce the following concept, inspired by complex analysis.

\begin{definition}[Logarithmic singularity]\label{df:main:LogSing}
	Let $p_1,\dots,p_k$ be distinct points in $\overline{\mathbb C}$.
	We say that a smooth 
    function\footnote{Recall that a function
    $f: U \to \mathbb R$ 
    defined on an open subset $U$ of $\mathbb C$ is
    smooth if it is
    infinitely differentiable on $U \setminus \{\infty\}$
    and, whenever $\infty \in U$, 
    the function $z \mapsto f(1/z)$ is infinitely
    differentiable
    on a neighborhood of $0$.}
	$f:\overline{\mathbb C} \setminus \{p_1,\dots,p_k\} \to \mathbb{R}$ has logarithmic singularities at $p_1,\dots,p_k$  with weights $c_1,\dots,c_k \in \mathbb R \setminus \{0\}$ when
	\begin{itemize}
    \item whenever $p_i\neq\infty$, the function $z \mapsto f(z) - c_i\log|z-p_i|$ can be smoothly extended in a neighborhood of $p_i$;
	\item whenever $p_i=\infty$, the function $z \mapsto f(1/z) - c_i \log|z|$ can be smoothly extended in a neighborhood of $0$.
	\end{itemize}
\end{definition}

From the point of view of electrostatics, the weight $c_i$ is the quantity of charge at $p_i$, hence the notation\footnote{A more traditional notation could be $q_i$, but in this writing, we use the notation $q$ for a typical point in $\overline{\mathbb{C}}$.}. 

The basic example that will give us a lot of information about the other ones is the function $f_0$ which has a logarithmic singularity of weight $1$ at $0$ and one of weight $-1$ at $\infty$. Another, more general example is to take any meromorphic function on $\overline {\mathbb C}$, namely a quotient of polynomials $Q(z) = P_1(z)/P_2(z)$, and to consider $\log|Q|$. This is the logarithmic potential of particles located at the zeros and at the poles of $Q$, having a positive charge for a zero and a negative charge for a pole. The main results of our work are about the fluctuations of these kinds of functions.

Thinking of $\mathbb S^2\subset\mathbb{R}^3$ as the extended plane $\overline{\mathbb C}$ via the stereographic projection leads naturally to consider the (half) chordal distance $d$ on $\overline{\mathbb{C}}$, see \eqref{eq:dchord} and \eqref{eq:d}. For all $p\in\overline{\mathbb{C}}$, the function $g_p=\log d(\cdot,p)$ has a single logarithmic singularity, at point $p$, with weight $1$. An equivalent and enlightening way of thinking about $f$ in Definition \ref{df:main:LogSing} is 
\begin{equation}\label{eq:LogSingAlt}
    f=h+\sum_{i=1}^kc_i\log d(\cdot,p_i)
    =h+\sum_{i=1}^kc_ig_{p_i}
\end{equation}
where $h:\overline{\mathbb{C}}\to\mathbb{R}$ is smooth, hence bounded and Lipschitz. In particular, this implies that $f\in L^p(\mu)$ for all $p\in[1,\infty)$.

We are now ready to state our first main result. The message is that the asymptotic variance depends on the test function only via the singularities weights. This phenomenon reminds the residue method in complex analysis.

\begin{theorem}[CLT for linear statistics with multiple-singularities test function]\label{th:LogSing}
	If $f:\overline{\mathbb{C}}\setminus\{p_1,\ldots,p_k\}\to\mathbb{R}$ is smooth, and has logarithmic singularities at the distinct points $p_1,\dots,p_k \in \overline{\mathbb C}$
	with weights $c_1,\dots,c_k\in\mathbb{R}\setminus\{0\}$, then
    \begin{equation}\label{eq:LogSing}
      \frac{L_n(f)-\mathbb{E}(L_n(f))}{\sqrt{\frac{1}{4}\log n}}
      =\frac{L_n(f)-n\int f\mathrm d \mu}{\sqrt{\frac{1}{4}\log n}} 
      \xrightarrow[n \to \infty]{\mathrm{d}}
	   \mathcal{N}(0, c_1^2+\cdots+c_k^2).
    \end{equation}
\end{theorem}

A proof of Theorem \ref{th:LogSing} is given in Section \ref{se:proof:th:LogSing}. The single singularity case is extracted from the special case of two-singularities \eqref{eq:logpot0:CLT}. A natural way to understand the phenomenon behind \eqref{eq:LogSing} is to start from the decomposition \eqref{eq:LogSingAlt}. The contributions coming from the singular parts $\log d(\cdot,p_i)$ are not independent, but the off-diagonal decay \eqref{eq:kernel-distance} of the determinantal kernel of the model yields an asymptotic decoupling on the logarithmic scale in the high-dimensional regime. On the other hand, the contribution of the smooth part $h$ is $O(1)$ thanks to Lemma \ref{lem:lip-variance}, in accordance with the GFF CLT \eqref{eq:GFFCLT}. It is thus absorbed by the logarithmic normalization in high dimension.

The smooth part of $f$ may depend on \(n\), provided its contribution remains negligible
at the logarithmic scale. More precisely, if $f_n=h_n+\sum_{i=1}^k c_i g_{p_i}$
with fixed $c_1,\ldots,c_k$ and $p_1,\ldots,p_k$, then the conclusion of Theorem \ref{th:LogSing} with $f_n$ instead of $f$ holds as soon as $\|h_n\|_{\mathrm{Lip.}}^2=o(\log n)$, thanks to Lemma \ref{lem:lip-variance}. 
A careful reading of the proof of Theorem \ref{th:LogSing} shows that if $p_1,\ldots,p_k$ are also allowed to depend on $n$, $f_n=h_n+\sum_{i=1}^kc_ig_{p_{n,i}}$, then the argument is uniform as long as they remain separated at a fixed positive distance. More generally, if 
\[
\Delta_n:=\min_{1\leq i\neq j\leq k}d(p_{n,i},p_{n,j})=o_{n\to\infty},
\]
then the proof of Theorem \ref{th:LogSing}, with the Lipschitz bound of Lemma \ref{lem:lip-variance} replaced by a quantitative version of the variance estimate \eqref{eq:GFFCLT:cov}, would give the same conclusion under the stronger sufficient condition 
\begin{equation}
      \|h_n\|_{\mathrm{Lip.}}^2=o(\log n)
      \qquad\text{and}\qquad
      \Delta_n=n^{-o(1)}.
\end{equation}
This comes from the variance bound of the smooth remainders created by
localizing the logarithmic singularities at scale comparable with
$\Delta_n$. It implies the condition $\Delta_n \gg \sqrt{\log n/n}$ that gives exponential decoupling via \eqref{eq:kernel-distance}. 

By using the linearity of $L_n$ and the stability of Gaussians by linear transforms, the CLT \eqref{eq:LogSing} is equivalent to a joint CLT for an arbitrary finite number of test functions with multiple singularities. 

\bigskip

Recall that for all $p\in\overline{\mathbb{C}}$, the function $g_p=\log d(\cdot,p)$ has a single logarithmic singularity, at point $p$, with weight $1$. It turns out that the function $(p,q)\in\overline{\mathbb{C}}\times\overline{\mathbb{C}}\mapsto g_q(p)=\log d(p,q)=g_p(q)$ is, up to additive and multiplicative constants, nothing else but the Green function of the round sphere $\mathbb{S}^2\equiv\overline{\mathbb{C}}$, in stereographic coordinates, see \eqref{eq:G}. It is one of the most natural and important functions with logarithmic singularities.
We have the following CLT.

\begin{corollary}[Green function : CLT]\label{co:Green:CLT}
 In the sense of finite-dimensional distributions,
 \begin{equation}\label{eq:Green:multi}
 {\left(\frac{L_n(g_p)+\frac{n}{2}}{\sqrt{\frac{1}{4}\log(n)}}\right)}_{p\in\overline{\mathbb{C}}}
 \xrightarrow[n\to\infty]{\mathrm{d}}
 {(\zeta_p)}_{p\in \overline{\mathbb{C}}}
 \end{equation}
 where ${(\zeta_p)}_{p\in{\overline{\mathbb{C}}}}$ is a white noise, a centered real Gaussian process with covariance $\mathbb{E}(\zeta_p\zeta_q)=\mathbb{1}_{p=q}$, $p,q\in \overline{\mathbb{C}}$. In particular
 \begin{equation}
 \frac{L_n(g_p)+\frac{n}{2}}{\sqrt{\frac{1}{4}\log(n)}}
 \xrightarrow[n\to\infty]{\mathrm{d}}
 \mathcal{N}(0,1)\label{eq:Green:CLT}.
 \end{equation}
\end{corollary}

Corollary \ref{co:Green:CLT} is proved in Section \ref{se:proof:co:Green:CLT}.

We consider now $f_z$, $z\in\mathbb{C}$, which has two logarithmic singularities, at $z$ and at $\infty$, of weights $1$ and $-1$.

\begin{corollary}[Logarithmic potential : CLT]\label{co:LogPot:CLT}
 In the sense of finite dimensional distributions,
  \begin{equation}\label{eq:LogPot:multi}
 {\left(\frac{L_n(f_z)-\frac{n}{2}\log(1+|z|^2)}{\sqrt{\frac{1}{2}\log(n)}}\right)}_{z\in\mathbb{C}}
 \xrightarrow[n\to\infty]{\mathrm{d}}
 {(\xi_z)}_{z\in\mathbb{C}}
 \end{equation}
 where ${(\xi_z)}_{z\in\mathbb{C}}$ is a centered real Gaussian process with covariance $\mathbb{E}(\xi_z\xi_w)=\frac{\mathbb{1}_{z=w}+1}{2}$, $z,w\in\mathbb{C}$. In particular,
 \begin{equation}
	  \frac{L_n(f_z)-\frac{n}{2}\log(1+|z|^2)}{\sqrt{\frac12\log n}}
    \xrightarrow[n\to\infty]{\mathrm{d}}\mathcal{N}(0,1).\label{eq:logpot:CLT}
 \end{equation}
\end{corollary}

Corollary \ref{co:LogPot:CLT} is proved in Section \ref{se:proof:co:LogPot:CLT}.
The special case $z=0$ in Corollary \ref{co:LogPot:CLT} is exactly \eqref{eq:logpot0:CLT}.

An analogue of Corollary \ref{co:LogPot:CLT} for the complex Ginibre ensemble is
obtained in \cite{MR3946715} using Fisher--Hartwig and Riemann--Hilbert
methods, and more recently in \cite{bourgadeetal} using Fisher--Hartwig and
loop equations methods. An advantage of the spherical ensemble over the complex
Ginibre ensemble lies in its absence of boundary. This is already at the heart of our previous work \cite{chafai-garcia-zelada-xu}.

\subsection{Correlations}

The following result says that ${(L_n(g_p))}_{p\in\overline{\mathbb{C}}}$ is a singular log-correlated random field, with a variance that blows up as $\log(n)$, while the off-diagonal covariance tends to an affine deformation of the log-distance.

\begin{theorem}[Green function : variance and covariance]\label{th:Green:Cov}
 For all $p,q\in\overline{\mathbb{C}}$, $p\neq q$, 
 \begin{equation}
 \mathrm{Var}(L_n(g_p))
 =\frac{H_n}{4}
 \quad\text{and}\quad 
 \mathrm{Cov}(L_n(g_p),L_n(g_q))
  =-\frac{1}{4}-\frac{1}{2}\log d(p,q)+o_{n\to\infty}(1)\label{eq:Green:cov}.
 \end{equation}
\end{theorem}

Theorem \ref{th:Green:Cov} is proved in Section \ref{se:proof:th:Green:Cov}.

The covariance asymptotics from \eqref{eq:Green:cov}, although expected\footnote{More precisely, according to \eqref{eq:GFFCLT} and \eqref{eq:GFFCLT:cov}, if $g_p$ and $g_q$ were regular, the limit covariance between $L_n(g_p)$ and $L_n(g_q)$ 
would be proportional to $\int\langle \nabla g_p(z), \nabla g_q(z) \rangle\mathrm{d}^2z=-\int g_p(z) \Delta g_q(z)\mathrm{d}^2z= -2\pi\int g_p \mathrm d (\delta_q - \mu)$ which coincides with $-\log d(p,q)$ up to constants.} from the GFF CLT \eqref{eq:GFFCLT:cov}, is not straightforward due to the singularities. Our proof of \eqref{eq:Green:cov} is based on the expansion of the Green function in terms of spherical harmonics, and the spectral decomposition of the Berezin transform of the determinantal kernel. 

Back to the characteristic polynomial, our third main result below says that ${(L_n(f_z))}_{z\in\mathbb{C}}$ is a singular log-correlated random field, with a variance and an off-diagonal covariance that both blow up as $\log(n)$ in high-dimension $n$.

\begin{corollary}[Logarithmic potential: variance and covariance]\label{co:LogPot:Cov}
  For all $z,w\in\mathbb{C}$, $z\neq w$, 
  \begin{align}    
    \mathrm{Var}(L_n(f_z))
	  &= -\frac{\log(1+|z|^2)}{2}+\frac{H_n+1}{2}+o_{n\to\infty}(1)\label{eq:LogPot:var}\\
    \mathrm{Cov}(L_n(f_z),L_n(f_w))
    &=-\frac{\log|z-w|}{2}+\frac{H_n+1}{4}+o_{n\to\infty}(1).\label{eq:LogPot:cov}
  \end{align}
\end{corollary}

Corollary \ref{co:LogPot:Cov} is proved in Section \ref{se:proof:co:LogPot:Cov}. The starting point is the metric splitting \eqref{eq:split-euclid} that expresses $f_z$ as $g_z-g_\infty$ up to an additive constant. The presence of $g_\infty$ in both $f_z$ and $f_w$ explains the presence of a $\log(n)$ divergence in their covariance. The formulas 
\eqref{eq:LogPot:var} and \eqref{eq:LogPot:cov} are deduced from \eqref{eq:Green:cov}.

\bigskip

Our proofs of Theorems \ref{th:LogSing} and \ref{th:Green:Cov} and Corollaries \ref{co:Green:CLT}, \ref{co:LogPot:CLT}, and \ref{co:LogPot:Cov}, are remarkably rather simple.
They do not use Fisher--Hartwig/Toeplitz determinant approach, Riemann--Hilbert problems, loop equations or Ward identities, Schwinger-Dyson equations, transport/energy-splitting, or Stein method. 

\subsection{Universality}

Corollary \ref{co:LogPot:CLT} and Theorem \ref{th:LogSing} about the spherical ensemble suggest different types of generalization or universality, related to \eqref{eq:gas:C} and \eqref{eq:gas:S} respectively.
\begin{enumerate}
    \item Matrix universality: CLT for the logarithm of the modulus of the characteristic polynomial for the ratio of independent complex Ginibre matrices (square matrices with iid entries of law $\mathcal{N}_{\mathbb{C}}(0,1)$).
    \item Beta-ensemble universality: CLT for linear statistics of test functions with logarithmic singularities, for the gas at inverse temperature $\beta$, namely with $\prod_{i<j}|z_i-z_j|^\beta$, beyond the determinantal case $\beta=2$.
    \item Geometric universality: CLT for test functions with logarithmic singularities for a model with general potential and for the determinantal process described in \cite{zbMATH06113080} for a more general surface.
\end{enumerate}
The matrix universality is considered in Theorem \ref{th:logpot:univ} below. The geometric universality is discussed in Section \ref{se:geometric}. We do not discuss the Beta universality, see also \cite{MR3788208,zbMATH07431029,MR4552960}.

\bigskip

Following \cite{chafai-garcia-zelada-xu}, we consider the following assumptions or conditions on an $n\times n$ Girko matrix $A$: 
\begin{enumerate}[label=(Cond\arabic*)]
\item\label{it:condensity} the law of $A_{11}$ has a bounded density function $\varphi$ with respect to the Lebesgue measure on $\mathbb{C}$, and
\begin{equation*}
\|\varphi\|_\infty\leq D_0\quad\text{ for some constant $D_0>0$.}
\end{equation*}
\item\label{it:condmom4} \emph{Fourth moment matching condition.}
There exists a small constant $c_0>0$ such that
\begin{equation*}
\mathcal{E}_1=\mathcal{E}_2=0, \quad \qquad \mathcal{E}_3 \leq n^{-\frac{1}{2}-c_0}, \quad\qquad \mathcal{E}_4 \leq  n^{-c_0},
\end{equation*}
where ($k$-th moment deviation of $A_{11}$ from standard complex Gaussian)
\begin{equation*}
\mathcal{E}_k=\max_{\substack{a,b\in\N\\a+b=k}} \left|\mathbb{E}[(\Re A_{11})^a(\Im A_{11})^b]-\mathbb{E}[(\Re \AG_{11})^a(\Im \AG_{11})^b]\right|.
\end{equation*}
\item\label{it:condmom} \emph{Finite moment condition.} 
For each $k\geq 1$, there exists constants $D_k>0$ (independent of $n$) such that
\begin{equation*}
\mathbb{E}[|A_{11}|^k] < D_k.
\end{equation*}

\end{enumerate}

\begin{theorem}[CLT : matrix universality]\label{th:logpot:univ}
 Theorem \ref{th:LogSing}, Corollaries \ref{co:Green:CLT}-\ref{co:LogPot:CLT},
 Theorem \ref{th:Green:Cov} up to the replacement of $\frac{H_n}{4}$ by $\frac{H_n}{4}+o_{n\to\infty}(1)$ in \eqref{eq:Green:cov}, and Corollary \ref{co:LogPot:Cov}
  are universal, in the sense that they remain valid when the spherical ensemble \eqref{eq:gas:C} on $\mathbb{C}$  is replaced by the spectrum of $AB^{-1}$ where $A$ and $B$ are two independent copies of Girko matrices under conditions \ref{it:condensity}-\ref{it:condmom4}-\ref{it:condmom}. 
\end{theorem}

\begin{remark}[Real spherical ensemble and universality]
One would expect to extend the above results to the real case $M=AB^{-1}$ with $A,B$ being independent copies of real-valued Girko matrices. The proof of matrix universality is essentially the same for both the complex and real cases. However, analogous results for the real spherical ensemble $M^{\mathrm{Gin}(\R)}=A^{\mathrm{Gin}(\R)} \big(B^{\mathrm{Gin}(\R)}\big)^{-1}$ is not proved yet. Indeed, the eigenvalues of the real spherical ensemble can be studied using Pfaffians instead
 of determinants as explained in
 \cite{zbMATH06125015}.   
 \end{remark}

A proof of Theorem \ref{th:logpot:univ} is given in Section \ref{se:proof:th:logpot:univ}.

\subsection{Non-normal CLT for radial power test functions}

Let us consider the linear statistics $L_n(k_s)$ for the test function\footnote{Known as the Riesz kernel. We can recover $f_0$ from $k_s$ as $s\to0$ through
$\left.\tfrac{\mathrm{d}}{\mathrm{d}s}\right|_{s=0}k_s(z)=\lim_{s\to0^+}\frac{|z|^{-s}-1}{s}=-\log|z|=-f_0(z)$, $z\neq0$.}
\begin{equation}
k_s(z)=|z|^{-s}, \quad s\in\mathbb{R}, z\neq0.
\end{equation}
The case $s=0$ is void: $k_0\equiv1$, $L_n(k_0)\equiv n$.
The invariance of the spherical ensemble by inversion gives, for all $s\in\mathbb R$,
\begin{equation}
L_n(k_s)\overset{\mathrm d}=L_n(k_{-s}).
\end{equation}
In particular, the law of $L_n(k_s)$ depends only on $|s|$, reducing the study to
$s>0$. This symmetry converts the case $s>0$, which is singular at the origin, into the case $s<0$, which is singular at infinity, and for which the heavy-tail nature of the spherical ensemble expresses itself naturally. The following theorem states that the high-dimensional fluctuations of
$L_n(k_s)$ are not Gaussian in general. The statistic
$L_n(k_s)$ is dominated, after centering when this is meaningful, by extreme radial particles, which is a typical heavy-tail phenomenon. Universality for the statistics of extreme eigenvalues is proved in \cite{chafai-garcia-zelada-xu} under a four moment matching condition. 

\begin{theorem}[High-dimensional analysis of linear statistics of radial
  powers]\label{th:radial-powers}
  Let $s>0$.
  \begin{enumerate}[label=\textup{(\roman*)}]
  \item \emph{Mean.}
   We have $L_n(k_s)\in L^1$ iff $s<2$, and in that case, 
    \begin{align*}
      \mathbb{E}(L_n(k_s))
      &=
      \sum_{k=1}^n
      \frac{\Gamma(k-\frac{s}{2})\Gamma(n+1-k+\frac{s}{2})}
      {\Gamma(k)\Gamma(n+1-k)}\\
      \frac{\mathbb{E}(L_n(k_s))}{n}
      &=%
        \int_0^\infty \frac{r^{-\frac{s}{2}}}{(1+r)^2}\mathrm{d}r
        =
        \Gamma(1+\tfrac{s}{2})\Gamma(1-\tfrac{s}{2})
        =
        \frac{\pi\frac{s}{2}}{\sin(\pi\frac{s}{2})}.
    \end{align*}
    Moreover if $s\ge2$, then $L_n(k_s)\not\in L^1$ for all $n$.
  \item \emph{Variance.} We have $L_n(k_s)\in L^2$ iff $s<1$,
    and in that case, 
    \begin{align*}
      \mathrm{Var}(L_n(k_s))
      &=
        \sum_{k=1}^n
        \left[
        \frac{\Gamma(k-s)\Gamma(n+1-k+s)}
        {\Gamma(k)\Gamma(n+1-k)}
        -
        \left(
        \frac{\Gamma(k-\frac{s}{2})\Gamma(n+1-k+\frac{s}{2})}
        {\Gamma(k)\Gamma(n+1-k)}
        \right)^2
        \right]
        \\
      \frac{\mathrm{Var}(L_n(k_s))}{n^s}
      &\xrightarrow[n\to\infty]{}
      \sum_{\ell=1}^{\infty}
      \left[
      \frac{\Gamma(\ell-s)}{\Gamma(\ell)}
      -
      \left(
      \frac{\Gamma(\ell-\frac{s}{2})}{\Gamma(\ell)}
      \right)^2
      \right]
      <\infty.
      \end{align*}
      Moreover if $s\ge1$, then $L_n(k_s)\not\in L^2$ for all $n$.
    \item \emph{Subcritical regime $s<2$.}
      Then
      \[
        \frac{L_n(k_s)-\mathbb{E}(L_n(k_s))}{n^{\frac{s}{2}}}
        \xrightarrow[n\to\infty]{\mathrm d}
        Y_s=
        \sum_{\ell=1}^{\infty}
        \left(
          G_\ell^{-\frac{s}{2}}
          -
          \frac{\Gamma(\ell-\frac{s}{2})}{\Gamma(\ell)}
        \right)
      \]
      where $G_1,G_2,\ldots$ are independent random variables with
      $G_\ell\sim\mathrm{Gamma}(\ell,1)$.\\ Moreover, the random series $Y_s$ converges
      almost surely, and is in $L^2$ iff $s<1$. 
    \item \emph{Critical regime $s=2$.} We have $L_n(k_2)\not\in L^1$,
    in particular $L_n(k_2)\not\in L^2$, for all $n$, and
      \[
        \frac{L_n(k_2)}{n}-\left(H_{n-1}-1\right)
        \xrightarrow[n\to\infty]{\mathrm d}
        Y_2=
        G_1^{-1}
        +
        \sum_{\ell=2}^{\infty}
        \left(
          G_\ell^{-1}-\frac1{\ell-1}
        \right).
      \]
      Equivalently, $\frac{L_n(k_2)}{n}-\log n \xrightarrow[n\to\infty]{\mathrm d} Y_2+\gamma-1$
        where $\gamma$ is the Euler constant,
      \item \emph{Supercritical regime $s>2$.} Then
        $L_n(k_s)\not\in L^1$, in particular 
    $L_n(k_s)\not\in L^2$, for all $n$, and
    \[
      \frac{L_n(k_s)}{n^{\frac{s}{2}}}
      \xrightarrow[n\to\infty]{\mathrm d}
      Z_s  = \sum_{\ell=1}^{\infty}G_\ell^{-\frac{s}{2}}
    \]
    and the random series $Z_s$ converges almost surely.     
\end{enumerate}
\end{theorem}

Theorem \ref{th:radial-powers} is a direct consequence of the Kostlan observation \eqref{eq:kostlan2}, as in Section \ref{ss:meanvarcum}. We omit the proof.

When $0<s<1$, the variance diverges, yet the normalized limit is not Gaussian. This shows that the unbounded singular test function $k_s$ falls outside the Costin--Lebowitz universality mechanism.

\subsection{CLT for counting function} It is natural to ask about the behavior for other singular test functions such that indicators of subsets. In this special case, the linear statistics is a counting function. Let us specialize to radial subsets, namely the indicator $\mathbb{1}_{D(0,R)}$ of the closed centered disc $D(0,R)=\{z\in\mathbb{C}:|z|\leq R\}$ of radius $R>0$. The following theorem states that the high-dimensional fluctuation of $L_n(\mathbb{1}_{D(0,R)})$ is Gaussian, at scale $n^{1/4}$ however.

The invariance by inversion of the spherical ensemble gives here
$L_n(\mathbb{1}_{D(0,R)})\overset{\mathrm{d}}{=}L_n(\mathbb{1}_{D(0,R^{-1})^c})=n-L_n(\mathbb{1}_{D(0,R^{-1})})$.

\begin{theorem}[CLT for disk counting function]\label{th:count}
For every fixed $0<R<\infty$,
\begin{align*}
\mathbb{E}(L_n(\mathbb{1}_{D(0,R)}))
&=\frac{R^2}{1+R^2}n,\quad
\mathrm{Var}(L_n(\mathbb{1}_{D(0,R)}))
\underset{n\to\infty}{\sim}
\frac{R}{\sqrt{\pi}(1+R^2)}\sqrt{n}\\
\frac{L_n(\mathbb{1}_{D(0,R)})-\frac{R^2}{1+R^2}n}
{n^{\frac{1}{4}}}
&\xrightarrow[n\to\infty]{\mathrm{d}}
\mathcal{N}\Bigr(0,\frac{R}{\sqrt{\pi}(1+R^2)}\Bigr).
\end{align*}
\end{theorem}

Theorem \ref{th:count} is a direct consequence of the Kostlan observation \eqref{eq:kostlan2}, as in Section \ref{ss:meanvarcum}, or the Costin--Lebowitz type result for determinantal point process from \cite{MR1894104}.  We omit the proof. 

More generally, it is known that for two-dimensional determinantal Coulomb gases and the indicators of regular subsets of the bulk, the variance behaves like $\sqrt{n}$. We refer for instance to \cite[Th.~1.6]{LaurentBenoit} for this result and the higher determinantal generalizations. For one-dimensional determinantal log-gases and the indicator of a finite union of intervals, the variance is proportional to the logarithm of the dimension, and the contributions of the singularities are asymptotically independent at the logarithmic scale. The same works for test functions with logarithmic singularities. A comprehensive classification of the nature of the fluctuation depending on the kind of singularity of the test function and the type of decay of the two-points correlation is still to be done. 

\subsection{About Fisher--Hartwig asymptotics.}
An idea that goes back at least to the works of Brézin, who should not be confused with Berezin, is to see the Fourier transform or characteristic function of $L_n(f_z)$ as a complex moment of the modulus of the characteristic polynomial. More precisely, for all $\theta\in\mathbb{R}$,
\begin{equation}
    \mathbb{E}(\mathrm{e}^{\mathrm{i}\theta L_n(f_z)})
    =\mathbb{E}(|P_n(z)|^{\mathrm{i}\theta}).    
\end{equation}
This gives a natural alternative point of view on the CLT for the logarithmic
potential, explored in \cite{MR3946715} for the complex Ginibre ensemble, and
more recently in \cite{bourgadeetal} for two-dimensional gases with compactly
supported equilibrium measure. See also
\cite{byun2025propertiesonecomponentcoulombgas,byun2025orthogonalpolynomialssphericalensemble,byun2026freeenergyexpansiondeterminantal}
for a link with charge insertions and related questions. The Fisher--Hartwig
asymptotics, see \cite{fisher-hartwig}, for the complex Ginibre ensemble
obtained in \cite{MR3946715} together with our Corollaries \ref{co:LogPot:CLT} and
\ref{co:LogPot:Cov} suggest the following.

\begin{openproblem}[Fisher--Hartwig asymptotics]\label{op:LogPot:FH}
Show that for all $z\in\mathbb{C}$ and $\gamma\in\mathbb{C}$ with $-2<\Re\gamma<2$,
\begin{equation}\label{eq:FH}
    \mathbb{E}(|P_n(z)|^\gamma)
	=\Bigl(\frac{n}{1+|z|^2}\Bigr)^{\frac{\gamma^2}{4}}
    (1+|z|^2)^{\frac{\gamma}{2}n}
    \frac{1+o_{n\to\infty}(1)}{G(1+\frac{\gamma}{2})G(1-\frac{\gamma}{2})}
\end{equation}
uniformly with respect to $z$ and $\gamma$ on compact subsets of $\mathbb{C}$ and of $\{\gamma\in\mathbb{C}:-2<\Re\gamma <2\}$ respectively. More generally, for all $z_1,\ldots,z_k\in\mathbb C$ distinct and $\gamma_1,\ldots,\gamma_k\in\mathbb{C}$ with $\Re\gamma_i>-2$, $1\leq i\leq k$, and $\Re(\gamma_1+\cdots+\gamma_k)<2$,
\[
\mathbb{E}\Bigl(\prod_{i=1}^k|P_n(z_i)|^{\gamma_i}\Bigr)
=\prod_{i=1}^k\Bigl(\frac{n}{1+|z_i|^2}\Bigr)^{\frac{\gamma_i^2}{4}}
\prod_{i<j}\Bigl(\frac{n}{|z_i-z_j|^2}\Bigr)^{\frac{\gamma_i\gamma_j}{4}}
\prod_{i=1}^k(1+|z_i|^2)^{\frac{\gamma_i}{2}n}
\frac{1+o_{n\to\infty}(1)}
{G\bigl(1-\frac{\gamma_1+\cdots+\gamma_k}{2}\bigr)\prod_{i=1}^kG\bigl(1+\frac{\gamma_i}{2}\bigr)}.
\]
\end{openproblem}
Here $G$ is the Barnes $G$-function.
Taking $z=0$ matches the formula for the Mellin transform of $|\det(AB^{-1})|$ or characteristic function of $\log|\det(AB^{-1})|$ that 
can be deduced from the explicit representation \eqref{eq:log|det|Gamma}, namely
\begin{equation}
\mathbb{E}(|P_n(0)|^\gamma)
=\prod_{k=1}^n\frac{\Gamma(k+\frac{\gamma}{2})\Gamma(k-\frac{\gamma}{2})}{\Gamma(k)^2}
=n^{\frac{\gamma^2}{4}}\frac{1+o_{n\to\infty}(1)}{G(1+\frac{\gamma}{2})G(1-\frac{\gamma}{2})}.
\end{equation}

Regarding $g_p$ instead of $f_z$, we get from \eqref{eq:logmgf}, for all $p\in\overline{\mathbb{C}}$ and $\gamma\in\mathbb{C}$ with $\Re\gamma>-2$, the formula and asymptotics
\begin{equation}\label{eq:Green:FH}
    \mathbb{E}(\mathrm{e}^{\gamma L_n(g_p)})
    =\Bigl(\frac{\Gamma(n+1)}{\Gamma(n+1+\frac{\gamma}{2})}\Bigr)^n
    \frac{G(n+1+\frac{\gamma}{2})}{G(1+\frac{\gamma}{2})G(n+1)}
    =\mathrm{e}^{-\frac{\gamma}{2}n-\frac{\gamma}{4}-\frac{\gamma^2}{8}}
    (2\pi)^{\frac{\gamma}{4}}    
    n^{\frac{\gamma^2}{8}}
    \frac{1+o_{n\to\infty}(1)}{G(1+\frac{\gamma}{2})}.
\end{equation}
More generally, Corollary \ref{co:Green:CLT} and Theorem \ref{th:Green:Cov} suggest the following.

\begin{openproblem}[Multi-point Fisher--Hartwig asymptotics for Green field]\label{op:Green:FH:mult}
Show that for all distinct $p_1,\ldots,p_k\in\overline{\mathbb{C}}$ and all $\gamma_1,\ldots,\gamma_k\in\mathbb{C}$ with $\Re\gamma_i>-2$ for all $1\leq i\leq k$, we have, denoting $\Gamma=\gamma_1+\cdots+\gamma_k$,
\begin{equation}\label{eq:Green:FH:mult}
    \mathbb{E}(\mathrm{e}^{\gamma_1L_n(g_{p_1})+\cdots+\gamma_kL_n(g_{p_k})})
=\mathrm{e}^{-\frac{\Gamma}{2}n-\frac{\Gamma}{4}-\frac{\Gamma^2}{8}}
(2\pi)^{\frac{\Gamma}{4}}
\frac{n^{\frac18(\gamma_1^2+\cdots+\gamma_k^2)}}
{\prod_{i<j}d(p_i,p_j)^{\frac{\gamma_i\gamma_j}{2}}}
\frac{1+o_{n\to\infty}(1)}
{\prod_{i=1}^kG\left(1+\frac{\gamma_i}{2}\right)}.
\end{equation}
\end{openproblem}

The case of two antipodal singularities such as $0$ and $\infty$ can be solved using the Kostlan observation.

The formula for the Green function is nicer than for the characteristic polynomial, due to the absence of the hidden singularity at $\infty$. On the other hand, the absence of a $2\pi$ factor for the characteristic polynomial formula is a consequence of charge neutrality after adding the hidden singularity at infinity.

There is a duality between the spherical ensemble and the truncated unitary ensemble, considered in \cite[Prop.~5.3~and~5.4]{forrester2025dualitiesrandommatrixtheory}. In particular, the even integer moments of the modulus of the characteristic polynomial are related. The Fisher--Hartwig asymptotics for the truncated unitary model is considered in \cite{deano2025asymptoticsclassplanarorthogonal}.

\subsection{About Wick exponential and Gaussian Multiplicative Chaos}

The random fields ${(L_n(g_p))}_{p\in\overline{\mathbb{C}}}$ and ${(L_n(f_z))}_{z\in\mathbb{C}}$ cannot converge pointwise due to their blowing variance (Theorem \ref{th:Green:Cov} and Corollary \ref{co:LogPot:Cov}). A basic idea behind the works \cite{zbMATH03960673,zbMATH06370363} is that if $Z$ is a Gaussian random variable, then for all real $\gamma$,
\begin{equation}
    \mathrm{e}^{\gamma Z-\gamma\mathbb{E}(Z)-\frac{1}{2}\gamma^2\mathrm{Var}(Z)}
    =\frac{\mathrm{e}^{\gamma Z}}{\mathbb{E}(\mathrm{e}^{\gamma Z})}.
\end{equation}
This renormalization is positive and has unit mean. It is known to yield nontrivial limits for suitable regularizations of singular Gaussian fields. In mathematical physics, it is denoted $:\mathrm{e}^{\gamma Z}:$ and is called the Wick exponential. Now for all $p\in\overline{\mathbb{C}}$, our $L_n(g_p)$ is asymptotically Gaussian thanks to Corollary \ref{co:Green:CLT}, has a variance blowing logarithmically and logarithmic correlations thanks to Theorem \ref{th:Green:Cov}, and these facts strongly suggest the following.

\begin{openproblem}[Gaussian Multiplicative Chaos]\label{op:Green:GMC}
Show that for at least all $0<\gamma<2$, there exists a random finite positive Borel measure $\mathrm{GMC}_\gamma$ on $\mathbb{S}^2$ such that for the topology of weak convergence of measures, denoting $g_x$ for $g_{T(x)}$,
    \begin{equation}\label{eq:Green:GMC}
        \frac{\mathrm{e}^{\gamma L_n(g_x)}}{\mathbb{E}(\mathrm{e}^{\gamma L_n(g_x)})}
        \mathrm{d}\nu(x)        
        \xrightarrow[n\to\infty]{\mathrm{d}}
        \mathrm{GMC}_\gamma.
    \end{equation}
\end{openproblem}

Both sides in \eqref{eq:Green:GMC} are random positive Borel measures on $\mathbb{S}^2$. The left hand side has mean $\nu$ and is absolutely continuous with respect to $\nu$. The limit $\mathrm{GMC}_\gamma$ is expected to be almost surely singular with respect to $\nu$. The covariance normalization suggests the larger subcritical range $0<\gamma<2\sqrt2$, but the range $0<\gamma<2$ is the natural second-moment range, often referred to as the $L^2$ phase.

A proof of \eqref{eq:Green:GMC} could involve regularization by cutoff and \eqref{eq:Green:FH:mult}, as in \cite{bourgadeetal}. Regarding the field $L_n(f_z)=L_n(g_z)-L_n(g_\infty)+nC_z$, see \eqref{eq:split-euclid}, the high-dimensional behavior of the Wick exponential 
\begin{equation}
    \frac{\mathrm{e}^{\gamma L_n(f_z)}}{\mathbb{E}(\mathrm{e}^{\gamma L_n(f_z)})}
    =\frac{|P_n(z)|^\gamma}{\mathbb{E}(|P_n(z)|^\gamma)}
\end{equation}
is more subtle due to the blowing covariance (Corollary \ref{co:LogPot:Cov}). In addition to \cite{bourgadeetal}, we refer to \cite{zbMATH06514487,zbMATH07207623} for the Wick exponential of linear spectral statistics, and
\cite{cipolloni2026gaussianmultiplicativechaosiid} for Girko universality.
See also \cite{chatterjee2026exactcalculationschargeneutrality} for related exact
Coulomb-gas and imaginary GMC calculations.

\subsection{About Cayley transform and 1D Cauchy ensemble.}
Using the stereographic projection to study the spherical ensemble is
already the key point of \cite{MR2926763,chafai-garcia-zelada-xu}. In the same spirit, the Cayley transform
\begin{equation}
  C(z)=\frac{z-\mathrm{i}}{z+\mathrm{i}}
\end{equation}
is a special Möbius transformation. It maps the extended upper half-plane
$\{z\in\mathbb{C}:\Im z\geq0\}\cup\{\infty\}$ to the closed unit disc
$\{z\in\mathbb{C}:|z|\leq1\}$, namely the Poincaré hyperbolic models. It also maps their boundaries, the compactified line
$\overline{\mathbb{R}}=\mathbb{R}\cup\{\infty\}$ to the unit circle
$\mathbb{S}^1$, $C(\infty)=1$. Regarding Coulomb gases, it maps the one-dimensional Cauchy ensemble to the circular ensemble, see \cite{MR2559435}. It turns out that
the reverse Cayley transform, coincides, up to two symmetries, on
$\mathbb{S}^1$, with the one-dimensional stereographic projection:
for all $(x_1,x_2)=x_1+\mathrm{i}x_2\in\mathbb{S}^1\setminus\{1\}\subset\mathbb{C}$,
\begin{equation}
  C^{-1}(x)=\mathrm{i}\frac{x+1}{1-x}=-\frac{x_2}{1-x_1}=-(\tau\circ S)(x)
\end{equation}
where $S$ is the symmetry exchanging $x_1$ and $x_2$, and where $\tau(x_1,x_2)=x_1/(1-x_2)$, is the one-dimensional stereographic projection from $\mathbb{S}^1\subset\mathbb{R}^2$ to $\overline{\mathbb{R}}$. The unit circle
$\mathbb{S}^1$ is globally invariant by $S$ while $\overline{\mathbb{R}}$ is
globally invariant by the symmetry $t\mapsto -t$. In the one-dimensional
situation, the natural random matrix model is associated to the gas on the
circle (Haar unitary random matrices), very well studied, while in the
two-dimensional situation, the natural model is rather associated to the gas
on the complex plane ($AB^{-1}$ with $A$ and $B$ iid complex Ginibre).

\section{Known properties of the spherical ensemble}
\label{se:properties}

\subsection{Stereographic projection.}
The two-sphere 
$\mathbb{S}^2\subset\mathbb{R}^3$ and the
compactified complex plane $\overline{\mathbb{C}}=\mathbb{C}\cup\{\infty\}$
are two equivalent descriptions of
the Riemann sphere. They are related, for instance, by the stereographic projection
$T:\mathbb{S}^2\to\overline{\mathbb{C}}$ with respect to the north pole $e_3 = (0,0,1)$
given by
\begin{equation}\label{eq:T}
  T(x)=\frac{x_1+\mathrm{i}x_2}{1-x_3}
  \quad\text{if}\quad x\neq e_3\quad\text{and}\quad T(e_3)=\infty.
\end{equation}
We give an ancient Greek geometric derivation in Figure \ref{fi:stereo}. In terms of
spherical angles, it writes
$T(x)=\mathrm{e}^{\mathrm{i\varphi}}\cot\bigl(\frac{\theta}{2}\bigr)$, where
$x=(x_1,x_2,x_3)=(\sin(\theta)\cos(\varphi),\sin(\theta)\sin(\varphi),\cos(\theta))$,
$\theta\in[0,\pi]$, $\varphi\in[0,2\pi)$. 
A holomorphic atlas of $\mathbb{S}^2$
can be obtained
using 
any pair of stereographic projections from two distinct points.

In the explicit covariance calculations we will prefer to 
think our objects as living on the 
sphere $\mathbb S^2$ in $\mathbb R^3$
and points will be denoted by $x$ or $y$,
as remarked in Subsection \ref{sub:notation}.
Otherwise, when the formulas are better behaved
in stereographic coordinates $z$, $w$ or $u$ in $\mathbb C$, we chose to use
$\overline {\mathbb C}$ whose points we denote by
$p$ or $q$.

\subsection{Rotational symmetry} The gas \eqref{eq:gas:S} on $\mathbb{S}^2$ is
invariant under the action of the orthogonal group $\mathrm{O}(3)$, which is the isometry group
of the round sphere. The
subgroup $\mathrm{SO}(3)$
of isometries that preserve the orientation is conjugated under the stereographic projection $T$ to
the projective special unitary group $\mathrm{PSU}(2)$ as
\begin{equation}
 \{T\circ R\circ T^{-1}:R\in\mathrm{SO}(3)\} = \mathrm{PSU}(2)\subset\mathrm{PSL}(2,\mathbb{C}).
\end{equation}
It follows that the gas \eqref{eq:gas:C} on $\mathbb{C}$ is invariant under
the Möbius transformations of the form 
\begin{equation}
    z\mapsto M_{a,b}(z)=\frac{az+b}{-\overline{b}z+\overline{a}},\quad |a|^2+|b|^2=1.
\end{equation}
In particular, it is invariant under $z\mapsto 1/z$ (take $a=0$ and $b=\mathrm{i}$), and under $z\mapsto -z$ (take $a=\mathrm{i}$ and $b=0$).

We have $\mathrm{O}(3)\setminus\mathrm{SO}(3)=\mathrm{SO}(3)S$ where
$S$ is an arbitrary fixed reflection. It is conjugated under $T$ to the
anti-Möbius transformation $z\mapsto M_{a,b}(\overline{z})$, $|a|^2+|b|^2=1$. The
group $\mathrm{O}(3)$ is conjugated under $T$ to the group generated by the
maps $z\mapsto M_{a,b}(z)$, $|a|^2+|b|^2=1$, and the reflection
$z\mapsto\overline{z}$. This group leaves invariant the gas
\eqref{eq:gas:C} on $\mathbb{C}$.

\subsection{Laplacian and Green function}
\label{sub:LaplaceGreen}
Recall that the classical 
Laplacian $\Delta$ on $\overline{\mathbb C}$ 
associates to any smooth function
$f$, defined on an open subset of  
$\overline{\mathbb C}$, the measure
$\Delta f = (\partial_x^2 f
+ \partial_y^2 f) 
\mathrm d x \mathrm d y$. Given a 
regular measure $\Lambda$ over
 $\overline{\mathbb C}$ of total mass one,
we say that a continuous symmetric function
$G:\overline{\mathbb C} \times \overline{\mathbb C}
\to (-\infty,\infty]$ 
is a (normalized) Green function 
if\footnote{The measure $\Lambda$ is necessary
since the integral of the Laplacian of a function on
a compact surface (and without boundary) is zero.}
\begin{equation}
\label{eq:GreenFunction}
\frac{1}{2\pi}\Delta G(p,\cdot) = -\delta_p + \Lambda
\end{equation}
for every $p \in \overline{\mathbb C}$.
This function is unique up to an additive constant
that is usually fixed by
requiring that $\int G(p,\cdot) \mathrm d \Lambda = 0$.
Since the measure involved in 
\eqref{eq:GreenFunction} is not regular,
the equation should be understood 
in the sense of distributions. Namely,
for every infinitely differentiable function
$\varphi:\overline{\mathbb C}\to\mathbb{R}$,
\begin{equation}
  \frac{1}{2\pi}\int_{\overline{\mathbb C}} G(p,q)\Delta\varphi (\mathrm d q)=
  -\varphi(p)+
  \int_{\overline{\mathbb C}}\varphi\mathrm d \Lambda.
\end{equation}
A concrete way to obtain $G$ is to consider
$V(z) = \int \log|z- w| \mathrm d \Lambda(w)$,
the logarithmic potential of $\Lambda$, which satisfies 
$\Delta V = 2\pi \Lambda$ on $\mathbb C$
and to set 
$G(z,w) =  -\log|z-w| + V(z) + V(w)$ which
verifies
\eqref{eq:GreenFunction}.
On the other hand, given $G$ obeying 
\eqref{eq:GreenFunction},
the function $W(z) = G(\infty,z)$ satisfies
$\Delta W = 2\pi(-\delta_{\infty} + \Lambda)$
so that
the Laplacian of $W$ on
$\mathbb C$ is exactly $2\pi \Lambda$.
Since $W$ is logarithmic at infinity,
$W$ is the logarithmic potential of $\Lambda$ up
to a constant\footnote{This can also be
seen by taking $w \to \infty$ in the identity
$G(z,w) - G(0,w) = \log|1-z w^{-1}| 
+ V(z) - V(0)$ which gives
$V(z) = G(z,\infty) - G(0,\infty) +V(0)$.}.
In particular,
\[- \log |z-w| = G(z,w) - G(z,\infty)
- G(w,\infty).\]
So, the study of the Green function helps
us to understand 
$\log|z-w|$. For 
the round sphere,
in other words choosing as $\Lambda$
the uniform measure $\nu$ on the sphere, 
we get the explicit formula
\[G(z,w) = - \log|z-w| + 
\frac{1}{2}\log (1 + |z|^2)
+\frac{1}{2}\log (1 + |w|^2) + \mathrm{constant}\]
because 
$\frac{1}{2}\log (1 + |z|^2)$ is logarithmic at
infinity and its Laplacian is $2\pi\mu$.
As explained in Subsection \ref{sub:Chordal},
using the identification
of $\overline{\mathbb C}$ with $\mathbb{S}^2$ and in the language most natural to 
$\mathbb{S}^2$ as a subset of $\mathbb R^3$,
the Green function is
\begin{equation}\label{eq:G}
  G(x,y)=-\frac{1}{2}-\log\Bigl(\frac{1}{2}{\|x-y\|}_{\mathbb{R}^3}\Bigr)=
  -\frac{1}{2}-\log d(p,q)
  = - \frac{1}{2} - g_p(q),
  \quad \text{with } p = T(x) \text{ and } q = T(y)
\end{equation}
where 
we have chosen the additive constant $-1/2$ in the definition \eqref{eq:G} so that
$\int_{\mathbb{S}^2}G(x,y)\mathrm{d}\nu(y)=0$ for all $x\in\mathbb{S}^2$.
By identifying smooth functions and regular measures
via the map $f \mapsto f \mathrm d \nu$,
the operator $-\Delta$ 
is an essentially self-adjoint operator
on $L^2_{\mathbb{C}}(\nu)$ with a 
a compact resolvent. Its
eigenvalues are given by $4\pi \ell(\ell+1)$, $\ell\in\{0,1,2,\ldots\}$. The
associated eigenspace $\mathcal{H}_\ell$ has dimension $2\ell+1$, and is
spanned by the spherical harmonics\footnote{They are explicit
  trigonometric polynomials with respect to latitude/longitude angle coordinates, hence the name.} ${(Y_{\ell,m})}_{-\ell\leq m\leq\ell}$, orthonormal in $L^2_{\mathbb{C}}(\nu)$. Note that their normalization
  differs from the classical ones
  because $\nu$ has unit mass rather than $4\pi$. The
orthogonal projector onto $\mathcal{H}_\ell$ is a kernel
operator with kernel given for all $x,y\in\mathbb{S}^2$ by
\begin{equation}\label{eq:addition-identity}
  c_\ell(x,y)
  =\sum_{m=-\ell}^{\ell}Y_{\ell,m}(x)\overline{Y_{\ell,m}(y)}
  =(2\ell+1)P_\ell(x\cdot y)
\end{equation}
where ${(P_\ell)}_{\ell\geq0}$ are the Legendre polynomials on $[-1,1]$,
orthogonal with respect to the uniform probability measure on $[-1,1]$, and
normalized with $P_\ell(1)=1$ for all $\ell\geq0$. We have, for all
$x,y\in\mathbb{S}^2$,
\begin{equation}\label{eq:green-series}
  G(x,y)
  =\frac{1}{2}\sum_{\ell=1}^\infty\frac{1}{\ell(\ell+1)}c_\ell(x,y).
\end{equation}
See \cite{muller1966spherical,atkinsonhan2012spherical}.
By the sphere and cylinder Archimedes theorem, for all $x\in\mathbb{S}^2$, 
\begin{equation}\label{eq:nunif}
    \text{the image of $\nu$ by
$y\in\mathbb{S}^2\mapsto x\cdot y\in\mathbb{R}$ is the uniform distribution on
$[-1,1]$}.
\end{equation}

\subsection{Planar, chordal, and geodesic distances}
\label{sub:Chordal}
The (half) chordal distance on 
the unit sphere
$\mathbb{S}^2 = \{x \in \mathbb{R}^3 : \|x\|_{\mathbb R^3} = 1\}$ is defined for all $x,y \in \mathbb{S}^2$ by
\begin{equation}\label{eq:dchord}
  d(x,y)
  =\frac{1}{2}{\|x-y\|}_{\mathbb{R}^3}\in[0,1].
\end{equation}
If $p$ and $q$ belong to
$\overline{\mathbb C}$, we define
$d(p,q)$ by $d(T^{-1}(p),T^{-1}(q))$. We
have $d(x,y)=\sin(\frac{\theta}{2})$ where $\theta\in[0,\pi]$ is the geodesic
distance (arc length) between the points $x$ and $y$ in
$\mathbb{S}^2$, and since $\cos(\theta)=x \cdot y$, we also have
\begin{equation}\label{eq:d2xy}
  d(x,y)^2
  =\frac{1-x\cdot y}{2}
  =\frac{1-\cos(\theta)}{2}.
\end{equation}
Using $T^{-1}(z)=(2\Re z,2\Im z,|z|^2-1)/(1+|z|^2)$
 in the first formula of
 \eqref{eq:d2xy} we get,
for all $z,w\in\mathbb{C}$,
\begin{equation}\label{eq:d}
  d(z,w) = \frac{|z-w|}{\sqrt{(1+|z|^2)(1+|w|^2)}}
  \quad\text{and}\quad
  d(z,\infty) = \frac{1}{\sqrt{1+|z|^2}}.
\end{equation}
Now, from \eqref{eq:d} we get that for all $w,z\in\mathbb{C}$, the planar
Euclidean distance $|w-z|$ on $\mathbb{C}$ admits the factorization
\begin{equation}\label{eq:dists}
|z-w|
=
\frac{d(z,w)}{d(z,\infty)d(w,\infty)}.
\end{equation}
See Figure \ref{fi:dists}.
This formula expresses the way the distances are deformed by $T$ when
the points approach the north pole. 
After taking the logarithm,
we get
\eqref{eq:G} up to an additive constant
and we obtain again that
\begin{align}\label{eq:split-euclid}
  f_z(w)
  =\log|w-z|
  &=\log d(w,z)-\log d(w,\infty)+\frac12\log(1+|z|^2)\\
  &=g_z(w)-g_\infty(w)+C_z
    \quad\text{where}\quad g_p=\log d(\cdot,p).
\end{align}
For all $p\in\overline{\mathbb{C}}$, the function
$g_p$ has only one singularity, at point $p$. For all $z\in\mathbb{C}$, the
function $f_z$ has two singularities, at points $z$ and $\infty$, and thanks
to \eqref{eq:split-euclid}, it is, up to the additive constant $C_z$, the
difference between the two functions $g_z$ and $g_\infty$ having one
singularity each. Moreover $g_z$ and $g_\infty$ are conjugated
by the isometry sending $z$ to $\infty$.
Finally, the functions $g_0=\log(\left|\cdot\right|/\sqrt{1+\left|\cdot\right|^2})$ and $g_\infty=\log(1/\sqrt{1+\left|\cdot\right|^2})$ are radial, and this is related to the fact that the points $0$ and $\infty$ are, respectively, the south and north poles of $\mathbb{S}^2$, on the vertical axis perpendicular to the complex plane. 

\subsection{Determinant and kernel}

The law \eqref{eq:gas:C} can be written as the probability measure
\begin{equation}
  \frac{1}{n!}\det{[K_n(z_i,z_j)]}_{1\leq i,j\leq n}\prod_{k=1}^n\mathrm{d}\mu(z_k)
  =
  \frac{1}{n!} 
  \Big|
  \det{[P_{n,i}(z_j)]}_{0\leq i,j\leq n-1}
  \Big|^2\prod_{k=1}^n\mathrm{d}\mu(z_k),
\end{equation}
where $K_n: \mathbb C \times \mathbb C \to \mathbb C$
and $P_{n,k}: \mathbb C \to \mathbb C$ are defined by
\begin{equation}\label{eq:Kn}
  K_n(z,w)=\sum_{k=0}^{n-1}
  P_{n,k}(z)\overline{P_{n,k}(w)},
  \quad
  P_{n,k}(z)=\sqrt{\frac{n!}{k!(n-1-k)!}}\frac{z^k}{(1+|z|^2)^{\frac{n-1}{2}}}.
\end{equation}
When $x$ and $y$ belong to $\mathbb S^2$,
we define $K_n(x,y)$
by $K_n(T(x),T(y))$.
Notice that Equation \eqref{eq:Kn} 
is telling us that $K_n$ is the kernel of 
the orthogonal projection
of $L^2_{\mathbb{C}}(\mu)$ onto $\mathcal P_n$,
the subspace of $L^2_{\mathbb{C}}(\mu)$ of polynomials of degree less than or equal to $n-1$ weighted with
	$(1+|z|^2)^{-(n-1)/2}$, namely,
\begin{equation}\label{eq:Pn}
    \mathcal P_n = 
    \bigg\{\frac{a_0 + a_1 z+ \dots + a_{n-1}z^{n-1}}{(1+|z|^2)^{(n-1)/2}}:
	a_0,\dots,a_{n-1} \in \mathbb C\bigg\},
\end{equation}
and $\{P_{n,k}: 0 \leq k \leq n-1 \}$ is the orthonormal
basis of $\mathcal P_n$ obtained by normalizing
the monomials. Performing the sum in \eqref{eq:Kn}, we get the short expression
\begin{equation}
  K_n(z,w) = n\frac{(1+z\overline w)^{n-1}}{(1+|z|^2)^{\frac{n-1}{2}}(1+|w|^2)^{\frac{n-1}{2}}}.
\end{equation}
Since $K_n$ represents an orthogonal projection
we also have, for all $z,u$, the kernel reproducing identity
\begin{equation}\label{eq:Kn:repro}
\int_{\mathbb{C}} K_n(z,w)K_n(w,u)\mathrm{d}\mu(w)=K_n(z,u).
\end{equation}
For all these aspects, we refer to \cite{MR2552864}.

\subsection{Kernel and chordal distance}
For all $z,w\in\mathbb{C}$, we have
\begin{equation}
  |z-w|^2=(1+|z|^2)(1+|w|^2)-|1+\overline{z}w|^2.
\end{equation}
This allows us to express $|K_n|$ in terms of the chordal distance $d$, namely, for
all $z,w\in\mathbb{C}$,
\begin{equation}\label{eq:kernel-distance}
  |K_n(z,w)| = n(1-d(z,w)^2)^{\frac{n-1}{2}}\leq n\mathrm{e}^{-\frac{n-1}{2}d(z,w)^2}.
\end{equation}
This off-diagonal decay of the kernel is crucial for the decoupling behind our CLTs.
If $x$ and $y$ belong to $\mathbb S^2$, we have
\begin{equation}
\label{eq:Knxy}
|K_n(x,y)|^2 =n^2\Big(\frac{1+x\cdot y}{2}\Big)^{n-1}.
\end{equation}

\subsection{Determinantal formula and Kostlan observation}

Apart from the spherical Laplace--Beltrami linear operator $\Delta$, which is
unbounded, all the remaining operators in this note are linear and bounded.
We denote by $K_n$ the operator defined by
$K_n(h)(w)=\int K_n(w,u)h(u)\mathrm{d}\mu(u)$
and by $j_n: \mathcal P_n \to L^2_{\mathbb C}(\mu)$
the inclusion operator so that
$K_n = j_n^{} j_n^*$. 
A key determinant formula for the random vector
 $Z_n=(Z_{n,1},\ldots,Z_{n,n})$ distributed according to \eqref{eq:gas:C} 
 is usually obtained by writing, 
 for a bounded measurable function 
 $\psi: \mathbb C \to \mathbb C$, that
 \begin{align}
 \mathbb E \Big( \prod_{i=1}^n \psi(Z_{n,i})\Big)
 &= 
 \frac{1}{n!}
 \int_{\mathbb C^n}
 \Big|\det{[P_{n,i}(z_j)]}_{0\leq i,j\leq n-1} \Big|^2 
 \prod_{i=0}^{n-1} \psi(z_i)     
 \prod_{i=0}^{n-1} \mathrm d \mu(z_i)     
 \\
 &=\frac{1}{n!}
 \int_{\mathbb C^n}
 \overline{\det{[P_{n,i}(z_j)]}_{0\leq i,j\leq n-1}}
 \det{[P_{n,i}(z_j) \psi(z_j)]}_{0\leq i,j\leq n-1}
 \prod_{i=0}^{n-1} \mathrm d \mu(z_i)            \\
 &=\det\Big[\langle P_{n,j},\psi P_{n,k}\rangle_{L^2_{\mathbb{C}}(\mu)}\Big]_{0\leq j,k\leq n-1}                   
                \label{eq:DetExplicit}     \\
 &
 = \det \big[j_n^* \psi j_n^{} \big] \label{eq:Det},
 \end{align}
 where $\psi$ inside the determinant in \eqref{eq:Det} 
 denotes
 the multiplication operator by
 $\psi$ on $L^2_{\mathbb C}(\mu)$.
The determinant can be interpreted as a Fredholm determinant for a trace-class operator.
This is the classical Costin--Lebowitz--Soshnikov approach for the
Laplace/Fourier transform of linear statistics of determinantal processes. See
\cite[Section 4.3.8]{MR2552864} and \cite{PhysRevLett.75.69,MR1894104}.

An observation which goes back to \cite{MR1148410}, see also \cite{MR2552864},
states that for all radial measurable $g:\mathbb{C}\to\mathbb{R}$, in the
sense that $g=\widetilde g(\left|\cdot\right|)$ for some measurable
$\widetilde g:\mathbb{R}_+\to\mathbb{R}$, we have, for all $t\in\mathbb{R}$,
\begin{equation}\label{eq:kostlan1}
  \mathbb{E}(\mathrm{e}^{\mathrm{i}tL_n(g)})
  =
  \prod_{k=0}^{n-1}\int_{\mathbb{C}}\mathrm{e}^{\mathrm{i}tg}|P_{n,k}|^2\mathrm{d}\mu.
\end{equation}
This can be obtained
from \eqref{eq:DetExplicit}
since
$\langle P_{n,j},\mathrm{e}^{\mathrm{i}tg} P_{n,k}\rangle_{0 \leq j,k \leq n-1}$
is a diagonal matrix in this case.
In probabilistic words, taking $g$ radial gives the following
equality in distribution for random multisets:
\begin{equation}\label{eq:kostlan2}
  \{|Z_{n,k}|:1\leq k\leq n\}\overset{\mathrm{d}}{=}\{\xi_{n,k}:1\leq k\leq
  n\}
\end{equation}
where $\xi_{n,1},\ldots,\xi_{n,n}$ are independent random variables
with
$\xi_{n,k}^2\sim\mathrm{BetaPrime}(k,n+1-k)=\frac{n!}{(k-1)!(n-k)!}\frac{x^{k-1}}{(1+x)^{n+1}}\mathbb{1}_{x\geq0}\mathrm{d}x$. In other words, $\xi_{n,k}^2$ has the law of $\Gamma_k/\Gamma_{n+1-k}$ where $\Gamma_k$ and $\Gamma_{n+1-k}$ are independent, where $\Gamma_a\sim\mathrm{Gamma}(a,1)$.

\subsection{Correlation densities and covariance of linear statistics}

Let us denote by $\rho_n$ the density \eqref{eq:gas:C}. The Fredholm determinantal formula \eqref{eq:Det} implies that for all $1\leq k\leq n$, the marginal density
\begin{equation}
   (z_1,\ldots,z_k)\in\mathbb{C}^k
   \mapsto\rho_{n,k}(z_1,\ldots,z_k):=\int_{\mathbb{C}^{n-k}}\rho_n(z_1,\ldots,z_n)\mathrm{d}^2z_{k+1}\cdots\mathrm{d}^2z_n
 \end{equation}
 or $k$-points correlation density\footnote{It is customary to use also the $k$-points correlation
$\rho_n^{(k)}(z_1,\ldots,z_k):=\det\bigl[K_n(z_i,z_j)\bigr]_{1\leq i,j\leq
  k}\mu(z_1)\cdots\mu(z_k)=\frac{n!}{(n-k)!}\rho_{n,k}(z_1,\ldots,z_k)$.} can be expressed using the kernel \eqref{eq:Kn} as
\begin{equation}
  \rho_{n,k}(z_1,\ldots,z_k)
  \mathrm{d}^2z_1\cdots\mathrm{d}^2z_k
  =\frac{(n-k)!}{n!}\det\bigl[K_n(z_i,z_j)\bigr]_{1\leq i,j\leq k}\prod_{i=1}^k\mathrm{d}\mu(z_i).
\end{equation}
In particular $\rho_{n,n}=\rho_n$. Since $K_n(z,z)=n$, we get, denoting $\mu(z)$ the density of $\mu$ at $z$,
\begin{equation}\label{eq:rho}
\frac{\rho_{n,1}(z)}{\mu(z)}=\frac{K_n(z,z)}{n}=1
\quad\text{and}\quad
\frac{\rho_{n,2}(z,w)}{\mu(z)\mu(w)}=\frac{K_n(z,z)K_n(w,w)-|K_n(z,w)|^2}{n(n-1)}
=\frac{n^2-|K_n(z,w)|^2}{n(n-1)}.
\end{equation}
This provides formulas for the mean and variance of linear statistics, namely 
\begin{align}
  \mathbb{E}(L_n(f))
  &=\int_{\mathbb{C}}f(z)K_n(z,z)\mathrm{d}\mu(z)=n\int_{\mathbb{C}}f(z)\mathrm{d}\mu(z)\label{eq:ELn}\\
   \mathrm{Cov}(L_n(f),L_n(g))
  &=\iint_{\mathbb{C}^2} (nf(z)g(z)-f(z)g(w)|K_n(z,w)|^2)\mathrm{d}\mu(z)\mathrm{d}\mu(w)\label{eq:CovLnLn2}
\end{align}
for all $f\in L^1(\mu)$ and for all $f,g\in L^2(\mu)$ respectively.
Using \eqref{eq:Kn:repro} we also have
\begin{equation}
   \mathrm{Cov}(L_n(f),L_n(g))
  =\frac{1}{2}\iint_{\mathbb{C}^2} (f(z)-f(w))(g(z)-g(w))|K_n(z,w)|^2\mathrm{d}\mu(z)\mathrm{d}\mu(w).\label{eq:CovLnLn}\\
\end{equation}

\subsection{Law of large numbers and GFF CLT}

Following \cite{MR2552864}, see also \cite{MR2926763}, almost surely,
\begin{equation}\label{eq:macro}
  \frac{L_n(f)-n\int f\mathrm{d}\mu}{n}
  =\frac{1}{n}\int f\mathrm{d}L_n-\int f\mathrm{d}\mu\xrightarrow[n\to\infty]{}0
\end{equation}
for all continuous and bounded $f:\mathbb{C}\to\mathbb{R}$. In order to
separate the fluctuations of $L_n$ around its average $\mathbb{E}L_n$
and the bias of the mean with respect to the
macroscopic limit $\mu$, it is natural to rewrite \eqref{eq:macro} as
\begin{equation}\label{eq:macro2}
  \underbrace{\frac{L_n-\mathbb{E}L_n}{n}}_{\mathrm{fluct.}}+\underbrace{\frac{1}{n}\mathbb{E}L_n-\mu}_{\mathrm{bias}}
  \xrightarrow[n\to\infty]{}0.
\end{equation}
The spherical ensemble has zero bias : the logarithmic potential of $\mathbb{E}\eta_n$ is $\frac{n}{2}\log(1+|z|^2)$, hence 
\begin{equation}\label{eq:bias}
    \mathbb{E}L_n=n\mu\quad\text{for all $n\geq1$}.
\end{equation}
This is also visible in \eqref{eq:rho} since the density of $\mathbb{E}\eta_n$ with respect to $\mu$ is $z\mapsto K_n(z,z)=n$.
Regarding the fluctuation, and following \cite{MR2346510}, the right functional space is the Sobolev space
\begin{equation}
 H^1(\overline{\mathbb{C}})
 :=\{f:\overline{\mathbb{C}}\to\mathbb{R}:f\circ T\in H^1(\mathbb{S}^2)\}
 =\Bigl\{f\in L^2(\mu):\int_{\mathbb{C}}|\nabla f(z)|^2\mathrm{d}^2z<\infty\Bigr\}.
\end{equation}
This should not be confused with $H^1(\mu)$. The last formulation comes from the fact that $\mu=\nu\circ T^{-1}$ and
\begin{equation}
\int_{\mathbb{S}^2}|\nabla_{\mathbb{S}^2}(f\circ T)|_{\mathbb{S}^2}^2\mathrm{d}\nu=\int_{\mathbb{C}}\frac{(1+|z|^2)^2}{4}|\nabla f(z)|^2\mathrm{d}\mu(z)=\frac{1}{4\pi}\int_{\mathbb{C}}|\nabla f(z)|^2\mathrm{d}^2z.
\end{equation}
Now still following \cite{MR2346510}, for all $f,g\in H^1(\overline{\mathbb{C}})$,
\begin{equation}\label{eq:GFFCLT:cov}
  \mathrm{Cov}(L_n(f),L_n(g))
  \xrightarrow[n\to\infty]{}
  \tfrac{1}{4\pi}\langle
  f,g\rangle,
\end{equation}
where
\begin{equation}
  \langle f,g\rangle
  =\int_{\mathbb{C}} \nabla f(z)\cdot\nabla g(z) \mathrm{d}^2z
  =\langle -\Delta g,f\rangle_{H^{-1},H^1}
  =\langle -\Delta f,g\rangle_{H^{-1},H^1}.
\end{equation}
This defines an inner product on $H^1(\overline{\mathbb{C}})/\mathbb{R}$.
Moreover, for all $f\in H^1(\overline{\mathbb{C}})\subset L^2(\mu)\subset L^1(\mu)$, we have the CLT
\begin{equation}\label{eq:GFFCLT}
 L_n(f)-\mathbb{E}L_n(f)
 =  \int f\mathrm{d}(L_n-\mathbb{E}L_n)
 \xrightarrow[n\to\infty]{\mathrm{d}}\mathcal{N}\bigl(0,\tfrac{1}{4\pi}\langle f,f\rangle\bigr).
\end{equation}
In other words, the fluctuations in the law of large numbers \eqref{eq:macro2}
are of order $1$, instead of the usual $\sqrt{n}$ for iid random particles. This is an important rigidity effect of the singular repulsion in the Coulomb gas \eqref{eq:gas:C}. 

The random centered measure $L_n-\mathbb{E}L_n$ converges as
$n\to\infty$, weakly with respect to $H^1(\overline{\mathbb{C}})/\mathbb{R}$ test functions, to a Gaussian
field with covariance operator $-\frac{1}{4\pi}\Delta$. It is the image by $\frac{-\Delta}{\sqrt{4\pi}}$ of the zero mean Gaussian Free Field (GFF), which has covariance operator
$(-\Delta)^{-1}$, with singular kernel $-\frac{1}{2\pi}\log d(\cdot,\cdot)$. The two-dimensional GFF is log-correlated. With the large deviations principle or high-dimensional analysis of
Boltzmann--Gibbs measure with mean-field interactions in mind, both
\eqref{eq:macro} and \eqref{eq:GFFCLT} can be guessed from the formal
rewriting of the Coulomb gas \eqref{eq:gas:C} as
\begin{equation}
  \exp\Bigl(-(n+1)^2\Bigl[\int Q\frac{\mathrm{d}\ell_n}{n+1}
  +\frac{1}{2}\iint_{\neq}\Bigl(-\frac{\Delta}{4\pi}\Bigr)^{-1}\frac{\mathrm{d}\ell_n^{\otimes 2}}{(n+1)^2}\Bigr]\Bigr)
  \quad\text{where}\quad
  \ell_n=\sum_{k=1}^n\delta_{z_k}.
\end{equation}

\section{Logarithmic singularities : proof of Theorem \ref{th:LogSing}}
\label{se:proof:th:LogSing}

This proof relies on the CLT \eqref{eq:logpot0:CLT} for $L_n(f_0)=\log|\det(AB^{-1})|$, the Lipschitz-variance Lemma \ref{lem:lip-variance}, and the decoupling Lemmas \ref{lem:Covariance} and \ref{lem:Independence}. It does not rely on the GFF CLT \eqref{eq:GFFCLT} or on the Kostlan observation \eqref{eq:kostlan2}.

If $f$ and $g$ have finite singularities at the same points and with the same weights, with sum of squares $\sigma^2$, then $f-g$ smoothly extend to $\overline{\mathbb C}$ so that, by  Lemma \ref{lem:lip-variance}\footnote{Since $f-g\in H^1$, we could use alternatively the GFF CLT \eqref{eq:GFFCLT} but this is not needed.}, $(L_n(f-g)-\mathbb{E}(L_n(f-g)))/\sqrt{\log(n)}\to0$ in $L^2$ as $n\to\infty$. Therefore
\begin{equation}\label{eq:LogSingEquiv}
    \frac{L_n(f)-\mathbb{E}(L_n(f))}{\sqrt{\frac{1}{4}\log(n)}}
    \xrightarrow[n\to\infty]{\mathrm{d}}
    \mathcal{N}(0,\sigma^2)
    \quad\text{if and only if}\quad
    \frac{L_n(g)-\mathbb{E}(L_n(g))}{\sqrt{\frac{1}{4}\log(n)}}
    \xrightarrow[n\to\infty]{\mathrm{d}}
    \mathcal{N}(0,\sigma^2).
\end{equation}

\subsection{Case of a single singularity at the origin}
\label{sub:SingleSingularity}

Suppose that $f:\overline{\mathbb{C}}\setminus\{0\}\to\mathbb{R}$ has a single singularity at $0$, with weight $c$. By scaling, we can assume that $c=1$. According to \eqref{eq:LogSingEquiv}, it suffices to consider an example of such a function.

A proof based on the Kostlan observation is given in Section \ref{ss:meanvarcum}. We present now an alternative proof, which is more original. The idea is to deduce the result from the two-singularities case $f_0=\log\left|\cdot\right|$, by splitting via localization. Choose a smooth function $\chi: [0,\infty) \to [0,1]$ satisfying $\chi(x) = 0$ for $x \geq 1/2$ and $\chi(x) = 1$ for $x \leq 1/4$.
Define the functions $\ell_0(x) = \log|x| \chi(|x|)$, which has a single logarithmic singularity, at $0$, and $\ell_\infty(x) = -\log|x| \chi(1/|x|)$ which has a single logarithmic singularity, at $\infty$. These functions have disjoint supports. We have 
\begin{equation}\label{eq:f000infty}
  L_n(\ell_0)
  \overset{\mathrm{d}}{=} 
  L_n(\ell_\infty).
\end{equation}
This comes from the invariance of \eqref{eq:gas:C} by the map $z\mapsto 1/z$.
Now 
\begin{equation}\label{eq:f0split}
    \ell_0 - \ell_\infty =
    \log|z| + h = f_0 + h
    \quad\text{for a smooth $h:\overline{\mathbb C} \to \mathbb R$}.
\end{equation}
From \eqref{eq:f000infty}, we get that $L_n(\ell_0)-L_n(\ell_\infty)$ is centered.
Since $L_n(f_0)$ is also centered, it follows that $L_n(h)$ is centered.
Since $h$ is Lipschitz, Lemma \ref{lem:lip-variance} implies that $L_n(h)/\sqrt{\log(n)}\to0$ in $L^2$ as $n\to\infty$, and therefore, by \eqref{eq:logpot0:CLT}, we get
\begin{equation}
\frac{L_n(\ell_0)}{\sqrt{\log(n)}}
-\frac{L_n(\ell_\infty)}{\sqrt{\log(n)}}
\xrightarrow[n\to\infty]{\mathrm{d}}
\mathcal{N}(0,\tfrac{1}{2})
\quad\text{and}\quad
\mathrm{Var}\left(\frac{L_n(\ell_0)}{\sqrt{\log(n)}}-\frac{L_n(\ell_\infty)}{\sqrt{\log(n)}}\right)
\xrightarrow[n\to\infty]{}\frac{1}{2}.
\end{equation}
By \eqref{eq:f000infty}, we can write
\begin{equation}
2\mathrm{Var}(L_n(\ell_0))
=\mathrm{Var}(L_n(\ell_0 - \ell_\infty))
+ 2 \mathrm{Cov}(L_n(\ell_0),L_n(\ell_\infty))
\end{equation}
and we can use Lemma \ref{lem:Covariance} 
to get that
$\mathrm{Cov}(L_n(\ell_0), L_n(\ell_\infty))$
converges to zero so that
the variance of
$L_n(\ell_0)/\sqrt{\log(n)}$ converges
to $1/4$. In particular,
the sequences of random vectors
$(X_n,Y_n)_{n \geq 2}$ defined by
\begin{equation} 
X_n = \frac{L_n(\ell_0)-
\mathbb E(L_n(\ell_0))}{\sqrt{\log(n)}}
\quad\text{and}\quad
Y_n=\frac{L_n(\ell_\infty)-
\mathbb E(L_n(\ell_\infty))}{\sqrt{\log(n)}}
\end{equation}
is tight. Take any subsequence such that
$(X_n,Y_n)_{n \geq 2}$ converges in law
to some random vector $(X,Y)$.
Let us define
\begin{equation}
\gamma(\alpha) = \mathbb E\big(\mathrm{e}^{\mathrm{i}\alpha X} \big)
= \lim_{n \to \infty} 
\mathbb E \big(\mathrm{e}^{\mathrm{i} \alpha X_n }\big)
\end{equation}
which is non-zero for $\alpha$ small enough.
So by Lemma \ref{lem:Independence} for
\begin{equation}
\Phi_n = \exp\Bigl(\mathrm{i}\alpha\frac{\ell_0}{\sqrt{\log(n)}}\Bigr)
\quad\text{and}\quad
\Psi_n = \exp\Bigl(\mathrm{i}\beta\frac{\ell_\infty}{\sqrt{\log(n)}}\Bigr)
\end{equation}
and by using
$\mathbb E(\Phi_n)
\mathrm{e}^{-\mathrm{i}\alpha n (\log n)^{-1/2}\int \ell_0 \mathrm{d} \mu}
= \mathbb E(\mathrm{e}^{\mathrm{i}\alpha X_n})$
and $\mathbb E(\Psi_n)
\mathrm{e}^{-\mathrm{i}\beta n (\log n)^{-1/2}\int \ell_\infty \mathrm{d} \mu}
= \mathbb E(\mathrm{e}^{\mathrm{i}\beta Y_n})$
we get that 
$\mathbb E(\mathrm{e}^{\mathrm{i}\alpha X_n + \mathrm{i}\beta Y_n})$
is asymptotically equivalent to
$\mathbb E(\mathrm{e}^{\mathrm{i}\alpha X_n}) \mathbb E(\mathrm{e}^{\mathrm{i}\beta Y_n})$
as long as 
$|\mathbb E(\mathrm{e}^{\mathrm{i}\alpha X_n}) \mathbb E(\mathrm{e}^{\mathrm{i}\beta Y_n})|$
is bounded from below.
The latter occurs at least for $\alpha$ and $\beta$ small since 
its limit is $|\gamma(\alpha) \gamma(\beta)|$.
This implies that
$\mathbb E(\mathrm{e}^{\mathrm{i}(\alpha X + \beta Y)})
=\gamma(\alpha) \gamma(\beta)$ for every $\alpha, \beta$ 
small enough and, in particular,
$\mathbb E(\mathrm{e}^{\mathrm{i}(\alpha (X - Y))})
=|\gamma(\alpha)|^2$ for $\alpha$ small.
Since 
$X-Y \sim \mathcal N(0,1/2)$,
we also know that 
$\mathbb E(\mathrm{e}^{\mathrm{i}\alpha(X-Y)})
= \mathrm{e}^{-\alpha^2/4}$
which implies that
$|\gamma(\alpha)|^2 = e^{-\alpha^2/4}$ at
least for $\alpha$ small enough.
Consider $\widetilde X$ an independent copy
of $X$ so that
$\mathbb E(e^{i(\alpha(X-\widetilde X))}) = 
|\gamma(\alpha)|^2$
for every $\alpha \in \mathbb R$.
The moments of  $X - \widetilde X$
can be obtained by differentiating
$|\gamma(\alpha)|^2$ at the origin so
that they coincide with the ones of
$X-Y$ implying that both $X-\widetilde X$ and $X - Y$ have the same law and
\[
|\gamma(\alpha)|^2 = \mathrm{e}^{-\alpha^2 /4} \quad \mbox{ for every } \alpha 
\in \mathbb R.
\]
But this entails that $\gamma(\alpha)$ is non-zero for every $\alpha$ so that
$X$ and $Y$ are independent
by Lemma \ref{lem:Independence}. Because of this independence and
since 
$X - Y$ is Gaussian, we may use Cramér's decomposition theorem
to conclude that $X$ and
$Y$ are Gaussian and by using that $X$ and $Y$ have the same law
we get
\begin{equation}
    X \sim \mathcal N \bigg(0,\frac{1}{4}\bigg)
\end{equation}
which is what we wanted : the CLT \eqref{eq:LogSing} for $k=1$ and $p_1=0$.

\subsection{Case of a single singularity at a point} 

Suppose that $f$ has a single singularity located at $p \in \overline{\mathbb C}$ having weight $c$. By linearity we can assume that $c=1$. By the $\mathrm{SO}(3)$ invariance of \eqref{eq:gas:S} and $\nu$, we can assume that $p=0$, so $h=f-\ell_0$ is smooth on $\overline{\mathbb C}$.
Lemma \ref{lem:lip-variance} gives $(L_n(h)-n\int h\mathrm{d}\mu)/\sqrt{\log n}\to0$ in $L^2$ as $n\to\infty$, so the CLT for $\ell_0$ gives
\begin{equation}
    \frac{L_n(f)-n\int f\mathrm{d}\mu}{\sqrt{\log(n)}}
    \xrightarrow[n\to\infty]{\mathrm{d}}
    \mathcal{N}(0,\tfrac{1}{4}).
\end{equation}

\subsection{General case} 

Consider $f$ with logarithmic singularities at $p_1,\ldots,p_k$ with weights $c_1,\ldots,c_k$. Take $k$ smooth functions $\chi_1,\ldots,\chi_k$ on $\overline{\mathbb{C}}$ with pairwise disjoint supports and such that $\chi_j\equiv 1$ in a neighborhood of $p_j$. Then,
\begin{equation}
f = c_1 f_1 + \dots + c_k f_k + h \quad \text{with}\quad 
f_j =c_j^{-1} f \chi_j,
\end{equation}
and $h:\overline{\mathbb{C}}\to\mathbb{R}$ is smooth. With this decomposition, $f_j$ has a single singularity at point $p_j$, with weight $1$. Now, we know that $(L_n(h)-n \int h \mathrm d \mu)/\sqrt{\log(n)}\to0$ in $L^2$ as $n\to\infty$ by Lemma \ref{lem:lip-variance}, and that
\begin{equation}
\frac{L_n(f_j)-n\int f_j \mathrm{d}\mu}{\sqrt{\log(n)}}
\xrightarrow[n\to\infty]{\mathrm{d}}\mathcal{N}(0,\tfrac{1}{4}).
\end{equation}
We may apply iteratively Lemma \ref{lem:Independence} to %
\begin{equation}
\label{eq:psiphiInduction}
\psi_n = \mathrm{e}^{\mathrm{i}(\alpha_1 f_1+\dots+\alpha_{j-1}f_{j-1})/\sqrt{\log(n)}}
\quad \mbox{ and } 
\quad \varphi_n =\mathrm{e}^{\mathrm{i}\alpha_{j} f_{j}/\sqrt{\log(n)}}
\end{equation}
to get that the random variables
$(L_n(f_i)-n\int f_i\mathrm{d}\mu)/\sqrt{\log(n)}$ jointly converge to independent centered Gaussians of variance $1/4$. 
Indeed, 
if we already know that
$(L_n(f_i)-n\int f_i\mathrm{d}\mu)/\sqrt{\log(n)}$ jointly converge to 
independent centered Gaussians of variance $1/4$
for $i \leq j-1$ we can take
$\psi_n$ and $\varphi_n$ from
\eqref{eq:psiphiInduction}
which satisfy the hypothesis
of Lemma 
\ref{lem:Independence}
since $|\mathbb E[\Psi_n]|$ and $|\mathbb E[\Phi_n]|$
have non-zero limits. 
We get
$\mathbb E[\Psi_n \Phi_n] \sim_{n \to \infty}
\mathbb E[\Psi_n] \mathbb E[\Phi_n]$
and, because this holds for every $\alpha_1,\dots,\alpha_j
\in \mathbb R$,
we get the desired 
asymptotic independence, namely, we get that
$(L_n(f_i)-n\int f_i\mathrm{d}\mu)/\sqrt{\log(n)}$ jointly converge to 
independent centered Gaussians of variance $1/4$
for $i \leq j$.

\subsection{Lemmas for localization and decoupling}
\label{se:lemmas}

\begin{lemma}[Uniform variance bound for Lipschitz functions]\label{lem:lip-variance}
  For all $g:\mathbb{C}\to\mathbb{R}$ and $n\geq1$,
  \begin{equation}
    \mathrm{Var}(L_n(g))\leq \frac{1}{2}{\|g\|}_{\mathrm{Lip.}}^2
    \quad\text{where}\quad
    {\|g\|}_{\mathrm{Lip.}}=\sup_{w\neq u}\frac{|g(w)-g(u)|}{d(w,u)}.
  \end{equation}
  In particular, if ${\|g\|}_{\mathrm{Lip.}}<\infty$ then $L_n(g)-\mathbb{E}(L_n(g))=O_{L^2,n\to\infty}(1)$.
\end{lemma}

\begin{proof}
  By the formula \eqref{eq:CovLnLn} for $\mathrm{Var}(L_n(g))=\mathrm{Cov}(L_n(g),L_n(g))$ and
  the Lipschitz bound, it remains to control 
  \begin{equation}
    \int\Bigl(\int d(w,u)^2|K_n(w,u)|^2\mathrm{d}\mu(w)\Bigr)\mathrm{d}\mu(u).
  \end{equation}
  By $\mathrm{SO}(3)$ invariance of $\mu$ and $K_n$, the integral above equals the
  integral with $u$ fixed. 
  Now, by combining \eqref{eq:nunif} and \eqref{eq:d2xy}, it follows that the image
  of $\mu$ by the map $w\mapsto d(w,u)^2$ is uniform on $[0,1]$.
Combining with the expression \eqref{eq:kernel-distance} of the kernel $K_n$ in terms of the chordal distance $d$, we get
\begin{equation}
    \int d(w,u)^2|K_n(w,u)|^2\mathrm{d}\mu(w)
    =n^2\int d(w,u)^2(1-d(w,u)^2)^{n-1}\mathrm{d}\mu(w)
    =n^2\int_0^1x(1-x)^{n-1}\mathrm{d}x
    =\frac{n}{n+1}\leq 1.
  \end{equation}
\end{proof}

\begin{lemma}[$L^2$ Decoupling] \label{lem:Covariance}
 If $f,g \in L^2_{\mathbb C}(\mu)$ have disjoint supports and at least one of them is compact in $\mathbb C$ or, equivalently,
 their supports in $\overline{\mathbb C}$ 
 are disjoint, then 
 \begin{equation}
  \mathrm{Cov}(L_n(f),L_n(g)) \xrightarrow[n\to\infty]{} 0.
  \end{equation}
\end{lemma}

\begin{proof}
	Denote by $\mathcal{A}$ the support of $f$ and by 
	$\mathcal{B}$ the support of $g$. Since $\mathcal{A}\cap\mathcal{B}=\varnothing$, we get,
    from \eqref{eq:CovLnLn},
	\begin{align*}
	\mathrm{Cov}(L_n(f),L_n(g))
	&= -
	\int_{\mathcal{A}\times\mathcal{B}}
	f(x) g(y) |K_n(x,y)|^2 \mathrm d \mu(x) \mathrm d \mu(y).							
	\end{align*}
	The assumptions on $\mathcal{A}$ and $\mathcal{B}$ give $\rho=d(\mathcal{A},\mathcal{B})=\inf_{(u,v)\in\mathcal{A}\times\mathcal{B}}d(u,v)\in(0,1)$, and thus, by \eqref{eq:kernel-distance},  
	\begin{equation}
	|\mathrm{Cov}(L_n(f),L_n(g))|
	\leq n^2(1-\rho^2)^{n-1}\|f\|_1\|g\|_1 
    \xrightarrow[n \to \infty]{}0. 
    \end{equation}
\end{proof}

\begin{lemma}[Decoupling] \label{lem:Independence}
	Let $\psi_n, \varphi_n : \overline{\mathbb C} \to \mathbb C$ be
	bounded measurable functions such that $|\psi_n|,|\varphi_n|\leq1$. Suppose there are
	two disjoint closed sets
	$\mathcal{A}, \mathcal{B}\subset \overline{\mathbb C}$ such that
	\begin{equation}
	    \psi_n(x) = 1 \mbox{ for } x \notin \mathcal{A} 
        \quad\text{and}\quad 
    	\varphi_n(x) = 1 \mbox{ for } x \notin  \mathcal{B}.
    \end{equation}
	Define
	$\Psi_n = \prod_{i=1}^n \psi_n(Z_{n,i})$ and
	$\Phi_n = \prod_{i=1}^n \varphi_n(Z_{n,i})$ where $Z_n=(Z_{n,1},\ldots,Z_{n,n})$ is distributed according to \eqref{eq:gas:C}.
	Then, if along a subsequence
    $|\mathbb E[\Psi_{n}]
	\mathbb E[\Phi_{n}]| \geq \varepsilon$
    for some $\varepsilon > 0$, along the same
    subsequence
	\begin{equation}
		\label{eq:indep}
		 \mathbb E[\Psi_{n} \Phi_{n}] 
         \sim_{n \to \infty} \mathbb E[\Psi_n]
	       \mathbb E[\Phi_n].
	\end{equation}
\end{lemma}

\begin{proof}
	
	Recall that $\mathcal P_n$ is the set
	of polynomials of degree less or equal than $n-1$ weighted
	with $(1+|z|^2)^{-(n-1)/2}$, as in \eqref{eq:Pn}.
	Recall also that
	$j_n^{}: \mathcal P_n \to L^2_{\mathbb C}(\mu)$ 
    denotes the (isometric) inclusion
    and that $j_n^*: L^2_{\mathbb C}(\mu) \to \mathcal P_n$ 
    is the orthogonal projection from $L^2_{\mathbb C}(\mu)$ onto the space of polynomials $\mathcal{P}_n$. The 
    determinant identity \eqref{eq:Det} gives
	\[\mathbb E[\Psi_n] = \det(j_n^* \psi_n j_n^{}), \quad 
	\mathbb E[\Phi_n] = \det(j_n^* \varphi_n j_n^{}),
	\quad \mbox{ and } 
	\quad 
	\mathbb E[\Psi_n \Phi_n] = \det(j_n^* \psi_n \varphi_n j_n^{}),\]
	where a function $g:\overline{\mathbb C} \to \mathbb R$ inside the determinant represent
	the multiplication operator by $g$.
	We omit the indices $n$ to simplify notation.
	Our goal is to study the expression
	\[ \frac{\mathbb E[\Psi \Phi]}{\mathbb E[\Psi] \mathbb E[\Phi]}
	=  \frac{\det(j^* \psi \varphi j)}
	{\det(j^* \psi j)
	\det(j^* \varphi j)} 
	= \det	\big( (j^* \psi \varphi j)(j^* \psi j)^{-1}(j^* \varphi j)^{-1}\big).
	\]
	Since 
	$|\det(1+A_n) - 1| \leq \|A_n\|_1 \mathrm{e}^{\|A_n\|_1}$ for any 
	matrix $A_n$ of size $n$, it is enough to control the operator
	\[1 - (j^* \psi \varphi j)(j^* \psi j)^{-1}(j^* \varphi j)^{-1}
	=\Big[
	(j^* \varphi j)
	(j^* \psi j)
	-
	(j^* \psi\varphi j)
	\Big]
	(j^* \psi j)^{-1}(j^* \varphi j)^{-1}.
	\]
	We will bound the operator norm of 
	$(j^* \psi j)^{-1}$ and of $(j^* \varphi j)^{-1}$
	and show that
	$\|(j^* \varphi j)
	(j^* \psi j)
	-
	(j^* \psi\varphi j)\|_1 \to 0$ which will allow us to conclude.
	Since $\psi = 1$ outside $\mathcal{A}$ and $\varphi = 1$ outside $\mathcal B$
	we get that
	\[\psi\varphi = \varphi + \psi - 1\]
	which allows us to write
	\[j^*  \psi\varphi j = j^* \varphi j + 
	j^* \psi j - j^* j.\]
	Using that $j^* j j^* = j^*$ and that $j j^* j = j$, we obtain
	\[(j^* \varphi j)(j^* \psi j) - 
	(j^* \psi\varphi j) 
	= \big( j^* (\varphi  - 1) j\big)
	\big( j^* (\psi  - 1) j\big).\]
	Now, it is enough to show that
	\begin{itemize}
		\item $\lim_{n \to \infty} \big\|\big( j^* (\varphi  - 1) j\big)
		\big( j^* (\psi  - 1) j\big) \big\|_1  = 0  $
		\item $\|(j^* \varphi j)^{-1}\|_{\mathrm{op.}} $ is bounded in $n$ and
		\item $\|(j^* \psi j)^{-1}\|_{\mathrm{op.}} $ is bounded in $n$.
	\end{itemize}
    We denote by $\|A\|_p$ the Schatten $p$-norm of $A$, the $\ell^p$ norm of its singular values. In particular $\|A\|_\infty=\|A\|_{\mathrm{op.}}$. 
    Now
	\begin{align}
    \big\|\big( j^* (\varphi  - 1) j\big)
	\big( j^* (\psi  - 1) j\big) \big\|_1
	&\leq \|j^*\|_4
	\big\|(\varphi  - 1) j
	 j^* (\psi  - 1)  \big\|_2 \|j\|_4 \\
	 &= \sqrt n
	 \Bigl(\int_{\mathbb C \times \mathbb C}
	 |\varphi(x) - 1|^2 |K(x,y)|^2 |\psi(y) - 1|^2 \mathrm d \mu(x) \mathrm d \mu(y)\Bigr)^{1/2}.
	  \end{align}
	Since $\mathcal{A}$ and $\mathcal{B}$ are disjoint compact
	closed subsets of the compact space $\overline{\mathbb C}$,
    Equation
	\eqref{eq:kernel-distance} tells us that there is some constant $C>0$
	such that $|K(x,y)| \leq C^{-1} n \mathrm{e}^{- C n}$ for every $x \in \mathcal{A}$ and $y \in \mathcal{B}$.
	We obtain
	\begin{align*}
		\sqrt n
	\Bigl(\int_{\mathbb C \times \mathbb C}
	|\varphi(x) - 1|^2 |K(x,y)|^2 |\psi(y) - 1|^2 \mathrm d \mu(x) \mathrm d \mu(y)\Bigr)^{1/2}
	&=
	\sqrt n
	\Bigl(\int_{\mathcal{B} \times \mathcal{A}}
	|\varphi(x) - 1|^2 |K(x,y)|^2 |\psi(y) - 1|^2 \mathrm d \mu(x) \mathrm d \mu(y)\Bigr)^{1/2}				\\
	& \leq  4C^{-1}n^{\frac{3}{2}} \mathrm{e}^{-Cn} 	 \xrightarrow[n \to \infty]{} 0 .
	\end{align*}
	Now we bound $\|(j^* \varphi j)^{-1}\|_{\mathrm{op.}} $ which is
	the inverse of the least singular value of $j^*\varphi j$.
	We have $\|j^* \varphi j\|_{\mathrm{op.}}
	\leq \|\varphi \|_{\mathrm{op.}} \leq 1$ by hypothesis 
	so that every singular value of $j^* \varphi j$ is  bounded by $1$.
	If we add the fact that the absolute value of the determinant
	is the product of singular values,
	we get that the least singular value is bounded from below by the absolute value
	of the determinant.
	It is enough to show that $|\det(j^* \varphi j)|$ is bounded from below
	but this is part of the hypotheses because
	$|\det(j^* \varphi j)| = |\mathbb E[\Phi]| 
    \geq 
    |\mathbb E[\Phi]\mathbb E[\Psi]| \geq \varepsilon $. A similar reasoning works
	to bound $\|(j^* \psi j)^{-1}\|_{\mathrm{op.}} $. %
  \end{proof}

\section{Green function : proof of Corollary \ref{co:Green:CLT} and Theorem \ref{th:Green:Cov}}
\label{se:proof:Green}

\subsection{Mean, variance, and one-point CLT via cumulants}
\label{ss:meanvarcum}

Following Lemma \ref{le:Green:cum}, for all fixed $p\in\overline{\mathbb{C}}$, since $\kappa_m(L_n(g_p))$ converges to a finite value as $n\to\infty$ for all $m\geq3$, while $\kappa_2(L_n(g_p))\to+\infty$, the Fréchet--Shohat theorem gives 
\begin{equation}
    \frac{L_n(g_p)-\mathbb{E}(L_n(g_p))}{\sqrt{\mathrm{Var}(L_n(g_p))}}\xrightarrow[n\to\infty]{\mathrm{d}}\mathcal{N}(0,1).
\end{equation}
Since moreover $\mathbb{E}(L_n(g_p))=-\frac{n}{2}$ and $\mathrm{Var}(L_n(g_p))\sim\frac{1}{4}\log(n)$, we get an alternative cumulants based proof of \eqref{eq:Green:CLT}.

\begin{lemma}[Cumulants of $L_n(g_p)$]\label{le:Green:cum}
  For all $p\in\overline{\mathbb{C}}$ and all $n\geq1$ and $m\geq1$,
  \begin{equation}\label{eq:Lngp:cum}
    \kappa_m(L_n(g_p)) = \frac{(-1)^m (m-1)!}{2^m}H_{n}^{(m-1)}    
  \end{equation}
  where $H^{(s)}_n=\sum_{k=1}^n\frac{1}{k^s}$ is the generalized harmonic number. In particular, with $H_n=H_n^{(1)}=\sum_{k=1}^n\frac{1}{k}$,
  \begin{equation}\label{eq:Green:mean:var}
    \mathbb{E}(L_n(g_p))=\kappa_1(L_n(g_p))=-\frac{n}{2},
    \quad \mathrm{Var}(L_n(g_p))=\kappa_2(L_n(g_p))=\frac{1}{4}H_n\sim\frac{\log(n)}{4}
   \end{equation}
   while for all $m\geq3$, denoting $\zeta(s)=\sum_{n=1}^\infty\frac{1}{n^s}$ the Riemann zeta function,
   \begin{equation}
    \lim_{n\to\infty}\kappa_m(L_n(g_p)) = \frac{(-1)^m (m-1)!}{2^m}\zeta(m-1).
  \end{equation}
\end{lemma}

\begin{proof}
  By the $\mathrm{O}(3)$ invariance of \eqref{eq:gas:S}, the law of
  $L_n(g_p)$ does not depend on $p$. We take $p=\infty$ in order to benefit from its radiality (we could alternatively take $g_0$). We have
  $g_\infty(w)=-\frac{1}{2}\log(1+|w|^2)$. For all $1\leq k\leq n$, with
  $\xi_{n,k}$ as in \eqref{eq:kostlan2}, the Laplace transform of
  $S_{n,k}=-\frac{1}{2}\log(1+\xi_{n,k}^2)$ is 
  \begin{equation}\label{eq:lap_uinf_wnk}
    \mathbb{E}(\mathrm{e}^{\theta S_{n,k}})
    =\mathbb{E}((1+|\xi_{n,k}|^2)^{-\frac{\theta}{2}})
    =\int_0^\infty\frac{n!}{(k-1)!(n-k)!}\frac{x^{k-1}(1+x)^{\frac{-\theta}{2}}}{(1+x)^{n+1}}\mathrm{d}x
    =\frac{n!}{(n-k)!}\frac{\Gamma(n-k+1+\frac{\theta}{2})}{\Gamma(n+1+\frac{\theta}{2})}
  \end{equation}
  for all real $\theta>-2(n-k+1)$. It follows then by \eqref{eq:kostlan2} that
  the log-Laplace transform of $S_n=\sum_{k=1}^nS_{n,k}=L_n(g_\infty)$ is
  \begin{equation}\label{eq:logmgf}
    \log\mathbb{E}(\mathrm{e}^{\theta S_n})
    =n\bigl(\log\Gamma(n+1)-\log\Gamma(n+1+\tfrac{\theta}{2})\bigr)
    +\sum_{k=0}^{n-1}\bigl(\log\Gamma(k+1+\tfrac{\theta}{2})-\log\Gamma(k+1)\bigr)
  \end{equation}
  for every $\theta > -2$. Differentiating
  \eqref{eq:logmgf} at $\theta=0$ gives the cumulants of $S_n$:
  \begin{equation}\label{eq:cum-psi}
    \kappa_m(S_n)
    =
      \frac{1}{2^m}\big(\sum_{k=1}^n \psi_{m-1}(k)-n\psi_{m-1}(n+1)\big)
    =\frac{1}{2^m}\sum_{k=1}^n \big(\psi_{m-1}(k)-
    \psi_{m-1}(n+1) \big)  
  \end{equation}
  where $\psi_m=\psi^{(m)}$ is the polygamma
  functions given by
  $\frac{\mathrm{d}^m}{\mathrm{d}\theta^m}\log\Gamma(x+\tfrac{\theta}{2})\big|_{\theta=0}=\frac{1}{2^m}\psi_{m-1}(x)$.
  We have by using a recursion argument after
  differentiating $\log \Gamma(s+1) = \log s + \log \Gamma(s)$ that for
  $k \in \{1,\dots,n\}$
  \[\psi_{m-1}(n+1) - \psi_{m-1}(k)
  = (m-1)!(-1)^{m-1}(H_n^{(m)} - H_{k-1}^{(m)}).\]
  The sum
    $\sum_{k=1}^n (H_n^{(m)} - H_{k-1}^{(m)})
  =\sum_{k=1}^n (\sum_{j=1}^n \frac{1}{j^m} 1_{j \geq k})
  =\sum_{j=1}^n (\sum_{k=1}^n \frac{1}{j^m} 1_{j \geq k})
  =\sum_{j=1}^n \frac{1}{j^{m-1}} = H_n^{(m-1)}$ 
  allows us to obtain
   \begin{equation}\label{eq:kappa-harmonic}
    \kappa_m(S_n) = \frac{(-1)^m (m-1)!}{2^m}H_{n}^{(m-1)}.
  \end{equation}
  
\end{proof}

\subsection{Proof of Corollary \ref{co:Green:CLT}}
\label{se:proof:co:Green:CLT}

From \eqref{eq:Green:mean:var} we know that $\mathbb{E}(L_n(g_p))=-\frac{n}{2}$, for all $p\in\overline{\mathbb{C}}$.
Proving \eqref{eq:Green:multi} is equivalent to prove the following multi-point CLT : for $k\geq1$ and all distinct points $p_1,\ldots,p_k\in\overline{\mathbb{C}}$, 
\begin{equation}
    \left(\frac{L_n(g_{p_1})+\frac{n}{2}}{\sqrt{\frac{1}{4}\log(n)}},\ldots,\frac{L_n(g_{p_k})+\frac{n}{2}}{\sqrt{\frac{1}{4}\log(n)}}\right)
    \xrightarrow[n\to\infty]{\mathrm{d}}
    \mathcal{N}(0,I_k).
\end{equation}
Equivalently, it suffices to prove that for all $c_1,\ldots,c_k\in\mathbb{R}$,
\begin{equation}
    \frac{L_n(c_1g_{p_1}+\cdots+c_kg_{p_k})+\frac{n}{2}(c_1+\cdots+c_k)}{\sqrt{\frac{1}{4}\log(n)}}
    \xrightarrow[n\to\infty]{\mathrm{d}}
    \mathcal{N}(0,c_1^2+\cdots+c_k^2).
\end{equation}
This holds by \eqref{eq:LogSing} since $c_1g_{p_1}+\cdots+c_kg_{p_k}$ has $k$ logarithmic singularities at points $p_1,\ldots,p_k$ with weights $c_1,\ldots,c_k$, and has mean $-\frac{n}{2}(c_1+\cdots+c_k)$ according to \eqref{eq:Green:mean:var}. 

\subsection{Proof of Theorem \ref{th:Green:Cov}}
\label{se:proof:th:Green:Cov}

The formula for the variance in \eqref{eq:Green:cov} comes from \eqref{eq:Green:mean:var}. The rest of this section is devoted to the proof of the formula for the off-diagonal covariance in \eqref{eq:Green:cov}. It is split in three lemmas.

\begin{lemma}[Covariance for antipodal points]\label{le:cov-antipodal}
  For all $p\in\overline{\mathbb{C}}$, denoting $p^*=T(-T^{-1}(p))$ its
  antipodal, for all $n\geq 1$, 
  \begin{equation}\label{eq:cov-limit-antip}
    \mathrm{Cov}\big(L_n(g_p),L_n(g_{p^*})\big)
    =-\frac{n}{4}\psi_1(n+1)=-\frac14+O_{n\to\infty}\Big(\frac{1}{n}\Big)
  \end{equation}
  where $\psi_1$ is the polygamma function of order $1$, second order derivative of $\log\Gamma$.
\end{lemma}

\begin{proof}
  By the invariance 
  under rotations of \eqref{eq:gas:S}, the result is the
  same for any $p$. We take $p=0$, which gives $p^*=\infty$. Now, for
  $t,s\in\mathbb{R}$, we set $g=t\,g_0+sg_\infty$. Since $g$ is radial, the
  Kostlan formula \eqref{eq:kostlan1} gives 
  \begin{equation}
    \mathbb{E}(\mathrm{e}^{\mathrm{i}(tL_n(g_0)+sL_n(g_\infty))})
    =\prod_{k=0}^{n-1}\frac{B(k+1+\mathrm{i}\frac{t}{2},n-k+\mathrm{i}\frac{s}{2})}{B(k+1,n-k)}.
  \end{equation}
  where $B$ is the Euler Beta function. Taking logarithms and using $B(a,b)=\Gamma(a)\Gamma(b)/\Gamma(a+b)$ gives
  \begin{equation}
    \log\mathbb{E}(\mathrm{e}^{\mathrm{i}(tL_n(g_0)+sL_n(g_\infty))})
    =\sum_{k=0}^{n-1}\Big(\log\Gamma(k+1+\mathrm{i}\tfrac{t}{2})+\log\Gamma(n-k+\mathrm{i}\tfrac{s}{2})-\log\Gamma(n+1+\mathrm{i}\tfrac{t+s}{2})
    \Big)+C,
  \end{equation}
  where $C$ does not depend on $(t,s)$. Only the last term contributes to the
  mixed derivative at $(t,s)=(0,0)$, hence
  \begin{equation}
    -\frac{\partial^2}{\partial t\partial s}
    \log \mathbb{E}(\mathrm{e}^{\mathrm{i}(tL_n(g_0)+sL_n(g_\infty))})\Big|_{t=s=0}
    = -\frac{n}{4}\psi_1(n+1),
  \end{equation}
  which is the covariance. Finally, the asymptotics comes from $\psi_1(x)=1/x+O_{x\to\infty}(1/x^2)$.
\end{proof}

\begin{lemma}[Covariance for distinct non-antipodal points]\label{le:cov-nonantipodal}
  If $p,q\in\overline{\mathbb{C}}$, $x\cdot y\in(-1,1)$ where $p=T(x)$, $q=T(y)$, then 
  \begin{equation}\label{eq:cov-limit}
    \mathrm{Cov}(L_n(g_p),L_n(g_q))
    =-\frac14-\frac14\log\Big(\frac{1-x\cdot y}{2}\Big)+o_{n\to\infty}(1)
    =-\frac14-\frac12\log d(p,q)+o_{n\to\infty}(1).
  \end{equation}
\end{lemma}

According to \eqref{eq:cov-limit-antip}, the formula \eqref{eq:cov-limit}
remains valid when the points are antipodal,
$x \cdot y = -1$, but \eqref{eq:cov-limit-antip}
is sharper.

\begin{proof}
For notational convenience,
in this proof we work
on $\mathbb S^2$
instead of $\overline{\mathbb C}$,
without explicitly referring to the map $T$
that sends $\mathbb S^2$ to $\overline{\mathbb C}$.
So, we may write
$g_x$ instead of the slightly more cumbersome
$g_{T(x)}$. 

  \medskip

  \noindent\emph{Step 1. Covariance via Berezin transform.} 

  \medskip
  
    From \eqref{eq:CovLnLn2}, we get, for all $f,g:\mathbb{S}^2\to\mathbb{R}$ in  $L^2(\nu)$,
  \begin{equation}\label{eq:cov-berezin}
    \mathrm{Cov}(L_n(f),L_n(g))
    =n\int_{\mathbb{S}^2}f(x)(g(x)-\mathcal{B}_ng(x))\mathrm{d}\nu(x)
    =n\langle f,(I-\mathcal{B}_n)g\rangle_{L^2(\nu)}
  \end{equation}
  where $\mathcal{B}_n$ is the Berezin transform\footnote{Note that $\mathcal{B}_n$ is 
    self-adjoint on $L^2_{\mathbb{C}}(\nu)$, with kernel
    $\frac{1}{n}|K_n(x,y)|^2$, and it is Markov: $\mathcal{B}_n1=1$ and $\mathcal{B}_ng\geq0$ when $g\geq0$. See also \cite{MR3494849}.}
  \begin{equation}\label{eq:berezin}
    (\mathcal{B}_ng)(y)=\frac{1}{n}\int_{\mathbb{S}^2} g(x)|K_n(x,y)|^2\mathrm{d}\nu(x),
  \end{equation}
  where 
  an explicit formula for $|K_n(x,y)|^2$ can be found in \eqref{eq:Knxy}.

  \medskip

  \noindent\emph{Step 2. Harmonic expansion via the Green kernel.}

  \medskip
  
    By denoting
  $G_x = G(x,\cdot) = -g_x -1/2$,
  we get from \eqref{eq:cov-berezin} that for all $x,y\in \mathbb S^2$
  \begin{equation}
    \mathrm{Cov}\big(L_n(G_x),L_n(G_y)\big)
    =n\langle G_x,(I-\mathcal{B}_n)G_y\rangle_{L^2(\nu)}.
  \end{equation}
  Since
  $G_x = \frac{1}{2} \sum_{\ell=1}^\infty
  \frac{1}{\ell(\ell+1)} c_\ell(x,\cdot)$
  as recalled in
  \eqref{eq:green-series}, where the
  convergence occurs in $L^2$,
  and since $c_\ell(x,\cdot)$ is an eigenvector
  of $I- \mathcal B_n$ of eigenvalue
  $1 - b_{n,\ell}$ from Lemma \ref{lem:eigenvalues}, 
  we can write
  \[(I - \mathcal B_n) G_x
  = \frac{1}{2}\sum_{\ell = 1}^\infty 
  \frac{1-b_{n,\ell}}{\ell(\ell + 1)}c_\ell(x,\cdot).\]
  Notice that $c_\ell(x,\cdot)$ and $c_{\ell'}(y,\cdot)$
  are orthogonal for $\ell \neq \ell'$
  since $c_\ell$ and $c_{\ell'}$ project 
  to orthogonal subspaces, 
  and notice that
  $\langle c_\ell(x,\cdot), c_\ell(y,\cdot) 
  \rangle_{L^2(\nu)} 
  = c_\ell(x,y)$ since $c_\ell$ is a projection.
  Then, we have
  \begin{equation}
  \mathrm{Cov}\big(L_n(G_x),L_n(G_y)\big)
    =
  n\langle G_x,(I-\mathcal{B}_n)G_y\rangle_{L^2(\nu)}
  =
  \frac{n}{4}\sum_{\ell=1}^\infty \big(1-b_{n,\ell}\big)\frac{c_\ell(x,y)}{\ell^2(\ell+1)^2}
  =
   \frac{n}{4}\sum_{\ell=1}^\infty \big(1-b_{n,\ell}\big)
    \frac{2\ell+1}{\ell^2(\ell+1)^2}P_\ell(x\cdot y),
  \end{equation}
where \eqref{eq:addition-identity},
$c_\ell(x,y)=(2\ell+1)P_\ell(x\cdot y)$,
was used.

\medskip
\noindent\emph{Step 3. Asymptotics.}
\medskip

\noindent
  By Lemma \ref{lem:eigenvalues} we know that
  $b_{n,\ell} =  
  \prod_{j=1}^{\ell} \alpha_{n,j}$
  with 
  $\alpha_{n,j} =\frac{n-j}{n+j}$.
  We may show that
  $n(1-b_{n,\ell}) \leq \ell (\ell+1)$. Indeed,
  if $X_1,\dots,X_{n-1}$ are independent Bernoulli
  random variables with
  $\mathbb E(X_j) = \alpha_{n,j}$ we get, for
  $\ell \in \{1,\dots,n-1\}$,
  \[1-b_{n,\ell} = 1-\prod_{j=1}^\ell \alpha_{n,j}
  = \mathbb P\big(\cup_{j=1}^\ell \{X_j = 0\} \big)
  \leq 
  \sum_{j=1}^\ell \mathbb P(X_j = 0)
  = \sum_{j=1}^\ell (1-\alpha_{n,j})
  =\sum_{j=1}^\ell \frac{2j}{n+j}
  \leq \frac{2}{n} \sum_{j=1}^\ell j
  = \frac{\ell (\ell+1)}{n}.\]
  Notice that
  $1-b_{n,\ell} \leq n^{-1} \ell(\ell+1)$ also holds 
  for $\ell \geq n$ since
  in that case $1-b_{n,\ell}  = 1$.
    It is known, see 
    \cite[Th.~7.32.2 with $\alpha=\beta=0$]{zbMATH03477793}\footnote{More precisely, this gives $|P_\ell(\cos(\theta))|\leq C\ell^{-1/2}\theta^{-1/2}$ for $1/\ell\leq\theta\leq\pi/2$. Combined with the simple bound $\sup_{\alpha\in[-1,1]}|P_\ell(\alpha)|\leq1$, this gives after discussion that  $\sqrt{\sin(\theta)}|P_\ell(\cos(\theta))|\leq C\ell^{-1/2}$ for all $0<\theta<\pi$, in other words $(1-\alpha^2)^{1/4}|P_\ell(\alpha)|\leq C\ell^{-1/2}$ for all $-1<\alpha<1$.}, that
    for $\alpha \in (-1,1)$, there is a constant
    $C_\alpha > 0$ such that
  \begin{equation}\label{eq:legendre-bound}
    |P_\ell(\alpha)|\leq\frac{C_\alpha}{\sqrt{\ell}},\quad \ell\geq 1.
  \end{equation}
  We can then bound\footnote{We see here the need
  for $x$ and $y$ not to be antipodal. Otherwise, we could not dominate the series
  by a summable one
  since the limit is not absolutely convergent
  because $P_\ell(-1) = (-1)^\ell$.}
  \begin{equation}
    \frac{n}{4} \big(1-b_{n,\ell}\big)
    \frac{2\ell+1}{\ell^2(\ell+1)^2}|P_\ell(x\cdot y)|
    \leq  \frac{2\ell + 1}{4\ell (\ell+1)} \frac{C_\alpha}{\sqrt \ell}, 
  \end{equation}
  where the right-hand side does not depend on $n$ and it is summable in $\ell$. 
  Thus, we are able to use the dominated convergence theorem as soon as we find the limit
  of $n(1-b_{n,\ell})$ as $n \to \infty$.
  For this we take the logarithm of $b_{n,\ell}$ and use that $\log(1-y)-\log(1+y)=-2y+O_{y\to0}(y^3)$ to get
  \begin{equation}
    \log b_{n,\ell}
    = \sum_{j=1}^\ell\Big(\log\Bigl(1-\frac{j}{n}\Bigr)-\log\Bigl(1+\frac{j}{n}\Bigr)\Big)
    = -\frac{\ell(\ell+1)}{n}+O_{n\to\infty}\Big(\frac{1}{n^3}\Big).
  \end{equation}
    Then, taking the exponential, we get, for fixed $\ell$,
    \begin{equation}
    b_{n,\ell} = 
    \exp\bigg(
    -\frac{\ell(\ell+1)}{n}+O_{n\to\infty}\Big(\frac{1}{n^3}\Big)
    \bigg)
    =1 - \frac{\ell(\ell+1)}{n} +
    O_{n\to\infty}\Big(\frac{1}{n^2}\Big)
    \end{equation}
    so that
    \[\lim_{n \to \infty} n(1-b_{n,\ell}) = 
    \ell(\ell+1).\]

\vspace{2mm}
  \noindent
  \emph{Step 4. Identification of the limit.} 

\vspace{2mm}
  \noindent
  We may
  now apply the dominated convergence theorem to conclude
  \begin{equation}
    \mathrm{Cov}\big(L_n(g_x),L_n(g_y)\big)
    = \frac14\sum_{\ell=1}^\infty \frac{2\ell+1}{\ell(\ell+1)}P_\ell(x \cdot y)+o_{n\to\infty}(1)
    = \frac14\sum_{\ell=1}^\infty \frac{1}{\ell(\ell+1)}c_\ell(x,y)+o_{n\to\infty}(1)
  \end{equation}
    which, by \eqref{eq:green-series},
    gives us
    $\frac{1}{2} G(x,y)$ or, equivalently, $
    = - \frac{1}{2} g_x(y) - \frac{1}{4}
    = - \frac{1}{2}\log d(p,q) - \frac{1}{4}$
  and we obtain \eqref{eq:cov-limit}.
\end{proof}

Since $(w,u)\mapsto |K_n(w,u)|=n(1-d(w,u)^2)^{\frac{n-1}{2}}$ depends only on
$d(w,u)$, the operator $\mathcal{B}_n$ is $\mathrm{SO}(3)$ invariant. Therefore it is
diagonal in the spherical harmonic decomposition. We state this in the
Legendre basis.

\begin{lemma}[Eigenvalues of the Berezin transform
  $\mathcal{B}_n$ \eqref{eq:berezin}]\label{lem:eigenvalues}
  Recall that
  $\mathcal H_\ell$ denotes the eigenspace
  of $-\Delta$ with eigenvalue $4\pi \ell (\ell+1)$.
  For every $n\geq1$, the 
  $\mathcal H_\ell$ is also eigenspace of 
  $\mathcal B_n$ of eigenvalue
\begin{equation}\label{eq:eigenvalue-formula}
    b_{n,\ell}
      =\prod_{j=1}^{\ell}\frac{n-j}{n+j}.
  \end{equation}
  Equivalently, for every $p \in 
  \overline{\mathbb C}$, the function
  $Q_\ell(q) = P_\ell(T^{-1}(p) \cdot T^{-1}(q))$
  is an eigenvector
  of $\mathcal B_n$ with eigenvalue
  $b_{n,\ell}$,
  \[\mathcal B_n Q_\ell = b_{n,\ell} Q_\ell.\]

\end{lemma}

Notice that $b_{n,\ell} = 0$ if $\ell \geq n$ and that
$b_{n,0} = 1$ since the product would be empty.

\begin{proof}
As in the proof of Lemma \ref{le:cov-nonantipodal},
we work on $\mathbb S^2$ instead of $\overline{\mathbb C}$
for notational simplicity.
Define $k_n(t) = n(\frac{1+t}{2})^{n-1}$
which is a polynomial of degree $n-1$.
   Combining
  \eqref{eq:kernel-distance} and \eqref{eq:d2xy} gives
  \begin{equation}
    \frac1n|K_n(x,y)|^2
    =n\bigl(1-d(x,y)^2\bigr)^{n-1}
    =n\Big(\frac{1+x\cdot y}{2}\Big)^{n-1}
    =k_n(x\cdot y).
  \end{equation}
  The Berezin transform \eqref{eq:berezin}
  of $g \in L^2_{\mathbb C}(\nu)$ is
  \begin{equation}
    (\mathcal B_n g) (x)
    =\int_{\mathbb{S}^2} \,k_n(x\cdot y) g(y)\mathrm{d}\nu(y).
  \end{equation}

\noindent
  \emph{Spherical-harmonic expansion of the kernel.} 
  Expand the polynomial $k_n$ in the basis
  of Legendre polynomials as
  \begin{equation}
  \label{eq:PolynomialExpansion}
  k_n = \sum_{\ell=0}^{n-1} b_{n,\ell} (2\ell + 1)
  P_\ell.
  \end{equation}
  From the orthogonality relations
  $\int_{-1}^1 P_\ell(t) P_{\ell'}(t) \mathrm d t
  =\frac{2}{2\ell +1} \delta_{\ell, \ell'}$
  we obtain
  \begin{equation}
  \label{eq:bnlFormula}
  b_{n,\ell} = \frac{1}{2} \int_{-1}^1 
  k_n(t) P_\ell(t) \mathrm d t . 
  \end{equation}
From \eqref{eq:PolynomialExpansion},
if $C_\ell$ denotes the orthogonal projection map
onto $\mathcal H_\ell$ which has kernel
$(2\ell + 1)P_\ell(x\cdot y)$, 
we get\footnote{The invariance under rotations
of $\mathcal B_n$ already
tells us that $\mathcal H_\ell$ are eigenspaces
of $\mathcal B_n$ but a direct calculation
also works as explained here.}
\begin{equation}
\label{eq:BnExpansion}
\mathcal B_n = \sum_{\ell=0}^{n-1} b_{n,\ell} 
C_\ell.
\end{equation}
This tells us that $\mathcal H_\ell$
is an eigenspace of $\mathcal B_n$
with eigenvalue $b_{n,\ell}$.

\vspace{2mm}

\noindent
  \emph{Computation of $b_{n,\ell}$.} 
  Since \eqref{eq:BnExpansion}
  is a sum from $0$ to $n-1$, we already
  know that $b_{n,\ell} = 0$ for $\ell \geq n$.
  By the change of variables
  $t=2s-1$ in \eqref{eq:bnlFormula} we can write
  \begin{equation}
    b_{n,\ell}
    =\frac{n}{2}\int_{-1}^1\Big(\frac{1+t}{2}\Big)^{n-1}P_\ell(t)\mathrm{d}t
    =n\int_0^1 s^{n-1}P_\ell(2s-1)\mathrm{d}s.
  \end{equation}
  Now, the Rodrigues formula for Legendre polynomials in the $s$-variable gives
  \begin{equation}
    P_\ell(2s-1)=\frac{(-1)^\ell}{\ell!}\frac{\mathrm{d}^\ell}{\mathrm{d}s^\ell}\big(s^\ell(1-s)^\ell\big).
  \end{equation}
  Next, by integrating by parts $\ell$ times, we get\footnote{The boundary
    terms vanish since $s^\ell(1-s)^\ell$ has zeros of order $\ell$ at $0$ and
    $1$.}
  \begin{equation}
    b_{n,\ell}
    =\frac{n}{\ell!}\int_0^1 \frac{\mathrm{d}^\ell}{\mathrm{d}s^\ell}(s^{n-1})s^\ell(1-s)^\ell\mathrm{d}s.
  \end{equation}
If $0\leq\ell\leq n-1$, then
  $\frac{\mathrm{d}^\ell}{\mathrm{d}s^\ell}(s^{n-1})=\frac{(n-1)!}{(n-1-\ell)!}s^{n-1-\ell}$,
  and thus
  \begin{equation}
    b_{n,\ell}
    =\frac{n}{\ell!}\frac{(n-1)!}{(n-1-\ell)!}\int_0^1 s^{n-1}(1-s)^\ell\mathrm{d}s
    =\frac{n}{\ell!}\frac{(n-1)!}{(n-1-\ell)!}
    \Bigg(\frac{(n-1)!}{(n+\ell)!}\ell!\Bigg)
    =\frac{(n-1)!}{(n-1-\ell)!} \frac{n!}{(n+\ell)!}
    =\prod_{j=1}^{\ell}\frac{n-j}{n+j}.
  \end{equation}

\end{proof}

\section{Logarithmic potential : proof of Corollaries \ref{co:LogPot:CLT} and \ref{co:LogPot:Cov}}
\label{se:proof:LogPot}

\subsection{Proof of Corollary \ref{co:LogPot:CLT}}
\label{se:proof:co:LogPot:CLT}

\subsubsection{Mean}

For all $z\in\mathbb{C}$, we can find the
expected value of $L_n(f_z)$ as
\begin{align}
  \mathbb{E}(L_n(f_z))
  &=\mathbb{E}(\log|\det(AB^{-1}-zI)|)\\
  &=\mathbb{E}(\log|\det(A-zB)|)-\mathbb{E}(\log|\det(B)|)\\
  &=\mathbb{E}(\log|(1+|z|^2)^{\frac{n}{2}}\det(B)|)-\mathbb{E}(\log|\det(B)|)\\
  &=\frac{n}{2}\log(1+|z|^2), \label{eq:LogPot:mean}
\end{align}
where we used $A-zB\overset{\mathrm{d}}{=}\sqrt{1+|z|^2}B$. Alternatively, we could use
\eqref{eq:split-euclid} and \eqref{eq:Green:mean:var}, or \eqref{eq:ELn}.

\subsubsection{CLT}

Proving \eqref{eq:logpot:CLT} is equivalent to show that for all $k\geq1$ and distinct $z_1,\ldots,z_k\in\mathbb{C}$, with $m_i=\frac{n}{2}\log(1+|z_i|^2)$,
\begin{equation}
    \left(\frac{L_n(f_{z_1})-m_1}{\sqrt{\frac{1}{2}\log(n)}},\ldots,\frac{L_n(f_{z_k})-m_k}{\sqrt{\frac{1}{2}\log(n)}}\right)
    \xrightarrow[n\to\infty]{\mathrm{d}}
    \mathcal{N}(0,K)
    \quad\text{where}\quad
    K_{ij}=\frac{\mathbb{1}_{i=j}+1}{2}.
\end{equation}
This is equivalent to state that for all $c_1,\ldots,c_k\in\mathbb{R}$,
\begin{equation}
    \frac{L_n(c_1f_{z_1}+\cdots+c_kf_{z_k})-(c_1m_1+\cdots+c_km_k)}{\sqrt{\frac{1}{2}\log(n)}}
    \xrightarrow[n\to\infty]{\mathrm{d}}
    \mathcal{N}(0,\langle Kc,c\rangle).
\end{equation}
Now $c_1f_{z_1}+\cdots+c_kf_{z_k}$ has $k+1$ logarithmic singularities at points $z_1,\ldots,z_k,\infty$ with weights
$c_1,\ldots,c_k,-(c_1+\cdots+c_k)$. Its mean is $c_1m_1+\cdots+c_km_k$ according to \eqref{eq:LogPot:mean}. The desired result follows then from \eqref{eq:LogSing} since
\begin{equation}
    c_1^2+\cdots+c_k^2+(c_1+\cdots+c_k)^2
    =\sum_{i=1}^kc_i^2+\sum_{1\leq i,j\leq k}c_ic_j
    =2\sum_{i,j=1}^kK_{ij}c_ic_j.
\end{equation}
Alternatively, the splitting \eqref{eq:split-euclid} gives
\begin{equation}
    c_1f_{z_1}+\cdots+c_kf_{z_k}-\tfrac{1}{n}(c_1m_1+\cdots+c_km_k)
    =c_1g_{z_1}+\cdots+c_kg_{z_k}-(c_1+\cdots+c_k)g_\infty
\end{equation}
connecting directly \eqref{eq:Green:multi} and \eqref{eq:LogPot:multi}.

\subsection{Proof of Corollary \ref{co:LogPot:Cov}}
\label{se:proof:co:LogPot:Cov}

\subsubsection{Variance}

To prove \eqref{eq:LogPot:var}, we combine \eqref{eq:split-euclid}, and \eqref{eq:Green:cov} to get, for all $z\in\mathbb{C}$,
\begin{align}
  \mathrm{Var}(L_n(f_z))
  &=\mathrm{Var}(L_n(g_z-g_\infty))\\
  &=\mathrm{Var}(L_n(g_z))+\mathrm{Var}(L_n(g_\infty))-2\mathrm{Cov}(L_n(g_z),L_n(g_\infty))\\
  &=\frac14H_n+\frac14H_n-2\bigl(-\frac14+\frac14\log(1+|z|^2)+o_{n\to\infty}(1)\bigr)\\
  &=\frac{H_n}{2}+\frac{1}{2}-\frac{1}{2}\log(1+|z|^2)+o_{n\to\infty}(1).
\end{align}

\subsubsection{Covariance}

To get \eqref{eq:LogPot:cov}, we write, by combining \eqref{eq:split-euclid} and \eqref{eq:Green:cov}, for all
$z,w\in\mathbb{C}$, $z\neq w$,
\begin{align}
    \mathrm{Cov}(L_n(f_z),L_n(f_w))
    &=\mathrm{Cov}(L_n(g_z)-L_n(g_\infty),L_n(g_w)-L_n(g_\infty))\\
    &=\mathrm{Cov}(L_n(g_z),L_n(g_w))
    -\mathrm{Cov}(L_n(g_z),L_n(g_\infty))
    -\mathrm{Cov}(L_n(g_w),L_n(g_\infty))
    +\mathrm{Var}(L_n(g_\infty))\\
    &=\frac{1}{4}-\frac{1}{2}\log d(z,w)+\frac{1}{2}\log d(z,\infty)+\frac{1}{2}\log d(w,\infty)+\frac{H_n}{4}\\
    &=\frac{1}{4}-\frac{1}{2}\log|z-w|+\frac{H_n}{4}+o_{n\to\infty}(1).
\end{align}

\section{Matrix universality : proof of Theorem \ref{th:logpot:univ}}
\label{se:proof:th:logpot:univ}

Throughout this section, we will use $A$ and $B$ to denote two independent copies of Girko matrices, and use $A^\mathrm{Gin}$ and $B^\mathrm{Gin}$ to denote two independent Ginibre matrices, which are special cases of Girko matrices. For Ginibre matrices, $M^\mathrm{Gin}=(A^{\mathrm{Gin}})(B^\mathrm{Gin})^{-1}$ is indeed the spherical ensemble. More generally we use ${}^{\mathrm{Gin}}$ superscript to indicate the underlying model which is in use.

Given $M=AB^{-1}\in \C^{n\times n}$, define the empirical spectral measure of $M$ by $\mu_M:=\frac{1}{n}\sum_{i=1}^n \delta_{\sigma_i}$, where $\{\sigma_i\}_{i=1}^n$ are the eigenvalues of $M$. 
In this section we use the following short-hand notations: 
\begin{align}
L_n(f):=\sum_{i=1}^n f(\sigma_i)=n\int f\dd\mu_M,
\end{align}
for any test function $f:\C \rightarrow \R$. In particular, for any fixed $w\in \C$ and $p \in \overline{\C}$, we write
\begin{align}
L_n(f_w)=&\log |\det (M-w)|=n\int f_w \dd\mu_M, \quad \qquad  f_w(\cdot):=\log|\cdot -w|,\label{eq_fw}\\
L_n(h_p)=&n\int h_p \dd\mu_M, \quad \qquad  h_p(\cdot):=
\begin{cases}
\log|\cdot -p| \chi(|\cdot -p|), \qquad & p\in \C \\
h_0\Big(\frac{1}{\cdot}\Big)=-\log |\cdot| \chi\Big(\frac{1}{|~\cdot~|}\Big), \qquad  & p=\infty
\end{cases},\label{eq_h}
\end{align}
where $\chi:\R \rightarrow \R^+$ is a smooth and non-increasing cut-off function centered at zero.  In addition, for any random variable $X$, we use $\mathbb{Q}[X]:=X-\E[X]$ to denote the centered version for brevity.

\medskip

The proof of Theorem \ref{th:logpot:univ} crucially relies on the following technical result of comparison. Instead of studying $L_n(g_p)$ in the decomposition \eqref{eq:LogSingAlt} of function with logarithmic singularities, we focus on $L_n(h_p)$ by a slight different decomposition in \eqref{eq_decomp} below. Though the behavior of $L_n(f_w)$ is determined by $L_n(g_w)$ and $L_n(g_\infty)$, we still include a direct comparison for $L_n(f_w)$ as a toy case in our proof.

\begin{theorem}[Comparison theorem]\label{th:replacement}\ \\
    Fix constants $C_1,C_2,C_3,C_4>0$ and integers $k_1,k_2,k_3\in \mathbb{N}$. Let $\{w_i\}_{i=1}^{k_1}$ be fixed points in $\mathbb{C}$ and $\{p_i\}_{i=1}^{k_2}$ be fixed points in $\overline{\mathbb{C}}$.   
     For any $1\leq i\leq k_3$, let $(f^{(i)}_n)$ be sequences of regular test functions such that {\footnote{Here $\mathcal{C}_0^{2}(\C\to\mathbb{R})$ means the class of $\mathcal{C}^2$ functions vanishing at infinity.}}
     \begin{equation}\label{eq_f}
     f^{(i)}_n \in\mathcal{C}_0^{2}(\C\to\mathbb{R}),
     \qquad\quad  \int|\Delta_z f^{(i)}_n(z)| \dd^2 z  \leq C_1(1+\log n)^{C_2},
     \end{equation}
     uniformly in $n$. Set $K=k_1+k_2+k_3$. Let $F_n:\mathbb{R}^{K}\to\mathbb{C}$ be a sequence in $\mathcal{C}^{5}(\R^K \to\mathbb{C})$ such that for any $1\leq l\leq 5$,
     \begin{equation}\label{eq_F}
    \max_{l_1+\cdots+l_{K}=l}\left|\frac{\partial^{l} F_n(x_1,\cdots, x_{K})}{\partial x_1^{l_1} \cdots \partial x_K^{l_K}}\right| \leq C_3\prod_{i=1}^K\Big(1+\log n +|x_i|\Big)^{C_4},
     \end{equation}
     uniformly in $n$ and for any $(x_1,\cdots, x_K)\in \R^K$. Define 
\begin{align}\label{XX}
\mathcal{X}_n=F_n\left(\mathbb{Q}\big[L_n(f_{w_1})\big], \cdots , \mathbb{Q}\big[L_n(f_{w_{k_1}})\big], \mathbb{Q}\big[L_n(h_{p_1})\big], \cdots , \mathbb{Q}\big[L_n(h_{p_{k_2}})\big],\mathbb{Q}\big[L_n(f_n^{(1)})\big],\cdots, \mathbb{Q}\big[L_n(f_n^{(k_3)})\big] \right),
\end{align}
     and define $\mathcal{X}^{\mathrm{Gin}}_n$ similarly by replacing  $M$ with $M^{\mathrm{Gin}}$. Then, for all $0<\delta<c_0/400$ with $c_0$ being the constant from \ref{it:condmom4}, we have
    \begin{align}\label{eq_diff}
    \Big|\mathbb{E}\big[\mathcal{X}_n\big]
    -\mathbb{E}\big[\mathcal{X}^{\mathrm{Gin}}_n\big]\Big| =O\big( n^{-\delta}\big).
    \end{align}
    Moreover,  for the expectations, we have 
    \begin{align}\label{eq_diff_mean}
    	\Big|\E\big[L_n(f)\big] 
    	-\E^{\mathrm{Gin}}\big[L_n(f)\big]\Big| =O\big( n^{-\delta}\big), \qquad \forall f \in \Big\{ f_{w_i},h_{p_i}, f^{(i)}_n\Big\},
    \end{align}
where $\E^{\mathrm{Gin}}$ denotes the corresponding expectation with $M$ replaced by $M^{\mathrm{Gin}}$.
\end{theorem}

Armed with Theorem \ref{th:replacement}, we are ready to prove Theorem \ref{th:logpot:univ}.

\subsection{Proof of Theorem \ref{th:logpot:univ}}

We first extend Theorem \ref{th:LogSing} beyond the spherical ensemble. From Definition \ref{df:main:LogSing}, it suffices to consider test functions with multiple logarithmic singularities of the following form
\begin{align}\label{eq_decomp}
f=c_1 h_{p_1}+\cdots +c_k h_{p_k}+\mathfrak{h}+\mathfrak{c}, \qquad c_i \in \R \setminus \{0\}, \qquad p_i \in \overline{\C},
\end{align}
with $h_{p}$ defined in \eqref{eq_h},  $\mathfrak{h} \in \mathcal{C}_0^2$ being a regular function satisfying (\ref{eq_f}), and $\mathfrak{c}\in \R$ being a constant. Note that the decomposition above is slightly different than \eqref{eq:LogSingAlt}, and the constant $\mathfrak{c}$ is determined by the value of $f$ at the infinity point.  Applying the comparison principle \eqref{eq_diff_mean} to functions $\{h_{p_i}\}_{i=1}^{k}$ and $\mathfrak{h}$, we have
\begin{align}\label{eq_goal_mean}
\E\big[L_n(f)\big]=&\sum_{i=1}^k c_i\E\big[L_n(h_{p_{i}})\big]+\E\big[L_n(\mathfrak{h})\big]+n \mathfrak{c}=\sum_{i=1}^k c_i\E^{\mathrm{Gin}}\big[L_n(h_{p_{i}})\big]+\E^{\mathrm{Gin}}\big[L_n(\mathfrak{h})\big]+n \mathfrak{c}+O(n^{-\delta})\\
=&\E^{\mathrm{Gin}}\big[L_n(f)\big]+O(n^{-\delta})=n\int f\mathrm d \mu+O(n^{-\delta}),
\end{align}
where we also used the Ginibre expectation in \eqref{eq:ELn}. It then suffices to prove
   \begin{equation}\label{eq_goal_dist}
   \frac{\mathbb{Q} [L_n(f)]}{\sqrt{\frac{1}{4}\log n}} =\frac{\sum_{i=1}^{k} c_i \mathbb{Q} [L_n(h_{p_i})]+\mathbb{Q} [L_n(\mathfrak{h})]}{\sqrt{\frac{1}{4}\log n}} 
      \xrightarrow[n \to \infty]{\mathrm{d}}
	   \mathcal{N}(0, c_1^2+\cdots+c_k^2),
    \end{equation}
    with $\mathbb{Q}[\cdot]=\cdot-\mathbb{E}[\cdot]$. This follows directly from Theorem \ref{th:replacement} for $k_1=0$, $k_2=k$, $k_3=1$, and choosing  
\begin{align}\label{eq_FF}
F_n(x_{1}, \cdots,x_{k+1})=\prod_{i=1}^{k+1} \exp\Big(\ii t_i \frac{x_i}{\sqrt{\frac{1}{4}\log n}}\Big), \qquad t_i \in \R,
\end{align}
with $x_i=\mathbb{Q} [L_n(h_{p_i})]$ for $1\leq i\leq k$ and $x_{k+1}=\mathbb{Q} [L_n(\mathfrak{h})]$. Setting $t_i=c_i t$ for $1\leq i\leq k$ and $t_{k+1}=t$, we have, 
\begin{align}
\E[F_n(\cdots)]=&\E \Big[ \exp \Bigl( \ii t \frac{\mathbb{Q}[L_n(f)]}{\sqrt{\frac{1}{4}\log(n)}}\Bigr)\Big]=\E^{\mathrm{Gin}} \Big[ \exp \Bigl( \ii t \frac{\mathbb{Q}[L_n(f)]}{\sqrt{\frac{1}{4}\log(n)}}\Bigr)\Big]+O(n^{-\delta})\nonumber\\
\xrightarrow[n\to\infty]{}&\mathrm{e}^{-\frac{1}{2}(c_1^2+\cdots+c_k^2) t^2}, \qquad \qquad t \in \R,
\end{align}
where we also used \eqref{eq:LogSing} in the last step. This proves \eqref{eq_goal_dist} which, together with \eqref{eq_goal_mean}, yields Theorem \ref{th:LogSing}.

We next prove the universality of Corollary \ref{co:Green:CLT} and Theorem \ref{th:Green:Cov} for test functions $g_p=\log d(p,\cdot)$ with $p\in \overline{\C}$. Note that from \eqref{eq:d}, $g_{\infty}$ can be split into a sum of $h_\infty$ defined in \eqref{eq_h} and a regular test function $\mathfrak{f}$, \ie 
	\begin{align}\label{eq_split}
		g_{\infty}(z)=-\frac{1}{2}\log (1+|z|^2)=h_\infty(z)+\mathfrak{f}(z),
	\end{align}
    with $h_{\infty}(z)=-\log |z| \chi\Big(\frac{1}{|z|}\Big)$ and $\mathfrak{f} \in \mathcal{C}_0^2$  satisfying \eqref{eq_f}. 
Moreover, from \eqref{eq:split-euclid} and \eqref{eq_split},  we  have  
	\begin{align}\label{eq_split1}
		g_{w}(z)=f_w(z)+g_{\infty}(z)-C_{w}=f_w(z)+h_\infty(z)+\mathfrak{f}(z)-C_{w}, \qquad w\in \C.
	\end{align}
    Thus using the comparison principle \eqref{eq_diff_mean} and the Ginibre expectation \eqref{eq:Green:mean:var}, we have
\begin{align}\label{eq_mean_2}   \E\big[L_n(g_{p})\big]=&\E^{\mathrm{Gin}}\big[L_n(g_{p})\big]+O\big( n^{-\delta}\big)=-\frac{n}{2}+O\big( n^{-\delta}\big), \qquad p \in \overline{\C}.
\end{align}
Hence, the universal joint law as in \eqref{eq:Green:multi} for test functions $g_{p_i}$ with $k$ fixed points $w_i \in \C$, follows from Theorem \ref{th:replacement} with $k_1=k$, $k_2=k_3=1$ and choosing a proper $F_n$ as in \eqref{eq_FF}. This proves the universality of Corollary \ref{co:Green:CLT}. The universality of Theorem \ref{th:Green:Cov} (up to an error $o_n(1)$) is proved similarly using Theorem \ref{th:replacement} by choosing
\begin{align}\label{eq_FFF}
F_n(x_1)=x_1^2, \qquad  F_n(x_1,x_2)=x_1x_2, \qquad x_1=\mathbb{Q}[L_n(g_p)], \quad x_2=\mathbb{Q}[L_n(g_q)].
\end{align}

Similarly, for test functions $f_{w_i}$ defined in \eqref{eq_fw}, with $k$ fixed points $w_i \in \C$, we have  
\begin{align}\label{eq_mean_1}
\E\big[L_n(f_{w})\big]=&\E^{\mathrm{Gin}}\big[L_n(f_{w})\big]+O\big( n^{-\delta}\big)=\frac{n}{2} \log (1+|w|^2)+O\big( n^{-\delta}\big),	
\end{align}
where we used \eqref{th:replacement} and \eqref{eq:LogPot:mean} for Ginibre matrices. Then the universal joint law as in \eqref{eq:LogPot:multi} for test functions $f_{w_i}$ follows from Theorem \ref{th:replacement} with $k_1=k$, $k_2=k_3=0$ and choosing a proper function $F_n$ as in \eqref{eq_FF}. Moreover, the universality of Corollary \ref{co:LogPot:Cov} is obtained similarly using Theorem \ref{th:replacement} with a proper $F_n$ chosen as in \eqref{eq_FFF}.

We hence finish the proof of Theorem \ref{th:logpot:univ}.

\medskip

The remaining section is devoted to proving Theorem \ref{th:replacement}.

\subsection{Proof of Theorem \ref{th:replacement}}

We will follow the same proof strategy of \cite[Theorem 4.1]{chafai-garcia-zelada-xu}. For completeness, we will recall several notations, key ingredients, and repeat some routine arguments used there.

We first normalize both Girko matrices $A$ and $B$ by $1/\sqrt{n}$, \ie $\A=A/\sqrt{n}$ and $\B=B/\sqrt{n}$. It is easy to check that $M=\A \B^{-1}$. We also set $\MG={\tilde A}^{\mathrm{Gin}}({\tilde B}^{\mathrm{Gin}})^{-1}$ with ${\tilde A}^{\mathrm{Gin}}=\AG/\sqrt{n}$, ${\tilde B}^{\mathrm{Gin}}=\BG/\sqrt{n}$ being normalized Ginibre matrices. For conventional consistency, we always write $\mathrm{Gin}$ as a superscript to denote the Ginibre case.

Throughout the proof, $c,C>0$ denote small and large constants, respectively, which may change from line to line.  For integers $k,l\in\N$, with $k<l$, we write 
$\llbracket k,l \rrbracket= \{k,\dots, l\}$ and 
$\llbracket k \rrbracket=
\llbracket 1,k \rrbracket$.  
We will often use the concept of ``with very high probability'' for an $n$-dependent event, meaning that for any fixed $D>0$ the probability of the event is bigger than $1-n^{-D}$ when $n\ge n_0(D)$, which may depend on the constants $C_i~(1\leq i\leq 4)$ from Theorem \ref{th:replacement} and $D_k$ from \ref{it:condmom}. We also recall the standard notion of \emph{stochastic domination}:  
given two families of non-negative random variables
\[
X=\left(X^{(n)}(u) \,:\, n\in\N, u\in U^{(n)}\right)\quad\text{and}\quad Y=\left(Y^{(n)}(u) \,:\, n\in\N, u\in U^{(n)}\right)
\] 
indexed by $n$ (and possibly some parameter $u$  in some parameter space $U^{(n)}$), 
we say that $X$ is {\it stochastically dominated} by $Y$, for brevity $X\prec Y$ or $X= O_\prec(Y)$, if for any small $\xi$ and large $D>0$ we have 
\begin{equation}
	\label{prec}
	\sup_{u\in U^{(n)}} \mathbb{P}\left[X^{(n)}(u)>n^\xi  Y^{(n)}(u)\right]\le n^{-D}
\end{equation}
for large enough $n\geq n_0(\xi,D)$, which may depend on the constants $C_i~(1\leq i\leq 4)$ from Theorem \ref{th:replacement} and $D_k$ from \ref{it:condmom}. 
We often use the notation $\prec$ also for deterministic quantities, then~\eqref{prec} holds with probability one. 
Properties of $\prec$ can be found in, \eg \cite[Proposition 6.5]{zbMATH06780221}.
\begin{lemma}[Proposition 6.5 in \cite{zbMATH06780221}]\label{lemma_dominant}\ %
	\begin{enumerate}
		\item $X \prec Y$ and $Y \prec Z$ imply $X \prec Z$;
		\item If $X_1 \prec Y_1$ and $X_2 \prec Y_2$, then $X_1+X_2 \prec Y_1+Y_2$ and $X_1X_2 \prec Y_1Y_2;$
		\item If $X \prec Y$, $\E Y \geq N^{-c_1}$ and $|X| \leq N^{c_2}$ almost surely with some fixed exponents $c_1$, $c_2>0$, then we have $\E X \prec \E Y$.
	\end{enumerate}
\end{lemma} 

\medskip

Below are two key ingredients from \cite{chafai-garcia-zelada-xu} and we summarize them in two lemmas. The first one is the local laws for the resolvent of the Hermitise matrix of $\A-z\B$, which is stated in \cite{chafai-garcia-zelada-xu} and essentially follows from \cite{zbMATH06347297} (see also \cite{zbMATH06221300,zbMATH06261246,zbMATH06330939}).
\begin{lemma}[Local laws]\label{lemma_local_law}
Define the Hermitise matrix of $X^z:=\A-z\B$ and its resolvent by
\begin{align}\label{eq_hermitise}
    H^{z}:=\begin{pmatrix}
		0  &   \A-z\B \\
		(\A-z\B)^*    & 0
	\end{pmatrix}  \in \C^{2n \times 2n}, \qquad G^{z}(w):=(H^{z}-w)^{-1}, \qquad \quad w \in \mathbb{C} \setminus \mathbb{R}.
\end{align}
Fix any small $\xi>0$. Then the following
	\begin{align}\label{eq:local_law_G_z}
		\Big| \big\langle \mathbf{x}, \big(G^z(w)-m^z(w)\big) \mathbf{y} \big\rangle \Big|  \prec \sqrt{
			\frac{1}{n \eta}}+\frac{1}{n\eta}, \qquad \Big|\frac{1}{2n}  \mathrm{Tr}G^{z}(w)-m^z(w)\Big|  \prec \frac{1}{n\eta},
	\end{align}
	hold uniformly for any $z \in \C$, $|w|\leq 10$ and $\eta=|\Im w| >0 $ and for any deterministic unit vectors $\mathbf{x}, \mathbf{y}\in \mathbb{C}^{2n}$, with the deterministic function $m^{z}$ given by
	\begin{align}\label{rho}
		m^{z}(w):=\frac{1}{\sqrt{1+|z|^2}} m_{sc}\Big( \frac{w}{\sqrt{1+|z|^2}}\Big), \qquad w\in \C \setminus \R,
	\end{align}
     where $m_{sc}$ is the Cauchy\,--\,Stieltjes transform of the semicircle distribution $\rho_{sc}(x)=\frac{1}{2 \pi} \sqrt{(4-x^2)_+}$. In particular, $|m^{z}|$ is uniformly bounded by one for any $z\in \C$.
     \end{lemma}

     \begin{proof}[Proof of Lemma \ref{lemma_local_law}] The local law was proved in \cite[Eq. (4.17)-(4.18)]{chafai-garcia-zelada-xu} for any $|z|\leq 1$ and $\eta \geq n^{-1+\xi}$ with $\xi>0$ being any fixed small constant. We next extend these results in two directions: 1) to the complementary $\eta$-regime $0 <\eta \leq n^{-1}$; 2) to the complementary $z$-regime $|z|>1$.
     \begin{itemize}
     \item[(1)] \textbf{Extension to $0< \eta \leq  n^{-1}$}. Fix any small constant $\xi>0$ and set $\eta_1=n^{-1+\xi}$. Using $G^2=-\ii \frac{\dd}{\dd \eta} G(\ii \eta)$ and that $\Big|\frac{\dd}{\dd \eta}m^z(\ii \eta)\Big|$ is uniformly bounded by some constant from (\ref{rho}), we have
     \begin{align}
		&\Big|\big\langle \mathbf{x}, \big(G^z(\ii \eta_1)-m^z(\ii \eta_1)\big)\mathbf{y} \big\rangle-\big\langle \mathbf{x}, \big(G^z(\ii \eta)-m^z(\ii \eta)\big)\mathbf{y} \big\rangle\Big|=\Big|\int_{\eta}^{\eta_1} \big\langle \mathbf{x}, \big(G^z(\ii s)\big)^2\mathbf{y} \big\rangle \dd s \Big|+O(\eta_1)\\
		& \qquad\qquad \qquad \qquad \qquad\qquad \qquad  \leq  \int_{\eta}^{\eta_1} \frac{1}{\eta}\sqrt{\big\langle \mathbf{x}, \Im G^z(\ii s) \mathbf{x} \big\rangle\big\langle \mathbf{y}, \Im G^z(\ii s) \mathbf{y} \big\rangle}   \dd s+O(\eta_1)\\
        &\qquad\qquad \qquad \qquad \qquad \qquad \qquad\leq  \eta_1 \sqrt{\big\langle \mathbf{x}, \Im G^z(\ii \eta_1) \mathbf{x} \big\rangle\big\langle \mathbf{y}, \Im G^z(\ii \eta_1) \mathbf{y} \big\rangle} \int_{\eta}^{\eta_1} \frac{1}{s^2} \dd s+O(\eta_1)\\
        &\qquad\qquad \qquad \qquad \qquad \qquad \qquad =O_\prec\Big( \frac{n^{\xi}}{n\eta}\Big)
	\end{align}
    where in the second line we used the Cauchy-Schwarz inequality and the Ward identity $GG^*(\ii s)=s^{-1}\Im G(\ii s)$, in the third line we used that the function $\eta \big\langle \mathbf{x}, \Im G^z(\ii \eta) \mathbf{x} \big\rangle$ is increasing in $\eta>0$, and in the last line we used that $\big\langle \mathbf{x}, \Im G^z(\ii \eta_1) \mathbf{x} \big\rangle \prec 1$ from \eqref{eq:local_law_G_z} and \eqref{rho} for $\eta_1=n^{-1+\xi}$ with $\xi>0$.  Since $\xi>0$ is chosen arbitrary small,  we have extended the first local law in \eqref{eq:local_law_G_z} for $\eta\geq n^{-1+\xi}$ to the regime $0< \eta \leq n^{-1}$, using the definition of $\prec$ in \eqref{prec}. One can extend the second local law  in (\ref{eq:local_law_G_z}) to the regime $0< \eta \leq n^{-1}$ using a similar argument. 
    
      \item[(2)] \textbf{Extension to $|z|> 1$}.
     Notice that from \eqref{eq_hermitise}, the matrix $X^{z}\overset{\mathrm{d}}{=}-zX^{z^{-1}}$ and thus 
     \begin{align}
     H^{z}\overset{\mathrm{d}}{=} - Z H^{z^{-1}}, \qquad Z:=\begin{pmatrix}
		z  &   0 \\
		0    & z^*
	\end{pmatrix}, \qquad G^{z}(w)\overset{\mathrm{d}}{=}-Z^{-1} \Big(H^{z^{-1}}+wZ^{-1}\Big)^{-1}.
    \end{align}   
    By the Schur complement formula, we have
    \begin{align}\label{eq_G_inverse}
    G^{z}(w)=\begin{pmatrix}
		\frac{1}{|z|} \big[G^{z^{-1}}\big(\frac{w}{|z|}\big)\big]_{11}  & -\frac{1}{z}  \big[G^{z^{-1}}\big(\frac{w}{|z|}\big)\big]_{12} \\
		-\frac{1}{z^*}\big[G^{z^{-1}}\big(\frac{w}{|z|}\big)\big]_{21}    & \frac{1}{|z|} \big[G^{z^{-1}}\big(\frac{w}{|z|}\big)\big]_{22}
	\end{pmatrix}, \qquad \quad G=:\begin{pmatrix}
		[G]_{11}  & [G]_{12} \\
		[G]_{21}    & [G]_{22}
	\end{pmatrix}.
    \end{align}
    Hence the local law for $G^{z}$ with $|z|>1$ follows from the local law for $G^{z^{-1}}$ proved in \cite{chafai-garcia-zelada-xu}. 
     
     \end{itemize}

       \end{proof}  
     
     The second one is the tail bound of the smallest singular value of $\A-z\B$ stated in \cite[Eq. (4.24)]{chafai-garcia-zelada-xu}, which directly follows from \cite[Lemma 3.4]{Davies} (see also \cite{invert,wegner}).
\begin{lemma}[Smallest singular value]\label{lemma_tail}
    There exists a constant $C$ depending on $D_0$ from \ref{it:condensity} such that the smallest singular value of $\A-z\B$, denoted by $\lambda_{\min}(\A-z\B)$,  satisfies
	\begin{align}\label{eq:tail_estimate}
		\mathbb{P} \Big(\lambda_{\min}(\A-z\B) \leq \eta \Big) \leq C n^2 \eta^2, \qquad \forall \eta>0,\quad z\in \C.
	\end{align}
\end{lemma}
\begin{remark}[Extension to the complementary regime]
   Though the above result is stated in \cite{chafai-garcia-zelada-xu} for all $|z|\leq 1$, it can be easily extended to the complementary regime by using that  $\lambda_{\min}(\A-z\B) \overset{\mathrm{d}}{=} |z|\lambda_{\min}(\A-z^{-1}\B) $ and thus $\mathbb{P} \big(\lambda_{\min}(\A-z\B) \leq \eta \big) = \mathbb{P} \big(\lambda_{\min} (\A-z^{-1}\B) \leq \eta/|z| \big) \leq C n^2 \eta^2$ for all $|z|\geq 1$. 
\end{remark}

Now we are ready to prove Theorem \ref{th:replacement}.

\begin{proof}[Proof of Theorem \ref{th:replacement}]
Using that $\det(\A \B^{-1}-w)=\det(\A-w\B)/\det(\B)$ and \cite{zbMATH03901742} Hermitization, we have
\begin{align}
     L_n(f_{w})=&\log |\det (\A\B^{-1}-w)|=\log |\det (\A-w\B)|-\log |\det (\B)| =\frac{1}{2}\log |\det (H^w)|-\frac{1}{2} \log |\det (H)|,
\end{align}
where  $H^w~(w\in \C)$ and $H$ are $2n \times 2n$ matrices defined by
\begin{align}\label{eq_H}
	H^w=\begin{pmatrix}
		0  &   X^w \\
		(X^w)^*    & 0
	\end{pmatrix}, \qquad 
	H=\begin{pmatrix}
		0  &   X \\
		(X)^*    & 0
	\end{pmatrix},\qquad  X^{w}=\A-w \B, \qquad X=\B.
\end{align}

For any fixed $w \in \C$, by spectral decompositions of $H^{(w)}$ and using that $\partial_\eta \log(\lambda^2+\eta^2)=\frac{2\eta}{\lambda^2+\eta^2}$, we have
\begin{align}\label{eq_1}
     L_n(f_{w})=& -\frac{1}{2}\int_{0}^{T} \Im \mathrm{Tr} \big[G^w(\ii \eta)-G(\ii \eta)\big] \mathrm{d} \eta+\frac{1}{2} \Big(\log |\det (H^w-\ii T)|-\log |\det (H-\ii T)|\Big),
\end{align}
where we choose $T=n^{100}$, $G^{w}$ and $G$ are the resolvents of $H^{w}$ and $H$ given by
\begin{align}\label{eq_resolvent}
G^{w}(\zeta)=(H^{w}-\zeta)^{-1},  \qquad G(\zeta)=(H-\zeta)^{-1}, \qquad \quad \zeta \in \mathbb{C} \setminus \mathbb{R}.
\end{align}
With a slight abuse of notation, we often use $H^{(w)}, G^{(w)}, X^{(w)}$ to denote either the matrix $H^{w}, G^{w}, X^{w}$ or the matrix $H, G, X$, respectively.
In particular, 
$$X\overset{\mathrm{d}}{=}X^{w=0}, \qquad H\overset{\mathrm{d}}{=}H^{w=0}, \qquad G\overset{\mathrm{d}}{=}G^{w=0}.$$ 
Note that the last term in \eqref{eq_1} almost vanishes, because 
	\begin{align}
		\log |\det (H^{(w)}-\ii T)|-2n \log T
        = \frac{1}{2}\sum_{\lambda \in \mathrm{Spec}(H^{(w)})} \log \left( 1+\Big(\frac{\lambda}{T}\Big)^2\right) \leq \frac{\mathrm{Tr} (H^{(w)})^2}{T^2} \prec n^{-100},
	\end{align}
using $\log(1+|x|)\leq |x|$ and $(H^{(w)})_{ij} \prec n^{-1/2}$ from the moment condition \ref{it:condmom}. Hence we obtain from \eqref{eq_1} that
\begin{align}\label{eq_det}
    L_n(f_{w}) = -\frac{1}{2}\int_{0}^{T} \Im \mathrm{Tr} \big[ G^w(\ii \eta)-G(\ii \eta)\big] \mathrm{d} \eta+O_\prec(n^{-100}). 
\end{align}
 Moreover, for any test function $f_n \in \mathcal{C}_0^2$ satisfying \eqref{eq_f}, using $\Delta_z \log |z|=2\pi\delta_0$ and integration by parts, we have 
\begin{align}
    L_n(f_{n})=&\frac{1}{2\pi}\int \Delta_z f_n(z) \log |\det (\A\B^{-1}-z)| \dd^2 z\\
    =&\frac{1}{2\pi}\int \Delta_z f_n(z) \log |\det (\A-z\B)|\dd^2 z\\
    =&-\frac{1}{4\pi} \int \Delta_z f_n(z) \int_{0}^{T} \Im \mathrm{Tr} G^z(\ii \eta) \dd \eta \dd^2 z+O_\prec(n^{-100}),\label{eq_int_f}
\end{align}
with $T=n^{100}$ and $G^z$ defined in \eqref{eq_resolvent}, where we used $\log |\det (\A\B^{-1}-z)|=\log |\det (\A-z\B)|-\log |\det \B|$ and  $\Delta_z \log |\det(\wt B)|=0$ in the second line, and in the last step we also used $\partial_\eta \log(\lambda^2+\eta^2)=\frac{2\eta}{\lambda^2+\eta^2}$.
Here $\Delta_z=4\partial_z \partial_{\bar z}=\partial^2_{\Re z}+\partial^2_{\Im z}$ denotes the two-dimensional Laplace operator on $\C$, and $\dd^2 z= \frac{1}{2}\ii(\dd z\wedge \dd \overline{z})=\dd \Re z \dd \Im z$ denotes the two dimensional area form on $\C$.

\medskip

To regularize the right side of both \eqref{eq_det} and \eqref{eq_int_f}, we introduce the following short-hand notations:
\begin{align}
\wt{L_n(f_w)}:=&-\frac{1}{2}\int_{\eta_0}^{T} \Im \mathrm{Tr} \big[ G^{w}(\ii \eta)-G(\ii \eta)\big] \mathrm{d} \eta, \qquad \qquad w\in \C, \label{eq_regular1}\\
\wt{L_n(f_n)}:=&-\frac{1}{4\pi} \int \Delta_z f_n(z) \int_{\eta_0}^{T} \Im \mathrm{Tr} G^{z}(\ii \eta) \dd \eta \dd^2 z, \qquad f_n\in \mathcal{C}_0^2 \mbox{~satisfying~} \eqref{eq_f}.\label{eq_regular2}
\end{align}
Combining $L^{1}$-norm bound of $\Delta_z f_n$ in (\ref{eq_f}) with the following lemma, whose proof is presented in the next subsection, we obtain that, for any fixed $k \in \N$,
\begin{align}\label{eq_lemma_tiny}
	\E \Big| L_n(f_w)-\wt{L_n(f_w)}\Big|^k=O\big( n^{-\epsilon/3}\big), \qquad \E \Big| L_n(f_n)-\wt{L_n(f_n)}\Big|^k=O\big( n^{-\epsilon/3}\big).
\end{align}
\begin{lemma}\label{lemma_tiny}
Fix any small $\epsilon>0$ and set $\eta_0=n^{-1-\epsilon}$. Then the following holds uniformly for any fixed $k \in \N$
\begin{align}\label{eq_small}
	\sup_{z\in \C}\mathbb{E} \Big| \int_{0}^{\eta_0}\Im \mathrm{Tr} G^z (\ii \eta) \mathrm{d} \eta \Big|^k=O\big( n^{-\epsilon/2}\big).
\end{align}
 \end{lemma}

We next regularize $L_n(h_p)$ for test function $h_{p}$ with $p\in \overline{\C}$. Recall the definition of $h_p$ in \eqref{eq_h}.  
It is clear that $h_p$ does not satisfy the $L^{1}$-norm condition (\ref{eq_f}). 
Then we define the regularized test function of $h_p$ by 
\begin{align}\label{eq_fun_reg}
\wt{h_p} (\cdot):=
\begin{cases}
\frac{1}{2}\log \big(|\cdot-p|^2+n^{-2L}\big) \chi(|\cdot-p|), \qquad & p\in \C\\
\wt{h_0}\Big(\frac{1}{\cdot}\Big)=\frac{1}{2}\log \Big(\frac{1}{|~\cdot~|^2}+n^{-2L}\Big) \chi\Big(\frac{1}{|~\cdot~|}\Big), \qquad &p=\infty
\end{cases},
\end{align}
for a sufficiently large $L>0$ to be fixed later, where $\chi:\R \rightarrow \R^+$ is a smooth and non-increasing cut-off function near zero. 
The follow result claims that, for any $p\in \overline{\C}$, $L_n(\wt{h_p})$ is approximately $L_n(h_p)$, whose proof is postponed till the next subsection.  
\begin{lemma}\label{lemma_LL}
Fix $p\in \overline{\C}$ and $K_0 \in \N$. Then there exists a large $L>0$ depending on $K_0$ such that the following holds true 
\begin{align}
\E \left| L_n(h_p)-L_n(\wt{h_p})\right|^k=O(n^{-100k}),
\end{align}
for any  $1\leq k\leq K_0$. 
\end{lemma}

By a direct computation, the regularized test function $\wt{h_p}~(p \in \C)$ defined in \eqref{eq_fun_reg} satisfies the condition (\ref{eq_f}), \ie 
\begin{align}
	\int|\Delta_z \wt{h_p}(z)| \dd^2 z
	=&\int \left| \Delta_z \Big[\frac{1}{2}\log \big(|z|^2+n^{-2L}\big) \chi(|z|)\Big] \right| \dd^2 z\\
	=& \int \frac{2 n^{-2L}}{(|z|^2+n^{-2L})^2} \chi(|z|)\dd^2 z +O(\log n)=O(\log n).\label{eq_L1}
\end{align}
The same also applies to $\wt{h_\infty}$ in \eqref{eq_fun_reg}, since the 
$L^{1}$-norm of the Laplacian is invariant under conformal composition including $z \rightarrow 1/z$. Similarly to \eqref{eq_regular2}, we introduce the following short-hand to regularize $L_n(h_{p})$:
\begin{align}\label{eq_regular3}
\wt{L_n(h_{p})}:=-\frac{1}{4\pi} \int \Delta_z \wt{h_{p}}(z) \int_{\eta_0}^{T} \Im \mathrm{Tr} G^{z}(\ii \eta) \dd \eta \dd^2 z,
\end{align}
with $\wt{h_{p}}$ defined in (\ref{eq_fun_reg}), and $\eta_0=n^{-1-\epsilon}$. Using \eqref{eq_L1}, Lemma \ref{lemma_tiny}, and Lemma \ref{lemma_LL}, we have for any fixed $k\in \N$
\begin{align}\label{eq_regular4}
\E \Big| L_n(h_p)-\wt{L_n(h_{p})}\Big|^k=O\big( n^{-\epsilon/3}\big).
\end{align}

Hence for any fixed $w_i\in \C~(1\leq i \leq k_1)$, $p_i\in \overline{\C}~(1\leq i \leq k_2)$ and regular test functions $f^{(i)}_n\in \mathcal{C}_0^2~(1\leq i \leq k_3) $ satisfying \eqref{eq_f},  we now readily use \eqref{eq_regular1}, \eqref{eq_regular2} and \eqref{eq_regular3} to regularize $\mathcal{X}_n$ in \eqref{XX} and define 
\begin{align}\label{eq_XX}	\wt{\mathcal{X}_n}:=F_n\left(\mathbb{Q}\big[\wt{L_n(f_{w_1})}\big], \cdots , \mathbb{Q}\big[\wt{L_n(f_{w_{k_1}})}\big], \mathbb{Q}\big[\wt{L_n(h_{p_1)}}\big], \cdots , \mathbb{Q}\big[\wt{L_n(h_{p_{k_2}})}\big],\mathbb{Q}\big[\wt{L_n(f_n^{(1)})}\big],\cdots, \mathbb{Q}\big[\wt{L_n(f_n^{(k_3)})}\big] \right).
\end{align}

Recall the local laws for $G^z~(z\in \C)$ from Lemma \ref{lemma_local_law}, which also applies to $G$. Then for any fixed $w\in \C$, 
\begin{align}\label{eq_bound1}
	\left|\mathbb{Q}\big[\wt{L_n(f_{w})}\big]\right| 
    \leq \int_{\eta_0}^{T} \left| (1-\E) \mathrm{Tr}  G^{w}(\ii \eta)\right| \mathrm{d} \eta+ \int_{\eta_0}^{T} \left| (1-\E) \mathrm{Tr} G(\ii \eta) \right| \mathrm{d} \eta
    =O_\prec\left(\int_{\eta_0}^{T} \frac{1}{\eta} \dd \eta \right)=O_{\prec}(\log n),
\end{align}
where we also used the third property of stochastic domination $\prec$ from Lemma \ref{lemma_dominant}. Note that $\max_{i,j} |G_{ij}(\ii \eta)|  \leq \|G(\ii \eta)\|_{\mathrm{op}} \leq \frac{1}{\eta}$ almost surely, so the condition in (3) of Lemma \ref{lemma_dominant} is satisfied. We remark that in the remaining part of this section, we will apply Lemma \ref{lemma_dominant} without specific mentioning.  Combining with (\ref{eq_f}) and (\ref{eq_L1}), we also have
\begin{align}\label{eq_bound2}
	\Big|\mathbb{Q}\big[\wt{L_n(h_{p})}\big]\Big|=O_\prec \Big((\log n)^C\Big), \qquad \Big|\mathbb{Q}\big[\wt{L_n(f_{n})}\big]\Big| =O_\prec \Big((\log n)^{C'}\Big),
\end{align}
for some constants $C,C'>0$.
Hence by a simple Taylor expansion using \eqref{eq_F}, \eqref{eq_bound1}, \eqref{eq_bound2}, together with  \eqref{eq_lemma_tiny} and \eqref{eq_regular4},  we obtain the following.
\begin{lemma}\label{lemma_regular}
Fix any small $\epsilon>0$ and set $\eta_0=n^{-1-\epsilon}$. Then the following holds
\begin{align}\label{eq_tiny}
\Big|\mathbb{E} \big[ \mathcal{X}_n \big] -\mathbb{E} \big[ \wt{\mathcal{X}}_n \big]  \Big|=O\big( n^{-\epsilon/4}\big).
\end{align}
Moreover, for any fixed $w\in \C$, any fixed $p\in \overline{\C}$, and any test function $f_n\in \mathcal{C}_0^2$ satisfying \eqref{eq_f}, 
    \begin{align}
    &\Big|\mathbb{E} \big[ {L_n(f)} \big] -\mathbb{E} \big[ \wt{L_n(f)}  \big]  \Big|=O\big( n^{-\epsilon/4}\big), \qquad f \in \big\{ f_{w}, h_{p}, f_{n}\big\},\label{eq_tiny2}
    \end{align}
    with $\wt{L_n(f_w)}$, $\wt{L_n(h_p)}$, $\wt{L_n(f_n)}$  defined in \eqref{eq_regular1}, \eqref{eq_regular2} and \eqref{eq_regular3}.
\end{lemma}

We remark that Lemma \ref{lemma_regular} also holds true if we replace $M=AB^{-1}$ with $\MG=\AG(\BG)^{-1}$, since the Ginibre matrix is also a Girko matrix. We next introduce the following key technical lemma for comparison.

 \begin{lemma}\label{lemma_GFT}
Choose $\epsilon<c_0/100$ with $c_0>0$ the small constant from \ref{it:condmom4} and let $\eta_0=n^{-1-\epsilon}$.  Then 
	\begin{align}\label{eq_GFT}
		\Big|\mathbb{E} \big[ \wt{\mathcal{X}}_n \big] -\mathbb{E}^{\mathrm{Gin}} \big[ \wt{\mathcal{X}_n} \big]  \Big|=O(n^{-c_0/4}),
	\end{align} 
    where we use $\E^{\mathrm{Gin}}$ to denote the corresponding expectation with $M$ replaced by $M^{\mathrm{Gin}}$.

Moreover, for any fixed $w\in \C$, any fixed $p\in \overline{\C}$, and any test function $f_n\in \mathcal{C}_0^2$ satisfying \eqref{eq_f}, we  have
    \begin{align}
    &\Big|\mathbb{E} \big[ \wt{L_n(f)} \big] -\mathbb{E}^{\mathrm{Gin}} \big[ \wt{L_n(f)} \big]  \Big|=O(n^{-c_0/3}), \qquad f \in \big\{ f_{w}, h_{p}, f_{n}\big\},\label{eq_GFT2}
    \end{align}
      with $\wt{L_n(f_w)}$, $\wt{L_n(h_p)}$, $\wt{L_n(f_n)}$  defined in \eqref{eq_regular1}, \eqref{eq_regular2} and \eqref{eq_regular3}.
\end{lemma}

The proof of Lemma \ref{lemma_GFT} is postponed till the end of this section. Then the first statement \eqref{eq_diff} follows directly from combining \eqref{eq_GFT} with \eqref{eq_tiny} for both $M$ and $M^{\mathrm{Gin}}$. Similarly the second statement in \eqref{eq_diff_mean} is obtained by combining \eqref{eq_tiny2} with \eqref{eq_GFT2}.  We hence finish the proof of Theorem \ref{th:replacement}.
 \end{proof}

\subsection{Proof of several lemmas}

In this section, we collect the proofs of several lemmas used in the previous section. 

\begin{proof}[Proof of Lemma \ref{lemma_tiny}]
	The case for $k=1$ was proved in \cite[Lemma 4.2]{chafai-garcia-zelada-xu}, following the proof of \cite[Lemma 4]{zbMATH07329432}. With a minor modification, the proof can be extended to general $k\geq 1$. For completeness, we present the details below. The proof crucially relies on the local laws in Lemma \ref{lemma_local_law} and the tail bound for the smallest singular value in Lemma \ref{lemma_tail}.

	Note that the spectrum of $H^{z}$ defined in \eqref{eq_H} is symmetrically distributed. By a simple spectral decomposition of $H^{z}$ , we have
	\begin{align}
		I:=\int_{0}^{\eta_0}\Im \mathrm{Tr} G^z (\ii \eta) \mathrm{d} \eta=2  \int_{0}^{\eta_0} \sum_{j=1}^{n} \frac{\eta}{(\lambda_j^z)^2+\eta^2} \dd \eta, \qquad \eta_0=n^{-1-\epsilon},
	\end{align}
	where $\big(\lambda^{z}_{j}\big)_{j=1}^n$ are the non-negative eigenvalues of $H^z$ which are ordered non-decreasingly.
	Let $L>0$ be a large number to be fixed.  Set $\eta_L=n^{-L}$ and $\eta_2=n^{-1-\epsilon/2}$. Then we split the above sum of eigenvalues into three parts:
	\begin{align}\label{eq_part}
		I=& \Big(\sum_{\lambda^z_j \geq \eta_2}+\sum_{\eta_L\leq \lambda^z_j \leq \eta_2} +\sum_{\lambda^z_j\leq \eta_L}\Big)  \int_{0}^{\eta_0}  \frac{2\eta}{(\lambda_j^z)^2+\eta^2} \dd \eta=:I_1+I_2+I_3.
	\end{align}
	For the first part of \eqref{eq_part}, we have
	\begin{align}
		I_1\leq & 2 \int_{0}^{\eta_0} \sum_{\lambda_j\geq \eta_2} \frac{\eta}{\lambda_j^2+\eta_2^2} \dd \eta \leq   \sum_{j=1}^{n} \frac{\eta_0^2}{(\lambda^z_j)^2+\eta_2^2}   = \frac{\eta_0^2}{2\eta_2} \Im\mathrm{Tr}G^{z}(\ii \eta_2)\leq \frac{\eta_0^2}{\eta_2} \frac{\eta_3}{\eta_2} \Im\mathrm{Tr}G^{z}(\ii \eta_3),
	\end{align}
	where in the last step we used that the function $\eta\Im\mathrm{Tr}G^{z}(\ii \eta)$ is non-decreasing in $\eta>0$, and chose 
	$\eta_3=n^{-1+\xi}$ for a small $0<\xi<\epsilon/2$.  Applying the local law in \eqref{eq:local_law_G_z} for $G^{z}(\ii \eta_3)$, the third property of Lemma \ref{lemma_dominant}, and that $m^z$ from \eqref{rho} is uniformly bounded, we obtain that 
	\begin{align}\label{eq_first}
		\mathbb{E} \big[|I_1|^k\big]=O\big( n^{-k\epsilon/2}\big).
	\end{align}

	We next estimate the second and  third part of \eqref{eq_part}.  Using that $\partial_\eta \log(\lambda^2+\eta^2)=\frac{2\eta}{\lambda^2+\eta^2}$, we have
	$$I_2+I_3=\Big(\sum_{\eta_L\leq \lambda^z_j \leq \eta_2} +\sum_{\lambda^z_j\leq \eta_L}\Big)  \int_{0}^{\eta_0}  \frac{2\eta}{(\lambda_j^z)^2+\eta^2} \dd \eta =\Big(\sum_{\eta_L\leq \lambda^z_j \leq \eta_2} +\sum_{\lambda^z_j\leq \eta_L}\Big) \log \Big( 1+\frac{\eta_0^2}{(\lambda^z_j)^2}\Big).$$
	Note that for $\eta_2=n^{-1-\epsilon/2}$ and $\eta_3=n^{-1+\xi}$ for any small $\xi>0$
	$$ \frac{\eta_3}{\eta_2} \Im\mathrm{Tr}G^{z}(\ii \eta_3) \geq \Im\mathrm{Tr}G^{z}(\ii \eta_2)=2\sum_{j=1}^{n} \frac{\eta_2}{(\lambda^z_j)^2+\eta_2^2}  \geq \frac{2}{\eta_2} \# \big\{ 1\leq j\leq n: \lambda_j^z \leq \eta_2 \big\}.$$
	Hence using the local law in \eqref{eq:local_law_G_z}, by the definition of stochastic domination in \eqref{prec}, we have
	$$\# \big\{ 1\leq j\leq n: \lambda_j^z \leq \eta_2 \big\} \prec 1.$$
	Note that $\big(\lambda^{z}_{ i}\big)_{i=1}^n$ coincide with singular values of $X^z=\wt A-z\wt B$ defined in (\ref{eq_H}) arranged non-decreasingly. Hence
	\begin{align}
		I_2= \sum_{\eta_L \leq  \lambda^z_j \leq \eta_2} \log \Big( 1+\frac{\eta_0^2}{(\lambda^z_j)^2}\Big) \prec (C_L \log n) \mathbf{1}_{ \lambda^z_1 \leq \eta_2}.
	\end{align}
	Using Lemma \ref{lemma_dominant} and the tail bound for the smallest singular value in \eqref{eq:tail_estimate}, we have
	\begin{align}\label{eq_middle}
		\E\big[ |I_2|^k \big] \prec  (C_L \log n)^k \P\big( \lambda^z_1 \leq \eta_2\big)=O\big( n^{-\epsilon/2}\big).
	\end{align}	
	The third part $I_3$ in \eqref{eq_part} can be bounded similarly. That is,
	\begin{align}
		I_3=& \sum_{ \lambda^z_j \leq \eta_L} \log \Big( 1+\frac{\eta_0^2}{(\lambda^z_j)^2}\Big) \leq  -2 \sum_{ \lambda^z_j \leq \eta_L} \log (\lambda^z_j) \leq 
		-2n \log \big(\lambda^z_1\big) \mathbf{1}_{\lambda^z_1 \leq \eta_L},
	\end{align}
	where we used that $\log ((\lambda^z_j)^2 + (\eta_0)^2) \leq 0$ 
	for $\lambda^z_j \leq \eta_L=n^{-L}$. Hence by integration by parts, we have
	\begin{align}
		\E\big[ |I_3|^k\big]\leq &  (2n)^k \mathbb{E} \Big[\Big(-\log \big(\lambda^z_1\big) \Big)^k \mathbf{1}_{\lambda^z_1 \leq \eta_L}  \Big]\\
		\leq&  (2n)^k\bigg(
		(L \log n)^k \mathbb P \big(\lambda^z_1 
		\leq n^{-L} \big) + k 
		\int_{L \log n}^{\infty} s^{k-1} \mathbb{P}\big(\lambda^z_1 \leq \mathrm{e}^{-s} \big)\dd s \bigg)=O\big( n^{-100}\big),\label{eq_third}
	\end{align}
	by choosing a sufficiently large $L>0$ depending on $k$ using \eqref{eq:tail_estimate}.
	
	Summing up \eqref{eq_first}, \eqref{eq_middle} and \eqref{eq_third}, we finish the proof of Lemma \ref{lemma_tiny}.
\end{proof}

\begin{proof}[Proof of Lemma \ref{lemma_LL}]
	We will present the proof for any fixed $p\in \C$. For the infinity point $p=\infty$, the proof is similar using that the law of $\A \B^{-1}$ is invariant under inversions. From now on, we will use the letter $w \in \C$ to replace $p$. 
	
	For notational simplicity, we use $(\sigma^w_i=\sigma_i-w)_{i=1}^n$ to denote the eigenvalues of the matrix $\A \B^{-1}-w$, which are arranged by modulus in a non-decreasing order, \ie $|\sigma^w_1| \leq |\sigma^w_2| \leq \cdots \leq |\sigma^w_n|$.  For any matrix $X\in \C^{n\times n}$, we use $(\lambda_i(X))_{i=1}^n$ to denote the singular values of $X$, which are arranged  in a non-decreasing order. Then we have 
	\begin{align}
		\big|L_n(h_w)-L_n(\wt{h_w})\big| \leq &\left| \sum_{|\sigma^w_i|\leq n^{-L/2}} \log|\sigma^w_i| \chi(|\sigma^w_i|) \right|+\left| \sum_{|\sigma^w_i|\leq n^{-L/2}} \frac{1}{2}\log \big(|\sigma^w_i|^2+n^{-2L}\big) \chi(|\sigma^w_i|) \right|\\
		& \qquad + \left| \sum_{|\sigma^w_i|\geq n^{-L/2}} \log \Big( 1+\frac{n^{-2L}}{|\sigma^w_i|^2} \Big) \chi(|\sigma^w_i|)\right|=:E_1+E_2+E_3.
	\end{align}
	It is easy to check $E_2\leq E_1$ and $|E_3| \leq n^{-L+1}$ using that $\log (1+|x|)\leq |x|$. We next estimate $E_1$.  Note that the matrices $\A$, $\B$ and $\A \B^{-1}-w$ are  almost surely invertible. By the Weyl–Horn–Johnson inequality, we have $\prod_{i=1}^k |\sigma^w_i| \geq \prod_{i=1}^k \lambda_i(AB^{-1}-w)$. Hence the first term $E_1$ is bounded by
	\begin{align}
		E_1= \Big|\sum_{|\sigma^w_i|\leq n^{-L/2}} \log|\sigma^w_i| \chi(|\sigma^w_i|)\Big| \leq \sum_{i=1}^n \big|\log(\lambda_i(\A \B^{-1}-w))\big| \mathbf{1}_{|\sigma^w_i| \leq n^{-L/2}}.
	\end{align}
	By the min-max theorem for singular values, we have
	\begin{align}
		\lambda_{i}(\A \B^{-1}-w) \geq \lambda_{i}(\A-w \B)\lambda_{n}(\B^{-1})= \|\B\|^{-1} \lambda_i(\A-w \B) \geq \|\B\|^{-1}\lambda_1(\A-w \B),
	\end{align}
	with $\|\B\|=\lambda_{n}(\B)$. Note that $\|\B\|_{\max} \leq \|\B\| \leq n \|\B\|_{\max}$ and $\B=B/\sqrt{n}$. From the moment condition \ref{it:condmom} on the matrix $B$ and by the Markov inequality, for any fixed large $D>0$, we have
    \begin{align}\label{eq_largest}
    \P\Big( \|\B\| \geq n \Big) \leq  \P\Big( \|\B\|_{\max} \geq n \Big) \leq n^2 \P\Big( |B_{11}| \geq n^{1/2} \Big) \leq n^2 \frac{\E|B_{11}|^{D}}{n^{D/2}} \leq C_D n^{-D/4},
    \end{align}
    for sufficiently large $n$. On the event $\Omega:=\{\|\B\| \leq n\}$, using $\lambda_{i}(\A \B^{-1}-w) \geq n^{-1} \lambda_1(\A-w \B) $  and $\lambda_1(\A-w \B) \leq n\lambda_{1}(\A \B^{-1}-w) \leq n|\sigma^w_1| \leq n^{-L/2+1}$, we have
	\begin{align}\label{eq_E1}
		E_1 \leq n \Big( \log n+\big|\log(\lambda_1(\A-w \B))\big| \Big)\mathbf{1}_{\lambda_1(\A-w\B) \leq n^{-L/2+1}}.
	\end{align}
    We next show that the contribution from the event $\Omega^c$ is negligible. Note that from (\ref{eq_det}), \eqref{eq_small} and the local law in (\ref{eq:local_law_G_z}), we obtain that, for any fixed $k\in \N$,
    $$\E \big[ |L_n(h_{w})|^{k}\big]=  n^k \Big(\int_{0}^{T} \big(\Im m^w(\ii \eta)-\Im m^{w=0}(\ii \eta)\big) \mathrm{d} \eta \Big)^k + O_\prec\big((\log n)^k\big)=O\big( (n \log n)^k\big),$$
    where we also used that $m^{z}$ defined in (\ref{rho}) satisfies $|m^{w}(\ii \eta)|=O(\eta^{-1})$. Hence by the Cauchy-Schwarz inequality and using \eqref{eq_largest}, we have 
    $$\E\Big[\big|L_n(h_w)|^{k} \mathbf{1}_{\Omega^c}\Big] \leq C_k (n \log n)^{k} \P \big( \|\B\| \geq n \big) \leq C_D (\log n)^k n^{k-D/4}.$$
    By choosing $D$ sufficiently large depending on $K_0$, the contribution to $\E |L_n(h_{w})|^{k}~(1\leq k\leq K_0)$ from the event $\Omega^c$ is negligible. The same also applies to $\E |L_n(\wt{h_{w}})|^{k}$. Hence using \eqref{eq_E1}, and that $E_2\leq E_1$, $|E_3| \leq n^{-L+1}$, we conclude
	\begin{align}
		\E\big|L_n(h_w)-L_n(\wt{h_w})\big|^{k}  \leq (2 n)^k \E \left[ \Big( \log n+\big|\log(\lambda_1(\A-w \B))\big| \Big)^k \mathbf{1}_{\lambda_1(\A-w\B) \leq n^{-L/2+1}} \right]+O(n^{-L+1}),
	\end{align}
    for any $k\leq K_0$. Thus using Lemma \ref{lemma_tail} we have
	\begin{align}
		\E\big|L_n(h_w)-L_n(\wt{h_w})\big|^{k} \leq & C_k n^k \E\left[ \Big( (\log n)^k+\big|\log(\lambda_1(\A-w \B))\big|^k \Big) \mathbf{1}_{\lambda_1(\A-w\B) \leq n^{-L/2+1}} \right]+O(n^{-k(L-1)})\\ 
		\leq & C'_k n^k \bigg(
		(L \log n)^k \mathbb P \big(\lambda_1(\A-w \B)
		\leq   n^{-L/2+1} \big) + 
		\int_{L/4 \log n}^{\infty} k s^{k-1}\mathbb{P}\big(\lambda_1(\A-w \B) \leq  \mathrm{e}^{-s} \big)\dd s \bigg) =O(n^{-100k}),
	\end{align}
	for any $k\leq K_0$, by choosing a sufficiently large $L>0$ which depends on $K_0$ using \eqref{eq:tail_estimate}. This proves Lemma \ref{lemma_LL}.
	\end{proof}

We next prove  Lemma \ref{lemma_GFT}. The proof follows the same strategy as \cite[Lemma 4.3]{chafai-garcia-zelada-xu}, which is a special case of Lemma \ref{lemma_GFT} with $k_1=k_2=0$ and $k_3=1$.  A crucial tool of the proof is the cumulant expansion formula presented below.
\begin{lemma}[Lemma 3.1 from \cite{zbMATH06775355}]\label{cumulant_expansion_lemma}
		Let $h$ be a complex-valued random variable with finite moments. Define the $(p,q)$-cumulant of $h$ to be 
   \begin{align}\label{cumulant_pq}
	\kappa^{(p,q)}(h):=(-\ii)^{p+q} \Big( \frac{\partial^{p+q}}{\partial s^p
     \partial t^q} \log \E \mathrm{e}^{\mathrm {i} s h+\mathrm{i} t \overline{h}} \Big) \Big|_{s,t=0},
    \end{align}
    provided that the right side above exists.
    Let $f: \C \times \C \longrightarrow \C$ be a smooth function and denote its derivatives by
	$$f^{(p,q)}(z_1,z_2):=\frac{\partial^{p+q}}{\partial z^p_1 \partial z^q_2} f(z_1,z_2).$$ 
	Then for any fixed $l \in \N$, we have
	\begin{align}\label{cumulant_expansion_eq}
	\E\big[ \bar{h} f(h, \bar{h})\big]=\sum_{p+q+1=1}^l \frac{1}{p!q!} \kappa^{(p,q+1)}(h)\E\big[ f^{(p,q)}(h,\bar h)\big] +R_{l+1}\,,
	\end{align}
	provided that all the expectations in \eqref{cumulant_expansion_eq} exist. The error term satisfies
	\begin{align}\label{cumulant_error}
	|R_{l+1}| \leq C_l \E |h|^{l+1} \max_{p+q+1=l+1} \Big\{ \sup_{|z| \leq M} |f^{(p,q)}(z,\bar{z})| \Big\}+C_l \E \Big[ |h|^{l+1} 1_{|h|>M}\Big] \max_{p+q+1=l+1} \|f^{(p,q)}(z,\bar{z})\|_{\infty},
	\end{align}
	and $M>0$ is an arbitrary fixed cutoff.
\end{lemma}

\begin{proof}[Proof of Lemma \ref{lemma_GFT}]
	
Recall the definitions of $\wt{\mathcal{X}_n}$, $\wt{L_n(f_{w})}$, $\wt{L_n(f_n)}$, and $\wt{L_n(h_{p})}$ from \eqref{eq_XX}, \eqref{eq_regular1}, (\ref{eq_regular2}) and (\ref{eq_regular3}). From (\ref{eq_bound1})-(\ref{eq_bound2}), we know $\wt{h_{p}}$ with $p \in \overline{\C}$ defined in (\ref{eq_fun_reg}) are also regular test function satisfying \eqref{eq_f}. So we will not distinguish $\wt{h_{p_i}}~(1\leq i\leq k_2)$ from $f^{(i)}_n~(1\leq i\leq k_3)$, and may assume $k_2=0$ and $k=k_1+k_3$ without loss of generality. It then suffices to study	\begin{align}\label{eq_cf}
\wt{\mathcal{X}_n}=F_n\left(\mathbb{Q}\big[\wt{L_n(f_{w_1})}\big], \cdots , \mathbb{Q}\big[\wt{L_n(f_{w_{k_1}})}\big],\mathbb{Q}\big[\wt{L_n(f_n^{(1)})}\big],\cdots, \mathbb{Q}\big[\wt{L_n(f_n^{(k_3)})}\big] \right), \qquad k=k_1+k_3.
\end{align}

We introduce two independent matrix flows below interpolating between general Girko matrices and Ginibre matrices. Define the following Ornstein\,--\,Uhlenbeck matrix flows
	\begin{equation}\label{OU}
		\dd  \A_t=-\frac{1}{2} \A_t \dd t+\frac{\dd  \mathcal{B}^{(1)}_t}{\sqrt{n}}, \qquad \dd  \B_t=-\frac{1}{2} \B_t \dd  t+\frac{\dd  \mathcal{B}^{(2)}_t}{\sqrt{n}}, 
	\end{equation}
	with initial ensembles $\A_0=\A=A/\sqrt{n}$ and $\B_0=\B=B/\sqrt{n}$, where $A,B$ are independent Girko matrices satisfying \ref{it:condensity} --\ref{it:condmom}, and $\mathcal{B}^{(1)}_t$ and $\mathcal{B}^{(2)}_t$ are two independent $n\times n$ matrices whose entries are i.i.d. standard complex Brownian motions. 	Then for any $w \in \C$, we set as in \eqref{eq_H}
	\begin{align}\label{xh}
		H^w_{t}=\begin{pmatrix}
			0  &   X^w_{t} \\
			(X^w_{t})^*    & 0
		\end{pmatrix}, \qquad 
		H_{t}=\begin{pmatrix}
			0  &   X_{t} \\
			(X_{t})^*    & 0
		\end{pmatrix},\qquad  X^{w}_{t}=\A_t-w \B_t, \qquad X_{t}=\B_t,
	\end{align}
	and denote the resolvents of $H^{w}_{t}$ and $H_{t}$ by $G^{w}_{t}$ and $G_{t}$, respectively. 
	With a slight abuse of notation, we often use $G_t^{(w)}$ to denote the resolvent of either $H_t^{w}$ or $H_t$.
	We also introduce the following notations:
	\begin{align}\label{def_HAB}
		H_t^{(1)}=\begin{pmatrix}
			0  &   \A_t \\
			(\A_t)^*    & 0
		\end{pmatrix}, \qquad H_t^{(2)}=\begin{pmatrix}
			0  &   \B_t \\
			(\B_t)^*    & 0
		\end{pmatrix}, \qquad \dd\mathcal{M}^{(l)}_t=\begin{pmatrix}
			0  &   \dd\mathcal{B}^{(l)}_t \\
			\big(\dd\mathcal{B}^{(l)}_t\big)^*    & 0
		\end{pmatrix}, \qquad l=1,2.
	\end{align}
	For notational simplicity, we may often drop the dependence $t$ and write $H^{(l)}_t=(h^{(l)}_{ij})_{i,j\in 
		\llbracket 2n \rrbracket}$ for $l=1,2$.
	Note that the non-zero entries $(h^{(l)}_{aB})$ for $a\in \llbracket n \rrbracket$ and 
	$B\in \llbracket n+1,2n \rrbracket$ are independent random variables that depend on time.

	Notice that $\mathrm{Tr} G_t^{(w)}(\ii \eta)$ is purely imaginary, because the spectrum of $H_t^{(w)}$ is symmetrically distributed, from the block structure of $H_t^{(w)}$ in \eqref{xh}. Hence we have $\mathrm{Tr} G_t^{(w)}(\ii \eta)=\ii\Im \mathrm{Tr} G_t^{(w)}(\ii \eta)$. Then we define the following short-hands for time-dependent versions of $\wt{L_n(f_{w_i})}$ and $\wt{L_n(f^{(i)}_n)}$:
	\begin{align}\label{eq_LL}
		\mathcal{L}_{i,t}:=\begin{cases}
			&%
		\frac{\ii}{2}\int_{\eta_0}^{T}  \mathrm{Tr} \big[ G_t^{w_i}(\ii \eta)-G_t(\ii \eta)\big] \mathrm{d} \eta, \quad\qquad\qquad 1\leq i\leq k_1, \\
		&%
		\frac{\ii}{4\pi} \int \Delta_z f^{(i-k_1)}_n(z) \int_{\eta_0}^{T} \mathrm{Tr} G^z_t(\ii \eta) \dd \eta \dd^2 z,  \qquad k_1+1\leq i\leq k,
		\end{cases}
	\end{align}
	where $T=n^{100}$ and $\eta_0=n^{-1-\epsilon}$ with $0<\epsilon<c_0/100$, and also the time-dependent version of $\wt{\mathcal{X}_n}$:  
    \begin{align}\label{eq_XXX}
    \wt{\mathcal{X}_{n,t}}:=F_n\left(\mathcal{L}_{1,t}-\E^{\mathrm{Gin}}\big[\mathcal{L}_{1}\big],\cdots,\mathcal{L}_{k,t}-\E^{\mathrm{Gin}}\big[ \mathcal{L}_{k}\big] \right),
    \end{align}
    where $\E^{\mathrm{Gin}}[\cdot]$ denotes the corresponding  expectation with respect to Ginibre matrices. In contrast to \eqref{eq_cf}, we use $\mathcal{L}_{i,t}-\E^{\mathrm{Gin}}\big[\mathcal{L}_{i}\big]$ instead of $\mathbb{Q}\big[\mathcal{L}_{i,t}\big]=\mathcal{L}_{i,t}-\E\big[\mathcal{L}_{i,t}\big]$, which simplifies  the computations since $\E^{\mathrm{Gin}}[\cdot]$ is  time-independent. This modification is really minor for comparison, since we also note that for $t=0$,
    \begin{align}\label{eq_minor}
    \Big|\E \big[\wt{\mathcal{X}_{n}}\big]- \E \big[\wt{\mathcal{X}_{n,0}}\big]\Big|=O((\log n)^C n^{-c_0/3}),
    \end{align}
    by a simple Taylor expansion of $F_n$ using (\ref{eq_F}), assuming we have proved the second statement (\ref{eq_GFT2}), \ie for $t=0$, 
	 \begin{align}\label{eq_goal1}
	 	\Big| \E [\mathcal{L}_{i,0}]-\E^{\mathrm{Gin}} [\mathcal{L}_{i}]\Big|=O(n^{-c_0/3}), \qquad 1\leq i\leq k.
	 \end{align}

	\medskip

	 We start with proving the second statement (\ref{eq_GFT2}).  We first use the matrix flows \eqref{OU} and Itô's formula to derive the dynamics of $G^{(w)}_t$, and hence $\II~(1\leq i\leq k)$, as functions of $\big\{h^{(1)}_{aB}\big\}$ and $\big\{h^{(2)}_{aB}\big\}$ in \eqref{def_HAB} for $a\in 
	\llbracket n \rrbracket$ and $B\in \llbracket n+1,2n \rrbracket$. More precisely, 
	\begin{align}\label{eq_dynamics}
		\dd \II=& \Theta_{i,t} \dd t+  \dd M_{i,t}, \qquad 1\leq i\leq k,
	\end{align}
	with the drift function $\Theta$ given by
	\begin{align}\label{eq_drift}
		\Theta_{i,t} =   \sum_{a=1}^n \sum_{B=n+1}^{2n} \sum_{l=1}^2 \left(-\frac{1}{2}  h^{(l)}_{aB} \frac{  \partial \II}{\partial h^{(l)}_{aB}}  -\frac{1}{2} \overline{h^{(l)}_{aB}} \frac{\partial \II}{\partial \overline{h^{(l)}_{aB}}}  + \frac{1}{n} \frac{\partial^2 \II}{\partial h^{(l)}_{aB} \partial  \overline{h^{(l)}_{aB}}} \right)
		=:&~\Theta^{(1)}_{i,t}+\Theta^{(2)}_{i,t}+\Theta^{(3)}_{i,t},
	\end{align}
	and the diffusion term $\dd M_t$ given by
	\begin{align}\label{eq_diffusion}
		\dd M_{i,t} =&\sum_{a=1}^n \sum_{B=n+1}^{2n} \sum_{l=1}^2  \left(\frac{ \partial \II}{\partial h^{(l)}_{aB}}  \frac{\dd\mathcal{M}^{(l)}_{aB}}{\sqrt{n}} + \frac{\partial \II}{\partial \overline{h^{(l)}_{aB}}} \frac{\big(\dd\mathcal{M}^{(l)}_{aB}\big)^*}{\sqrt{n}}  \right).
	\end{align}

	Taking the expectation of \eqref{eq_dynamics}, the diffusion terms in (\ref{eq_diffusion}) vanish from the martingale property of Brownian motions, and we have
	\begin{align}
		\frac{\dd \E[\II]}{\dd t}=\E\big[\Theta^{(1)}_{i,t}\big]+\E\big[\Theta^{(2)}_{i,t}\big]+\E\big[\Theta^{(3)}_{i,t}\big].
	\end{align}
We next apply the cumulant expansion formula  in Lemma \ref{cumulant_expansion_lemma} to the first two drift terms, \ie $\Theta^{(1)}_{i,t}$ and $\Theta^{(2)}_{i,t}$ in (\ref{eq_drift}), with respect to $\big(h^{(l)}_{aB}\big)$ and $\big(\overline{h^{(l)}_{aB}}\big)$.  We remark that the cumulant expansion formula has been widely used in random matrix theory recently, \eg \cite{zbMATH06775355,zbMATH07549410,cipolloni2024universalityextremaleigenvalueslarge}. 
From \ref{it:condmom4}-\ref{it:condmom}, we find that the second order cumulants of $\big(h^{(l)}_{aB}\big)$ are invariant over time, \ie
\begin{align}\label{eq_var}
	\kappa^{(1,1)}\big(h^{(l)}_{aB}\big)=\frac{1}{n}, \qquad \kappa^{(0,2)}\big(h^{(l)}_{aB}\big)=\kappa^{(2,0)}\big(h^{(l)}_{aB}\big)=0, 
\end{align}
and the higher-order cumulants are bounded by
\begin{align}\label{eq_cumu_bound}
	|\kappa^{(p,q)}\big(h^{(l)}_{aB}\big)| \lesssim  n^{-2-c_0},\qquad p+q=3,4, \qquad |\kappa^{(p,q)}\big(h^{(l)}_{aB}\big)| \lesssim  n^{-\frac{p+q}{2}}, \qquad p+q\geq 5,
\end{align}
uniformly for any $t\geq 0$. We then write out the second order terms in the expansions of $\Theta^{(1)}_{i,t}$ and $\Theta^{(2)}_{i,t}$, \ie
\begin{align}\label{eq_cumulant_exp}
	\mathbb{E} \big[  \Theta^{(1)}_{i,t}+\Theta^{(2)}_{i,t} \big]=&-\frac{1}{2}\sum_{a=1}^n \sum_{B=n+1}^{2n} \sum_{l=1}^2  \mathbb{E} \left[ h^{(l)}_{aB} \frac{  \partial \II}{\partial h^{(l)}_{aB}}  +\overline{h^{(l)}_{aB}} \frac{\partial \II}{\partial \overline{h^{(l)}_{aB}}} \right]\nonumber\\
	=&-\frac{1}{n}\sum_{a=1}^n \sum_{B=n+1}^{2n} \sum_{l=1}^2  \mathbb{E} \left[   \frac{  \partial^2 \II}{\partial h^{(l)}_{aB} \partial \overline{h^{(l)}_{aB}}}  \right]+\mbox{higher order terms},
\end{align}
using \eqref{eq_var}. Then we observe that the above second order cumulant expansion terms will cancel precisely with the third drift term in \eqref{eq_drift}.
This cancellation is a consequence of the fact that the second moments of the matrix entries in \eqref{OU} do not change with time. Therefore, the remaining terms in \eqref{eq_derivative} are the higher order cumulant expansion terms in \eqref{eq_cumulant_exp}.

	To estimate the higher order terms in \eqref{eq_cumulant_exp}, we introduce the following differentiation rules  obtained from \eqref{xh} and the definition of resolvent:
\begin{align}
	&\frac{\partial (G^w_t)_{ij}}{\partial h^{(1)}_{aB}}=-(G^w_t)_{ia} (G^w_t)_{Bj}, \qquad \frac{\partial (G^w_t)_{ij}}{\partial \overline{h^{(1)}_{aB}}}=-(G^w_t)_{iB} (G^w_t)_{aj},\label{eq_diff_1}\\
	&\frac{\partial (G^w_t)_{ij}}{\partial h^{(2)}_{aB}}=w(G^w_t)_{ia} (G^w_t)_{Bj}, \qquad \frac{\partial (G^w_t)_{ij}}{\partial \overline{h^{(2)}_{aB}}}=\overline{w} (G^w_t)_{iB} (G^w_t)_{aj},\label{eq_diff_2}\\
	&\frac{\partial (G_t)_{ij}}{\partial h^{(1)}_{aB}}=\frac{\partial (G_t)_{ij}}{\partial \overline{h^{(1)}_{aB}}}=0, \qquad\qquad \frac{\partial (G_t)_{ij}}{\partial h^{(2)}_{aB}}=-(G_t)_{ia} (G_t)_{Bj}, \qquad\qquad \frac{\partial (G_t)_{ij}}{\partial \overline{h^{(2)}_{aB}}}= -(G_t)_{iB} (G_t)_{aj},\label{eq_diff_3}
\end{align}
for any $i,j \in \llbracket 2n \rrbracket $, $a \in 
\llbracket n \rrbracket$, and 
$B\in \llbracket n+1,2n \rrbracket$, together with the following lemma. The proof relies on the above differentiation rules and the local laws in Lemma \ref{lemma_local_law}, and is postponed until the end of this section.
\begin{lemma}\label{lemma_diff}
	For any $p,q \in \N$ with $p+q\geq 1$, we have
	\begin{align}\label{eq_bound_k}
		\sup_{t \in[0,T_0]} \left|  \frac{\partial^{p+q} \II }{\partial (h^{(l)}_{aB})^{p} (\overline{h^{(l)}_{aB}})^{q}} \right| \prec n^{2(p+q)\epsilon}, \qquad 1\leq i\leq k, \qquad l=1,2, \qquad T_0=100 \log n.
	\end{align}
\end{lemma}

 Applying Lemma \ref{cumulant_expansion_lemma}, we stop the cumulant expansions in (\ref{eq_cumulant_exp}) at the fourth order, \ie 
\begin{align}\label{eq_derivative_F0}
	 \frac{\dd \mathbb{E}\big[ \II\big]}{\dd t}
	=&-\frac{1}{2} \sum_{a=1}^n \sum_{B=n+1}^{2n}
	\sum_{l=1}^2 \left( \sum_{p+q+1= 3}^{4}\frac{\kappa^{(p+1,q)}(h^{(l)}_{aB})}{p!q!}
	\mathbb{E} \left[   \frac{\partial^{p+q+1} \II}{\partial (h^{(l)}_{aB})^{p+1} \partial (\overline{h^{(l)}_{aB}})^{q} } \right] \right)\nonumber\\
	&-\frac{1}{2} \sum_{a=1}^n \sum_{B=n+1}^{2n} \sum_{l=1}^2 \left(\sum_{p+q+1 = 3}^{4}\frac{\kappa^{(p+1,q)}(\overline{h^{(l)}_{aB}})}{p!q!}
	\mathbb{E} \left[   \frac{\partial^{p+q+1} \II }{\partial (\overline{h^{(l)}_{aB}})^{p+1} \partial (h^{(l)}_{aB})^{q}}\right]\right)+R_5,
\end{align}
with the truncating error $R_5$ given as in (\ref{cumulant_error}) and bounded by
\begin{align}\label{eq_error}
	|R_5|=O_\prec(n^{-1/2+c\epsilon}),
\end{align}
for some constant $c>0$, using the fifth cumulant bound \eqref{eq_cumu_bound}, the differentiation rules \eqref{eq_diff_1}-\eqref{eq_diff_3}, and the local law estimates in \eqref{eq:local_law_G_z}. Such truncation argument is standard and frequently used in previous works, so we postpone the details till the end of this section. Moreover, using  the cumulant bounds in \eqref{eq_cumu_bound} and Lemma \ref{lemma_diff}, the third and fourth order terms with $p+q+1=3,4$ above are bounded by $O_\prec(n^{-c_0+(p+q+1)\epsilon})$.  Therefore, we obtain
\begin{align}\label{eq_derivative_F}
	\left| \frac{\dd \mathbb{E}\big[ \II\big]}{\dd t}\right|=O(n^{-c_0/2}),
\end{align}
using that $\epsilon<c_0/100$ is chosen sufficiently small. We remark that the above error bounds holds uniformly for any $0\leq t\leq 100 \log n$. 	

Integrating \eqref{eq_derivative_F} over $t \in [0, T_0]$ with  $T_0=100\log n$,  we obtain 
\begin{align}\label{eq_long_time}
	\sup_{t \in [0,T_0]}\Big| \mathbb{E}\big[ \mathcal{L}_{i,t}\big]-\mathbb{E}\big[ \mathcal{L}_{i,T_0}\big] \Big| =O(n^{-c_0/3}). 
\end{align}
Note that $H^{(w)}_{t}$ is defined as in \eqref{xh} with the time dependent matrix $X^{(w)}_t$ under the evolution \eqref{OU}. For any fixed $t\ge 0$, $X^{(w)}_t$ in distribution equals to
\begin{align}\label{interpolate}
 \mathcal{X}^{(w)}_t := \mathrm{e}^{-\frac{t}{2}} X^{(w)}+\sqrt{1-\mathrm{e}^{-t}} {X^{\mathrm{Gin}}},
\end{align}
with $\mathcal{X}^{(w)}_0=X^{(w)}$ defined in \eqref{eq_H} and $\mathcal{X}^{(w)}_\infty={X^{\mathrm{Gin}}}$ being the Ginibre counterpart which is independent of $X^{(w)}$. We use $\mathcal{G}_t^{(w)}$ to denote the corresponding resolvents defined as in (\ref{eq_H}) and (\ref{eq_resolvent}) and $\mathcal{G}_t^{(w)} \overset{\mathrm{d}}{=} G_t^{(w)}$. Using the resolvent identity, we obtain
\begin{align}\label{eq_final}
	\|\mathcal{G}^{(w)}_{T_0}(\ii \eta)-\mathcal{G}^{(w)}_{\infty}(\ii \eta)\|_{\mathrm{op}} \leq \|\mathcal{G}^{(w)}_{T_0}(\ii \eta)\|_{\mathrm{op}} \|\mathcal{G}^{(w)}_{\infty}(\ii \eta)\|_{\mathrm{op}}\|\mathcal{X}^{(w)}_{T_0}-\mathcal{X}^{(w)}_{\infty}\|_{\mathrm{op}} 
	\prec n^{-50}, \qquad \eta \geq \eta_0,
\end{align}
where we also used  $\|\mathcal{G}^{(w)}(\ii \eta)\|_{\mathrm{op}} \leq \eta^{-1}$, $\|\mathcal{X}\|_{\mathrm{op}} \leq n \|\mathcal{X}\|_{\mathrm{max}}$, and $\big|(\mathcal{X}^{(w)}_{T_0})_{ij}-(\mathcal{X}^{(w)}_{\infty})_{ij}\big| \prec n^{-100}$ which follows from \eqref{interpolate} and the moment condition \ref{it:condmom}. Thus 
\begin{align}\label{eq_long_time_2}
		\Big| \mathbb{E}\big[ \mathcal{L}_{i,T_0}\big]-\mathbb{E}^{\mathrm{Gin}}\big[ \mathcal{L}_{i}\big] \Big|  =O(n^{-10}). 
\end{align}
Combining \eqref{eq_long_time} with \eqref{eq_long_time_2}, we conclude
\begin{align}\label{eq_conclude}
	\sup_{t \in [0,T_0]}\Big| \mathbb{E}\big[ \mathcal{L}_{i,t}\big]-\mathbb{E}^{\mathrm{Gin}}\big[ \mathcal{L}_{i}\big]  \Big| =O(n^{-c_0/3}). 
\end{align}
We hence finish the proof of the second statement (\ref{eq_GFT2}) in Lemma \ref{lemma_GFT}.

	\medskip

	The proof of the first statement (\ref{eq_GFT}) is similar but with more complexity. Recall the definition of $\wt{\mathcal{X}_{n,t}}$ from \eqref{eq_XXX}. Since $\E^{\mathrm{Gin}}[\mathcal{L}_{i}]$ is a deterministic number independent of $t$, the same dynamics as in (\ref{eq_dynamics}) also applies to $\mathcal{L}_{i,t}-\E^{\mathrm{Gin}}\big[\mathcal{L}_{i}\big]$. Using Itô's formula, (\ref{eq_dynamics}), and the chain rule to $\wt{\mathcal{X}_{n,t}}$, then taking the expectation, we have
	\begin{align}\label{eq_derivative}
		\dd \mathbb{E}\big[ \wt{\mathcal{X}_{n,t}}\big]=\sum_{j=1}^k \mathbb{E} \Big[ \big( \partial_{j} \wt{\mathcal{X}_{n,t}}\big) \big( \Theta^{(1)}_{j,t}+\Theta^{(2)}_{j,t}+\Theta^{(3)}_{j,t} \big) \Big] \dd t + \frac{1}{2} \sum_{j_1=1}^k \sum_{j_2=1}^k  \mathbb{E}\Big[ \big(\partial^2_{j_1j_2} \wt{\mathcal{X}_{n,t}} \big) \big\< \dd M_{j_1,t}, \dd M_{j_2,t}  \big\>\Big],
	\end{align}
	 where $\partial_{j}$ denotes the partial derivative of $F_n$ with respect to the $j$-th coordinate, and $\partial^2_{j_1j_2}$ denotes the second order partial derivative with respect to the $j_1$-th and $j_2$-th coordinate. From \eqref{eq_diffusion} and the 
	covariations 
	$\langle \mathrm d \mathcal M_{aB}^{(\ell)}, 
	\mathrm d \overline{\mathcal M}_{a'B'}^{(\ell')} \rangle
	= \delta_{\ell,\ell'} \delta_{a,a'} \delta_{B,B'}$ 
	and
	$\langle \mathrm d \mathcal M_{aB}^{(\ell)}, 
	\mathrm d \mathcal M_{a'B'}^{(\ell')} \rangle
	= 0$ with $a \in 
	\llbracket n \rrbracket$ and 
	$B\in \llbracket n+1,2n \rrbracket$, we obtain
	\begin{align}\label{eq_variation}
		\frac{1}{2} \mathbb{E}\Big[ \big(\partial^2_{j_1j_2} \wt{\mathcal{X}_{n,t}} \big) \big\< \dd M_{j_1,t}, \dd M_{j_2,t}  \big\>\Big]=&
		\frac{1}{2n} \sum_{l=1}^2 \mathbb{E} \left[\big(\partial^2_{j_1j_2} \wt{\mathcal{X}_{n,t}} \big)
		\sum_{a=1}^n \sum_{B=n+1}^{2n}  \left(   \frac{ \partial \mathcal{L}_{j_1,t}}{\partial h^{(l)}_{aB}}      \frac{\partial \mathcal{L}_{j_2,t}}{\partial \overline{h^{(l)}_{aB}}} +\frac{ \partial \mathcal{L}_{j_2,t}}{\partial h^{(l)}_{aB}}      \frac{\partial \mathcal{L}_{j_1,t}}{\partial \overline{h^{(l)}_{aB}}}  \right)\right] \dd t.
	\end{align}

	Recall the definitions of $\Theta^{(1)}_{j,t}$, $\Theta^{(2)}_{j,t}$, and $\Theta^{(3)}_{j,t}$ from (\ref{eq_drift}). To compute expectations of the first two drift terms in \eqref{eq_derivative}, we apply the cumulant expansion formula  with respect to $\big(h^{(l)}_{aB}\big)$ and $\big(\overline{h^{(l)}_{aB}}\big)$ as in (\ref{eq_cumulant_exp}). Since the variances of matrix entries \eqref{eq_var} are invariant in time,  we  again observe that the second order terms in the cumulant expansions cancel precisely with the third drift term and the term in \eqref{eq_variation}.  We then stop the expansions at the fourth order as in \eqref{eq_derivative_F0} with the truncating error term denoted by $\wt{R_5}$,\ie
	\begin{align}\label{eq_cumulant_exp_FF}
		\frac{\dd \mathbb{E}\big[ \wt{\mathcal{X}_{n,t}}\big]}{\dd t}
		=&-\frac{1}{2} \sum_{a=1}^n \sum_{B=n+1}^{2n}
		\sum_{l=1}^2 \left( \sum_{p+q+1= 3}^{4}\frac{\kappa^{(p+1,q)}(h^{(l)}_{aB})}{p!q!}
		\mathbb{E} \left[   \frac{\partial^{p+q+1} \wt{\mathcal{X}_{n,t}}}{\partial (h^{(l)}_{aB})^{p+1} \partial (\overline{h^{(l)}_{aB}})^{q} } \right] \right)\nonumber\\
		&-\frac{1}{2} \sum_{a=1}^n \sum_{B=n+1}^{2n} \sum_{l=1}^2 \left(\sum_{p+q+1 = 3}^{4}\frac{\kappa^{(p+1,q)}(\overline{h^{(l)}_{aB}})}{p!q!}
		\mathbb{E} \left[   \frac{\partial^{p+q+1}\wt{\mathcal{X}_{n,t}} }{\partial (\overline{h^{(l)}_{aB}})^{p+1} \partial (h^{(l)}_{aB})^{q}}\right]\right)+\wt{R_5}.
	\end{align}
  One can use the differentiations rules in (\ref{eq_diff_1})-(\ref{eq_diff_3}) and the chain rule to compute the partial derivatives of $\wt{\mathcal{X}_{n,t}}$ defined in (\ref{eq_XXX}). Note that $$\frac{\partial^{p+q} }{\partial (h^{(l)}_{aB})^{p} (\overline{h^{(l)}_{aB}})^{q}}\Big(\mathcal{L}_{i,t}-\E^{\mathrm{Gin}}\big[\mathcal{L}_{i}\big] \Big) =\frac{\partial^{p+q} \mathcal{L}_{i,t}}{\partial (h^{(l)}_{aB})^{p} (\overline{h^{(l)}_{aB}})^{q}}.$$ 
Moreover, combining \eqref{eq_conclude} with \eqref{eq_bound1}-\eqref{eq_bound2}, we know
    \begin{align}\label{eq_bound3}
    \left| \mathcal{L}_{i,t}-\E^{\mathrm{Gin}}\big[\mathcal{L}_{i}\big] \right| \prec (\log n)^C, \qquad 1\leq i\leq k,
    \end{align}
    uniformly for $t\in [0,T_0]$.
Using the upper bounds of derivatives of $F_n$ in \eqref{eq_F} and \eqref{eq_bound3}, Lemma \ref{lemma_diff} directly implies
	\begin{align}\label{eq_bound_F}
		\sup_{t \in[0,T_0]} \left|  \frac{\partial^{p+q} \wt{\mathcal{X}_{n,t}} }{\partial (h^{(l)}_{aB})^{p} (\overline{h^{(l)}_{aB}})^{q}} \right| \prec (\log n)^{C} n^{(p+q)\epsilon}, \qquad p,q \in \N, \qquad p+q\geq 1,
	\end{align}
for some constant $C>0$. Hence using the cumulant bounds in \eqref{eq_cumu_bound} and \eqref{eq_bound_F}, the third and fourth order terms with $p+q+1=3,4$ above are bounded by $O_\prec\Big((\log n)^{C}n^{-c_0+(p+q+1)\epsilon}\Big)$. Moreover, the truncating error term $\wt{R_5}$ can be bounded by $O_\prec(n^{-1/2+c'\epsilon})$ using a similar argument as in \eqref{eq_error}. Therefore, we obtain
	\begin{align}\label{eq_cumulant_exp_F}
		\Big|\frac{\dd \mathbb{E}\big[ \wt{\mathcal{X}_{n,t}}\big]}{\dd t} \Big|
		=&O(n^{-c_0/2}),
	\end{align}
with $\epsilon<c_0/100$ chosen sufficiently small, uniformly for $t\in [0,T_0]$. Repeating the same arguments as in \eqref{eq_long_time}-\eqref{eq_conclude} and combining with \eqref{eq_minor}, we finish the proof of \eqref{eq_GFT} and hence Lemma \ref{lemma_GFT}.	
\end{proof}

It remains to prove Lemma \ref{lemma_diff}.

\begin{proof}[Proof of Lemma \ref{lemma_diff}]
	
	Recall $\II$ in \eqref{eq_LL} and the differentiation rules in \eqref{eq_diff_1}-\eqref{eq_diff_3}. We start with $1\leq i\leq k_1$ for fixed points $w_i \in \C$. Below we will present the proof for $l=2$ using \eqref{eq_diff_2}, and the proof for $l=1$ is very similar using \eqref{eq_diff_1}. 
	Using \eqref{eq_diff_2},\eqref{eq_diff_3}, and that $G^2=-\ii \frac{\dd}{\dd \eta} G(\ii \eta)$, we have
	\begin{align}\label{diff_L}
		\frac{\partial \II }{\partial h^{(2)}_{aB}}
		=&\frac{\ii}{2}  \int_{\eta_0}^{T}  \big[ w_i\big(G_t^{w_i}(\ii \eta)\big)^2  +\big(G_t(\ii \eta)\big)^2 \big]_{Ba}  \dd \eta =   -\frac{1}{2}\big[w_i G^{w_i}_t(\ii \eta_0)+G_t(\ii \eta_0)\big]_{Ba} +O_\prec(n^{-10}), \qquad 1\leq i\leq k_1,
	\end{align}
	where we also used that $\|G(\ii \eta)\|_{\mathrm{op}} \leq \eta^{-1}$ for $\eta=T=n^{100}$. Note that the variances of entries of $X^z_t$ in \eqref{xh} do not change with time (see \eqref{eq_var}). Hence, the local laws as in Lemma \ref{lemma_local_law} also hold true for $G^z_t$ with a fixed $t\geq 0$. 
Using \eqref{eq:local_law_G_z} and \eqref{rho} for $\eta_0=n^{-1-\epsilon}$, $\mathbf{x}=\mathbf{e}_{x}$, and $\mathbf{y}=\mathbf{e}_{y}$, we conclude that, for each fixed $t\geq 0$,
	$$
	\sup_{z\in \C} \max_{x,y \in \llbracket 2n \rrbracket}   \Big|\big( G^{z}_t(\ii \eta_0) \big)_{xy}\Big| \prec n^{\epsilon}, \qquad \mbox{~with~} \quad \eta_0=n^{-1-\epsilon}.
	$$
	Due to the H\"{o}lder continuity of $G^{z}_t$ in $t$, one can show that the above bound indeed holds true simultaneously for any $t\in [0,T_0]$ with $T_0=100\log n$, \ie
	\begin{align}\label{eq_G_bound}
		\sup_{t\in [0,T_0]}  \left\{    \sup_{z\in \C} \max_{x,y \in \llbracket 2n \rrbracket}  \Big|\big( G^{z}_t(\ii \eta_0) \big)_{xy}\Big| \right\} \prec n^{\epsilon}, \qquad \mbox{~with~} \quad \eta_0=n^{-1-\epsilon},
	\end{align}
	using a grid argument on $[0,T_0]$ by taking the union bound.  More precisely, from the definition of $\prec$ in \eqref{prec}, fixing any small $\xi>0$ and large $D>100$, and setting $t_j:=n^{-100}j$, we have 
    \begin{multline}\label{union}
    \P\left( \exists j \in \lfloor n^{100}T_0 \rfloor : \sup_{z\in \C} \max_{x,y \in \llbracket 2n \rrbracket}   \Big|\big( G^{z}_{t_j}(\ii \eta_0) \big)_{xy}\Big| \geq n^{\epsilon+\xi}  \right) \\
    \leq \sum_{j \in \lfloor n^{100}T_0 \rfloor}\P\Big(  \sup_{z\in \C} \max_{x,y \in \llbracket 2n \rrbracket}   \Big|\big( G^{z}_{t_j}(\ii \eta_0) \big)_{xy}\Big| \geq n^{\epsilon+\xi} \Big) \leq n^{-D+100}.
    \end{multline}
    By choosing $D$ sufficiently large, this proves \eqref{eq_G_bound} for all $t_i$ simultaneously. Note that from the dynamics (\ref{OU}), (\ref{eq_H}) and the resolvent identity, then for any $|t-t'| \leq n^{-100}$ and $|z|\leq 1$, we have
    \begin{align}
    \max_{x,y \in \llbracket 2n \rrbracket}\Big| \big(G^{z}_t(\ii \eta_0)\big)_{xy}-\big(G^{z}_{t'}(\ii \eta_0)\big)_{xy}\Big| \leq \|G_t^{z}(\ii \eta_0)-G^{z}_{t'}(\ii \eta_0)\|_{\mathrm{op}} \leq  \frac{1}{\eta_0^2} \|H_t^{z} -H_{t'}^{z} \|_{\mathrm{op}} \leq  \frac{n}{\eta_0^2} \|H_t^{z} -H_{t'}^{z} \|_{\mathrm{max}} \prec n^{-10},
    \end{align}
    where we also used  $\|G^z_t(\ii \eta)\|_{\mathrm{op}} \leq \eta^{-1}$ and that OU paths in (\ref{OU}) are almost surely $\alpha$-H\"{o}lder continuous with $\alpha<1/2$. For $|z|>1$, a similar estimate holds true using the relation \eqref{eq_G_inverse}. This, together with \eqref{union},  proves \eqref{eq_G_bound}.  The same bound also applies to $G_t$ using that $G_t\overset{\mathrm{d}}{=}G^{z=0}_t$.
    Hence we obtain from \eqref{diff_L} that
	\begin{align}
		\sup_{t\in [0,T_0]} \Big|\frac{\partial \II }{\partial h^{(2)}_{aB}}\Big|  =O_\prec (n^{\epsilon}), \qquad 1\leq i\leq k_1.
	\end{align} 
	The same result also applies to $\partial /\partial \overline{h^{(l)}_{aB}}$. In general, for any $p,q \in \N$ with $p+q\geq 1$, using the differentiation rules in \eqref{eq_diff_1}-\eqref{eq_diff_3} and the local law estimates in \eqref{eq_G_bound} repeatedly, we obtain the desired estimate \eqref{eq_bound_k} for any $(p+q)$-th derivatives of $\II$ for $\eqref{eq_bound_k}$. 

    The proof of \eqref{eq_bound_k} for $k_1 \leq i\leq k$ is similar, by noting that (\cf \eqref{diff_L})
    	\begin{align}\label{diff_LL}
		\frac{\partial \II }{\partial h^{(2)}_{aB}}
		=&\frac{\ii}{4\pi} \int \Delta_z f^{(i)}_n(z) \int_{\eta_0}^{T}  z\big[ \big(G_t^z(\ii \eta)\big)^2  \big]_{Ba}  \dd \eta \dd^2 z=-\frac{1}{4\pi} \int \Delta_z f^{(i)}_n(z)  \big(z G^z_t(\ii \eta_0)\big)_{Ba} \dd^2 z+O_\prec(n^{-10}).
	\end{align}
    It is clear from \eqref{eq_G_bound} that $\big|z\big( G^z_t(\ii \eta_0)\big)_{Ba}\big| \prec n^{\epsilon}$ for $|z|\leq 1$, and this bound can be extended to $|z|\geq 1$ using the relation \eqref{eq_G_inverse}. Combining this with the $L^{1}$-norm bound in (\ref{eq_f}), we obtain the desired upper bound as in \eqref{eq_bound_k}.  This ends the proof of Lemma \ref{lemma_diff}.
\end{proof}

We end this section with the proof of Eq. \eqref{eq_error}.

\begin{proof}[Proof of equation \eqref{eq_error}]
	We will only present the details for $1\leq i\leq k_1$. The proof for $k_1+1 \leq i\leq k$ is quite similar using the $L^{1}$-norm bound in (\ref{eq_f}), so we omit it. 
	
	Recall  (\ref{cumulant_error}) and the differentiation rules from (\ref{eq_diff_1})-(\ref{eq_diff_3}). To bound the truncating error in \eqref{eq_derivative_F0}, it then suffices to bound, for any $a\in \llbracket n \rrbracket $ and $B\in \llbracket n+1,2n \rrbracket$, 
	\begin{align}\label{R_5}
		|R^{(aB)}_{5}| &\leq  C \E[ |h^{(l)}_{aB}|^{5}] \E \Big[ \max_{p+q+1=5} \Big\{ \sup_{|z| \leq n^{-1/2+\xi}} \Big| \frac{\partial^{p+q+1} }{\partial (h^{(l)}_{aB})^{p+1} \partial (\overline{h^{(l)}_{aB}})^{q} } \II^{(aB)}(z) \Big| \Big\} \Big]\nonumber\\
		&+C  \E \Big[ |h^{(l)}_{aB}|^{5} 1_{|h_{aB}|> n^{-1/2+\xi}}\Big] \E \Big[\max_{p+q+1=5} \Big\{ \sup_{z \in \C} \Big| \frac{\partial^{p+q} }{\partial (h^{(l)}_{aB})^{p} \partial (\overline{h^{(l)}_{aB}})^{q} } \II^{(aB)}(z)  \Big| \Big\} \Big],
	\end{align}
	with a fixed small $\xi>0$, and where $h^{(l)}_{aB}$ for $l=1,2$ are defined in (\ref{def_HAB}), and $\II^{(aB)}(z)$ is obtained from $\II$ (which is defined in (\ref{eq_LL})) by replacing the random variable $h^{(l)}_{aB}$  with the deterministic value $z$, and replacing $h^{(l)}_{Ba}$ with $\overline{z}$. To keep the proof short, we will only consider $l=1$, and the proof for $l=2$ is similar with a slight modification. In such case, we write $h_{aB}=h_{aB}^{(1)}$, drop the dependence on $t$ for brevity, and define
		$$\mathfrak{H}(z): = H^{(aB)}+z E^{(Ba)}+\overline{z} E^{(aB)},\qquad \mathfrak{G}(z) :=\big( \mathfrak{H}(z) -\ii \eta_0 \big)^{-1},$$
with  $H^{(aB)}:=H^{w_i}-h_{aB}E^{(aB)}-\overline{h_{aB}}E^{(Ba)}$ and $E^{(aB)}:=(\delta_{aB})_{i,j=1}^{2n}$. By a similar argument as in (\ref{diff_L}), we obtain
\begin{align}
	\frac{\partial \II^{(aB)}(z) }{\partial h_{aB}}= \frac{1}{2}\big[\mathfrak{G}(z)\big]_{Ba} +O_\prec(n^{-100}),
\end{align}
where we used (\ref{eq_diff_1}), (\ref{eq_diff_3}), and that $G^2=-\ii \frac{\dd}{\dd \eta} G(\ii \eta)$. One can compute higher order partial derivatives similarly.

We first estimate the second line of (\ref{R_5}). Using the deterministic upper bound for $\max_{i,j} |\mathfrak{G}_{ij}(z)| \leq \|\mathfrak{G}(z)\|_{\mathrm{op}} \leq \eta^{-1}_0=n^{1+\epsilon}$ uniformly for any $z\in \C$, we obtain from (\ref{eq_diff_1}) and (\ref{eq_diff_3}) that
$$\max_{p+q+1=5} \left\{ \sup_{z \in \C} \Big| \frac{\partial^{p+q} }{\partial (h_{aB})^p \partial (\overline{h_{aB}})^{q}} \big[\mathfrak{G}(z)\big]_{Ba} \Big| \right\} =O\big(n^{5(1+\epsilon)}\big).$$
Moreover, from the moment condition \ref{it:condmom} and by the Markov inequality, we have
$$\E \Big[ |h_{aB}|^{5} 1_{|h_{aB}|> n^{-1/2+\xi}}\Big] \leq \Big( \E \Big[ |h_{aB}|^{10}\Big] \Big)^{1/2} \Big( \P\big( |h_{aB}|> n^{-1/2+\xi} \big) \Big)^{1/2}  \leq n^{-D}$$
for any large $D>0$, since all the moments are finite. Thus the second line of \eqref{R_5} can be bounded by $O(n^{-100})$ by choosing $D$ sufficiently large.

We next estimate  the first line of (\ref{R_5}). Using the second resolvent identity, we have
	\begin{align}\label{resolvent_expansion}
		\mathfrak{G}(0)=G^{w_i}+ \mathfrak{G}(0) \Big( h_{aB}E^{(aB)}+\overline{h_{aB}} E^{(Ba)} \Big) G^{w_i}.
	\end{align}
	From the local law in~\eqref{eq_G_bound}, we have $ \max_{i,j}|G^{w_i}_{ij}| \prec n^{\epsilon}$. In addition, we have $|h_{aB}| \prec n^{-1/2}$ from the moment condition \ref{it:condmom}. Therefore, we have from \eqref{resolvent_expansion} that $\max_{i,j}|\mathfrak{G}_{ij}(0)| \prec n^{\epsilon}$. Similarly, we have
	\begin{align}\label{resolvent_expansion2}
		\mathfrak{G}(z)=\mathfrak{G}(0)- \mathfrak{G}(z) \Big( z E^{(aB)}+\overline{z} E^{(Ba)} \Big) \mathfrak{G}(0),
	\end{align}
	and thus
	\begin{align}
		\sup_{|z|<n^{-1/2+\xi}}\Big\{ \max_{i,j} \Big|\mathfrak{G}_{ij}(z)\Big|\Big\} \prec n^{\epsilon}.
	\end{align}
	Further using (\ref{eq_diff_1}) and (\ref{eq_diff_3}), we obtain that
	$$\max_{p+q+1=5} \left\{ \sup_{|z| < n^{-1/2+\xi}} \Big| \frac{\partial^{p+q} }{\partial (h_{aB})^p \partial (\overline{h_{aB}})^{q}} \big[\mathfrak{G}(z)\big]_{Ba} \Big| \right\}  \prec n^{\epsilon}.$$
	Together with $\E|h_{aB}|^5\le Cn^{-5/2}$ from the moment condition \ref{it:condmom} and Lemma \ref{lemma_dominant}, the first term on the right side of (\ref{R_5}) is bounded by $O_{\prec}(n^{-5/2+c\epsilon})$. 
	
	Hence summing over $a \in \llbracket n \rrbracket,B\in \llbracket n+1,2n \rrbracket$ the truncation error $R_5$ in the cumulant expansions satisfies $|R_5|=O_{\prec}(n^{-1/2+c\epsilon})$.

\end{proof}

\section{Discussion about geometric universality}
\label{se:geometric}

\subsection{Discussion about background charge universality}
\label{se:Background}

The name Coulomb gas given to \eqref{eq:gas:C} 
is justified by thinking of it as
the law of $n$ positively charged particles at equilibrium with
charge $c=1/(n+1)$, confined by the 
background charge $-\mu$ and at inverse
temperature $\beta = 2(n+1)^2$. 
Indeed, the total (potential) energy would be
\[H_n(z_1,\dots,z_n) = 
-c^2 \sum_{i<j}  \log |z_i - z_j| + c\sum_{i=1}^n Q(z_i)\]
so that the Boltzmann-Gibbs measure with
energy $H_n$ and inverse temperature $\beta$ coincides with
\eqref{eq:gas:C},
\[\mathrm{e}^{-\beta H_n}
= \mathrm{e}^{-2(n+1)\sum_{k=1}^nQ(z_k)}\prod_{i<j}|z_i-z_j|^2.\]
We may think our space $\mathbb C$ as filled with
a different background charge $-\mu$ and ask if
Theorem \ref{th:LogSing} still holds true. 
If we wish to allow
singularities at any points, even at $\infty$, 
we should choose 
a  probability measure $\mu$ 
on $\overline{\mathbb C}$ having 
a strictly positive smooth density everywhere (a volume form). 
More precisely, this means that
$\mu$ has a strictly positive smooth
density $\rho$ with respect to Lebesgue measure
and that
$|z|^{-4}\rho(1/z)$ smoothly extends and 
it is positive
around $0$. An equivalent way to obtain this
is to consider a smooth function
$Q: \mathbb C \to \mathbb R$ with
a strictly positive Laplacian and such that
$\widetilde Q(z) = Q(1/z) + \log|z|$ extends
smoothly at zero and $\Delta \widetilde Q(0) > 0$.

Due to the logarithmic interaction
together with the choice of inverse temperature,
we still get a determinantal point process
governed by the projection
on $L^2_{\mathbb C}(\mathrm{e}^{-2(n+1)Q(z)} \mathrm d^2 z )$
onto the space of polynomials
of degree less or equal than $n-1$. Due to the
behavior of $Q$ at infinity, the 
measure $\sigma$ given by $\mathrm d \sigma(z) = 
\mathrm{e}^{-4Q(z)} \mathrm d^2 z $
has a strictly positive smooth density everywhere
on $\overline{\mathbb C}$,
as defined above. The space
$\widetilde {\mathcal P_n}$ that replaces \eqref{eq:Pn}
would be
\begin{equation}\label{eq:Pn:replace}
    \widetilde{\mathcal P_n} = 
    \bigg\{(a_0 + a_1 z+ \dots + a_{n-1}z^{n-1})\mathrm{e}^{-(n-1)Q(z)}:
	a_0,\dots,a_{n-1} \in \mathbb C\bigg\},
\end{equation}
and the $L^2_{\mathbb C}(\mu)$ space is replaced by 
$L^2_{\mathbb C}(\sigma)$. Then, 
if we consider any orthonormal basis $\{P_{n,k}: 0 \leq k \leq n-1 \}$ of $\widetilde{\mathcal P_n}
\subset L^2_{\mathbb C}(\sigma)$, the kernel
$K_n$ is replaced by the orthogonal projection 
$\widetilde {K_n}(z,w)$
of $L^2_{\mathbb C}(\sigma)$ onto
$\widetilde {\mathcal P_n}$
\[\widetilde {K_n}(z,w)
= \sum_{k=0}^{n-1} P_{n,k}(z) \overline{P_{n,k}(w)}.\]
In the general setting
the inequality in \eqref{eq:kernel-distance}
needs to be replaced by the weaker
$|K_n(x,y)|\leq Cn\mathrm{e}^{-c\sqrt{n}d(x,y)}$
some $C,c>0$. Here we can keep the same (half) chordal
distance $d$ if we prefer.
This weaker inequality is enough to show 
Lemmas
\ref{lem:lip-variance}, \ref{lem:Covariance} and \ref{lem:Independence}. The only missing step
is handling, for every $p \in \overline{\mathbb C}$, a fixed function with a single
logarithmic singularity at $p$,
i.e., a substitute for Lemma \ref{le:Green:cum}
or Subsection \ref{sub:SingleSingularity}. 
Since the local behavior matches that of
the spherical ensemble
and the fluctuations only feel the
singular part,
 we should obtain the same kind of fluctuations.

\subsection{Discussion about surface universality}
\label{se:Surface}

Here we recall one of the standard generalizations
of the spherical ensemble to other surfaces. Consider a 
compact Riemann
surface $S$ and choose
a holomorphic line bundle $L$ over $S$ endowed with
a (smooth) Hermitian metric $h$. 
Denote by $K$ the canonical line bundle over $S$ or,
in other words, the complex dual $T^{*_{\mathbb C}}S$
of the tangent bundle of $S$.
For any integer $n \geq 1$, the line bundle
$L_n = L^{\otimes_n} \otimes K$ is endowed with
a linear map
$\langle \cdot, \cdot \rangle:\overline{L_n} 
\otimes L_n \to  \overline K \otimes K
= \Lambda^2 S \otimes \mathbb C$
which is defined, for $v_1,v_2$ in the fiber of 
$L^{\otimes_n}$ at $x$
and $f_1,f_2$ in the fiber of $K$ at $x$ by
\[\langle v_1 \otimes f_1, v_2 \otimes f_2 \rangle_x 
= (2\mathrm{i})^{-1} h_x^{\otimes_n}(v_1,v_2) \overline{f_1} \wedge f_2.
\]
In particular,
for a non-zero vector $\eta$ of $L_n$ at $x$, the product
$\langle \eta,\eta \rangle_x $ is a strictly positive\footnote{A $2$-form
$\omega$
is strictly positive if locally
$\omega = f\mathrm{d}x \wedge\mathrm{d}y$ for a 
strictly positive function $f$.}
real two-form. Notice that the term $(2\mathrm{i})^{-1}$ 
serves to that purpose since
$\overline{\mathrm{d}z} \wedge \mathrm{d}z = (\mathrm{d}x - \mathrm{i} \mathrm{d}y) \wedge (\mathrm{d}x + \mathrm{i} \mathrm{d}y)
= 2\mathrm{i} \mathrm{d}x \wedge \mathrm{d}y$. 
Denote by $\omega \in \Omega^2(S)$ the curvature of 
$(L,h)$ multiplied by $\mathrm{i}/2\pi$. 
We assume that
$\omega$ is strictly positive so that, in particular,
the degree $\mathrm{deg}(L)$ of $L$ is strictly positive and
the dimension $d_n$ of the space $H^0(S,L_n)$ of holomorphic sections of $L_n$ behaves like $\mathrm{deg}(L)n$
as $n$ goes to infinity.
We replace the sequence of spaces 
of weighted polynomials $\mathcal P_n$
by the sequence of spaces of holomorphic sections $H^0(S,L_n)$. Our process can be defined, taking a basis
$\sigma_1,\dots,\sigma_{d_n}$ of $H^0(S,L_n)$,
as
\[
\frac{1}{Z_n}\langle \sigma_1\wedge \dots \wedge \sigma_{d_n},
\sigma_1\wedge \dots \wedge \sigma_{d_n} \rangle
\in \Omega^{2d_n}(S^{d_n}),\]
for $Z_n$ a normalization constant.
In an equivalent fashion, we may 
denote by $K_n$ the kernel of the orthogonal
projection of
$L^2(S,L_n)$ onto $H^0(S,L_n)$. 
This $K_n$ is understood as a section
of $L_n \otimes \overline{L_n}$
so that it can couple to a section 
of $L_n$ 
via $\langle \cdot, \cdot \rangle$
to give us an element of $L_n$.
It makes sense to 
talk about $\det{[K_n(z_i,z_j)]}_{1\leq i,j\leq d_n}$
using $\langle \cdot, \cdot \rangle$ with the combinatorial definition of the
determinant, so that the law of our process
has the equivalent description
\begin{equation}
\label{eq:LawSurface}
\frac{1}{d_n!}\det{[K_n(z_i,z_j)]}_{1\leq i,j\leq d_n}
\in \Omega^{2d_n}(S^{d_n}).
\end{equation}
Let us consider $(Z_{n,1},\dots,Z_{n,d_n})$
following the law given by \eqref{eq:LawSurface}.
Here are some classical aspects known about
this sequence of point processes.
\begin{itemize}
\item Almost surely, $\lim_{n \to \infty} \frac{1}{n}
\sum_{i=1}^{d_n} \delta_{Z_{n,i}}
= \omega$ in the topology
induced by bounded continuous test functions.
Furthermore, there is a large deviation 
principle associated to it
with rate function given by
the Green function energy associated
to $\omega$. See \cite{BermanLDP}.
    \item
For every smooth function $f:S \to \mathbb R$,
the random variable
$\sum_{i=1}^{d_n} f(Z_{n,i})
- n \int_X f \omega$
converges to a centered\footnote{The absence
of bias comes from the kernel asymptotics on the diagonal.} Gaussian random variable
with variance
$\frac{1}{4\pi}\int df \wedge *df$.
See \cite{10.1007/978-3-030-01588-6_5}. See also \cite{bourgoin}.
    \item 
For every $x \in S$, choosing
a diffeomorphism $\varphi: U \subset S \to V \subset 
T_x S$ that satisfies $\varphi(x) = 0$ and
$d \varphi_x = \mathrm{Id}_{T_x S}$,
the point process $\sum_{i=1}^{d_n} \sqrt{n} 
\delta_{\varphi(Z_{n,i})} 1_{Z_{n,i} \in U}$
converges to the Ginibre point process
in $(T_xS, \omega_x)$. See
\cite{10.1007/978-3-030-01588-6_5}.
\end{itemize}

\noindent
\textbf{Coulomb gases example:}
When $S=\overline{\mathbb C}$ and $L$ is the hyperplane line
bundle $\mathcal O(1)$, any metric $h$ on $S$
is identified with a function
$\mathrm{e}^{-2V}$ on $\mathbb C$ such that
$w \mapsto |w|^{-2} \mathrm{e}^{-2V(1/w)}$ extends smoothly
and it is positive at $0$.
Equivalently, $w \mapsto \log|w| + V(1/w)$
extends smoothly at $0$ or, in other words,
$V$ has a logarithmic singularity
at $\infty$ of weight $-1$.
We have fixed a trivialization
of $L|_{\mathbb C}$ and in this trivialization at
a point $z \in \mathbb C$,
$h_z(1,1) = \mathrm{e}^{-2V(z)}$.
The product of
the section $1 \otimes \dots \otimes 1 \otimes 
dz$ of $L_n=L^{\otimes_n} \otimes K$ with itself is
\[\langle 
1 \otimes \dots \otimes 1 \otimes 
dz , 1 \otimes \dots \otimes 1 \otimes 
dz \rangle_z = \mathrm{e}^{-2n V(z)} dx \wedge dy.\]
Since $L^{\otimes_n} \otimes K$
has degree $n \mathrm{deg}(L) + \mathrm{deg}(K)
= n - 2$, the space of holomorphic sections
is identified with
the space of polynomials of degree less or equal
than $n-2$.
Then $\mathcal P_{n} = H^0(S,L_{n+1})$ and
\[ \langle p_1, p_2 \rangle_{L^2(L_{n+1},h)}
= \int_{\mathbb C} \overline{p_1(z)}
p_2(z) \mathrm{e}^{-2(n+1)V(z)} \mathrm d^2 z.\]
The curvature 
of $(L,h)$ is 
$\overline{\partial} \partial (-2V)
= -i \Delta V dx \wedge dy$ 
so that $\omega = \frac{\Delta V}{2\pi} dx \wedge dy$ whose
integral is necessarily\footnote{The Laplacian
of a smooth function $f: \overline{\mathbb C}\setminus 
\{p_1,\dots,p_k\} \to \mathbb R$ 
with logarithmic singularities
of weights $c_i$ at $p_i$  has integral
(over $\overline{\mathbb C} \setminus \{p_1,\dots,p_k\}$)
equal to $-2\pi$ times $c_1+\dots+c_k$. 
Other way of looking at this phenomenon is
by noticing that the Laplacian over $\overline{\mathbb C}$
is $2\pi \sum_{i=1}^n c_i\delta_{p_i}$ 
plus a regular measure so that the total mass
of $\Delta f$ is still $0$ on
$\overline{\mathbb C}$.} 1.
The spherical ensemble corresponds to the case
$V(z) = \frac{1}{2}\log(1+|z|^2)$.

\medskip

For the generalization, we consider
    $\mathcal{S}$ the set of smooth functions $f:S\setminus\{p_1,\ldots,p_m\}\to\mathbb{R}$ with logarithmic singularities at $p_1,\ldots,p_m$ with weights $c_1,\ldots,c_m$, meaning that for each $j \in \{1,\dots,m\}$ there are holomorphic coordinates $\varphi_j:\mathbb D \to S$ with $\varphi_j(0)=p_j$
    such that $z \mapsto f(\varphi_j(z)) - c_j\log|z|$ extends smoothly at $0$. If this holds for some 
    holomorphic coordinates $\varphi_j$ around $p_j$, it holds for any $\varphi_j$ around $p_j$ because the change of coordinates are holomorphic.
    In the general case it does
    not seem to hold that 
    $|K_n(x,y)| \leq Cn\mathrm{e}^{-c n d(x,y)^2}$
    as in 
    \eqref{eq:kernel-distance}, but it
    does hold that
    \[
    |K_n(x,y)|\leq Cn\mathrm{e}^{-c\sqrt{n}d(x,y)},
    \]
    for some constants $C,c>0$ and all $x,y\in S$,
    see for instance \cite[right after Th.~2.4]{MR2016088}.
    Here $d$ is the distance induced by any
    Riemannian metric on $X$ 
    (any two such such distances are equivalent
    in the compact surface $X$).
    This allows us to show Lemmas
    \ref{lem:lip-variance}, \ref{lem:Covariance} and
    \ref{lem:Independence}. As in the background charge universality section, the fluctuations for
    a single logarithmic singularity is missing.
    We expect this to be similar
    to the spherical ensemble since the support of
    the function
    can be chosen to be in any open set containing
    the point $x$ so that it should depend
    on the local behavior of
    the point process which is the same as the
    one of the spherical ensemble.
    
\subsection{Heuristics on the variance asymptotics}
\label{se:HeuristicsVarianceAsymptotics}

Since our approach to deal with
a single logarithmic singularity
involves a somewhat explicit calculation
 for Lemma \ref{le:Green:cum} with
 $\log(1+|z|^2)$ and for
Subsection \ref{sub:SingleSingularity} with
$\log|z|$, we would like to give some heuristic 
explanation to the variance asymptotics. 
So, in the context of Section \ref{se:Surface}
or Section \ref{se:Background}, we take 
a function $f$ on $S$
having logarithmic singularities at 
$p_1, \dots, p_m$ with weights 
$c_1,\dots,c_m$ and try to explain
\begin{equation}
\mathrm{Var}(L_n(f)) \underset{n\to\infty}{\sim} 
(c_1^2 + \dots + c_m^2)\frac{\log(n)}{4} .
\end{equation}
To simplify matters, we consider
$S = \overline{\mathbb C}$ so that the setting simplifies
to the one in
Section \ref{se:Background}.

\subsubsection{Reduction by localization to the singularities.}
	In order to localize, we consider a compactly supported smooth function $\chi: \mathbb D \to [0,1]$ satisfying $\chi(z) = 1$ for
	$|z| \leq  1/2$. If $f\in\mathcal{S}$, then we can choose neighborhood coordinates
	$\varphi_j:\mathbb D \to S$ whose images are 
    pairwise disjoint
	and	construct the functions $f_j: S \setminus \{p_1,\dots,p_m\} \to \mathbb{R}$ as
	\begin{equation}
    f_j(p) = \left\{ \begin{array}{ll} c_j\chi(\varphi_j^{-1}(p)) \log|\varphi_j^{-1}(p)| & \text{ if }
	p \in \varphi_j(\mathbb D)  			\\
	0 & \mbox{ if } p \notin \varphi_j(\mathbb D) \end{array}\right.  .
    \end{equation}
	The function $g=f-(f_1+\cdots+f_m)$ is smooth so that the random variable
    $L_n(g)-n\int g\mathrm{d}\nu$ converges in distribution as $n\to\infty$
    and its variance converges to the variance of its limit.
	Since we expect $\mathrm{Var}(L_n(f))$ behave logarithmically, the fluctuations of $L_n(g)$ will go to zero after normalization and we can look at $\mathrm{Var}(L_n(f_1+\dots+f_m))$. By construction, $f_1,\ldots,f_m$ have disjoint compact support, and $f_j$ 
    has a single logarithmic singularity
    at $p_j$ with weight $c_j$.
    
\subsubsection{Asymptotic de-correlation of singularities.}
    Since $f_1$ and $f_2$ have disjoint and compact supports, we have
    $\rho=\mathrm{dist}(\mathrm{supp}(f_1),\mathrm{supp}(f_2))>0$.
    Moreover, the off-diagonal exponential decay of the kernel \eqref{eq:kernel-distance} for the spherical ensemble gets substituted by
    $|K_n(p,q)|\leq Cn\mathrm{e}^{-c\sqrt{n}d(p,q)}$, for some constants $C,c>0$ and all $p,q\in S$,
    see for instance \cite[right after Th.~2.4]{MR2016088}. By proceeding as in Lemma 
    \ref{lem:Covariance}, this gives
   \begin{equation}
    |\mathrm{Cov}(L_n(f_1),L_n(f_2))|
	\leq C^2n^2\mathrm{e}^{-2c\rho\sqrt{n}}
	\Bigl(\int|f_1|\mathrm{d}\sigma\Bigr) 
	\Bigl(\int|f_2|\mathrm{d}\sigma\Bigr). 
	\end{equation}
    The integrals above are finite because $\log\left|\cdot\right|$ is integrable near $0$
	so that the covariance vanishes as $n\to\infty$. Then,
	\begin{equation}
    \mathrm{Var}(L_n(f_1+\cdots+f_m))
    \underset{n\to\infty}{\sim}
    \mathrm{Var}(L_n(f_1))+\cdots+\mathrm{Var}(L_n(f_m))    
	\end{equation}
    and the asymptotic variance problem is therefore further reduced to the single singularity case $m=1$.

\subsubsection{Reduction to variance of truncation.} 
    Let $f:S\setminus\{p\}\to\mathbb{R}$ be
    one of the $f_j$ so
    that its support is contained in the coordinate
    neighborhood
    $\varphi_j(\mathbb D)$ and we may
    rescale it so that $c_j=1$. The variance is
	\begin{equation}
	\mathrm{Var}(L_n(f)) 
    =\frac{1}{2}\int_{\mathbb{D}\times\mathbb{D}}(\chi(z)\log|z|-\chi(w)\log|w|)^2|K_n(z,w)|^2\mathrm{d}
    \sigma(z)\mathrm{d}\sigma(w).
    \end{equation}
	   To proceed, we use the truncation decomposition
	\begin{equation}
        \log|z| = \ell_n(z) +  r_n(z)
        \quad\text{where}\quad
        \ell_n(z) = \max\big(\log|z|,-\tfrac{1}{2}\log(n)\big)
        \quad\text{and}\quad
	      r_n(z) = \log|z| - \ell_n(z).
    \end{equation}
	We have $\ell_n(z)=\log|z|$ if $|z| \geq 1/\sqrt{n}$ while
	$\ell_n(z)=-\frac{1}{2}\log n$ if $|z| \leq 1 /\sqrt{n}$,
	which tells us that $r_n(z)=0$ if $|z| \geq 1/\sqrt{n}$.
	For the variance associated to $r_n$ we have
	\begin{equation}
	\int_{\mathbb D \times \mathbb D} 
	(\chi(z)r_n(z) - \chi(w)r_n(w))^2|K_n(z,w)|^2\mathrm{d}\sigma(z)\mathrm{d}\sigma(w)
	\leq 2\int_{\mathbb D}(\chi(z)r_n(z))^2K_n(z,z)\mathrm{d}\sigma(z),
	\end{equation}
	where in the last term we have used	that $K_n$ is a projection and we have discarded
	an integral of a negative term.	Now, we use that $K_n(z,z) \leq Cn$ for some $C>0$ 
	and that $\sigma$ has a bounded density on any fixed compact set of $\mathbb D$ to obtain
	\begin{equation}
	\int_{\mathbb D}
	(\chi(z)r_n(z))^2 K_n(z,z) \mathrm{d}\sigma(z)
	\leq  \widetilde C n \int_{|z| \leq \frac{1}{\sqrt n}} |r_n(z)|^2 \mathrm d^2 z 			
	= 2\pi \widetilde C n \int_0^{\frac{1}{\sqrt n}}(\log (\sqrt n r))^2 r\mathrm d r			
    =2\pi \widetilde C \int_0^1 s(\log s)^2 \mathrm d s<\infty.
	\end{equation}
    Therefore $\sup_{n\geq1}\mathrm{Var}(L_n(\chi r_n))<\infty$. It follows that if 
	$\mathrm{Var}(L_n(\chi\ell_n))\to+\infty$ as $n\to\infty$ then
	\begin{equation}
		\mathrm{Var}(L_n(f))
		\underset{n\to\infty}{\sim}
    	\mathrm{Var}(L_n(\chi\ell_n)).
     \end{equation}						
     This reduces the problem to the computation of the asymptotic variance for a sequence of bounded
     functions.

\subsubsection{Asymptotic variance of the sequence of truncated functions.} Recalling that
	\begin{equation}
	\mathrm{Var}(L_n(f))
    =
	\frac{1}{2}\int\Big(\int(f(z)-f(w))^2|K_n(z,w)|^2\mathrm d\sigma(z)\Big)\mathrm d\sigma(w),												\end{equation}
	we may follow the ideas of
   \cite[Theorem 5.8]{10.1007/978-3-030-01588-6_5} 
   to exploit the asymptotics
    \begin{equation}
      \frac{1}{n^2}\Big| K_n \Big(z + \frac{v_1}{\sqrt n} ,z + \frac{v_2}{\sqrt n}\Big) \Big|^2
	\xrightarrow[n \to \infty]{} C_z^2\mathrm{e}^{ -\pi C_z \sigma_z |v_1-v_2|^2}
    \quad\text{with}\quad
    C_z = \lim_{n \to \infty} \frac{1}{n} K_n(z,z)
    \end{equation}
	and where $\sigma_z$ is the density at
	$z$ of $\sigma$ with respect to Lebesgue measure \cite[Theorem 1.1]{MR1794066}. 
	We obtain %
	\begin{equation}
	\mathrm{Var}(L_n(f))
	\underset{n\to\infty}{\sim}
	\frac{1}{4}\int
	\Big(
	\frac{1}{ 2 \pi \sigma_w} \| \nabla f_w \|^2 \Big) \mathrm d \sigma(w)									
	= \frac{1}{4 \pi} \int \|\nabla f_w\|^2 \mathrm d^2 w
	\end{equation}
    holds even for a function $f$
    that depends nicely enough on $n$.
	So, if this asymptotics holds for $f=\chi\ell_n$ despite its dependency on $n$,
    then $\mathrm{Var}(L_n(\chi\ell_n))$ would behave like
	\begin{equation}
	\frac{1}{4 \pi} \int \|\nabla (\chi \ell_n)(w)\|^2 \mathrm d^2 w
	=\frac{1}{4 \pi} \int_{|w| \geq \frac{1}{2}} \|\nabla (\chi \log|\cdot|)(w)\|^2 \mathrm d^2 w	+ \frac{1}{4\pi} \int_{\frac{1}{\sqrt n}}^{1/2} \frac{1}{r^2} (2\pi r \mathrm d r)					
	=O(1) + \frac{1}{4}\log(n).
	\end{equation}

{
  \footnotesize
  \bibliographystyle{abbrvnat}
  \bibliography{radius}
}

\begin{figure}[p]
\centering
\includegraphics[width=.45\textwidth]{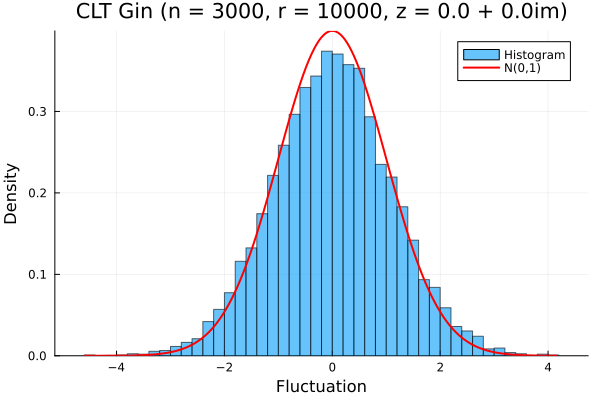}
\includegraphics[width=.45\textwidth]{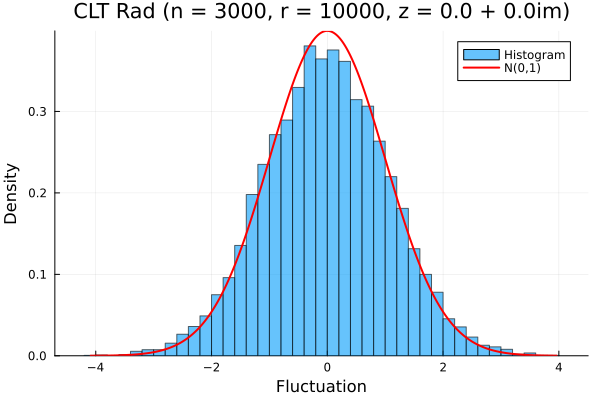}\\
\includegraphics[width=.45\textwidth]{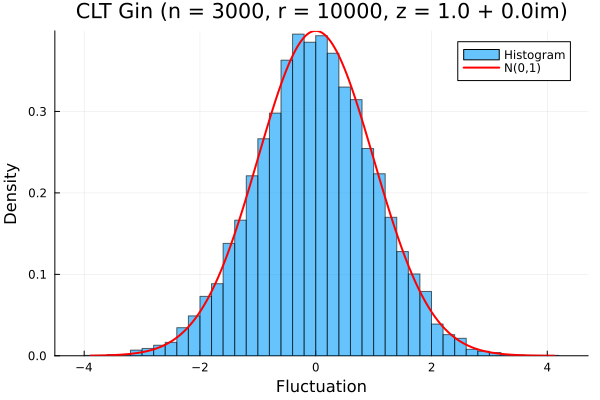}
\includegraphics[width=.45\textwidth]{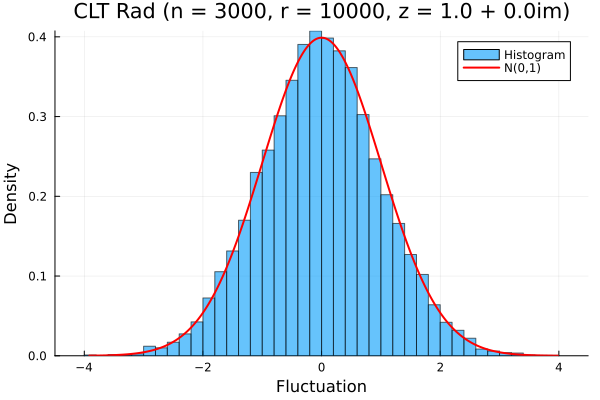}\\
\includegraphics[width=.45\textwidth]{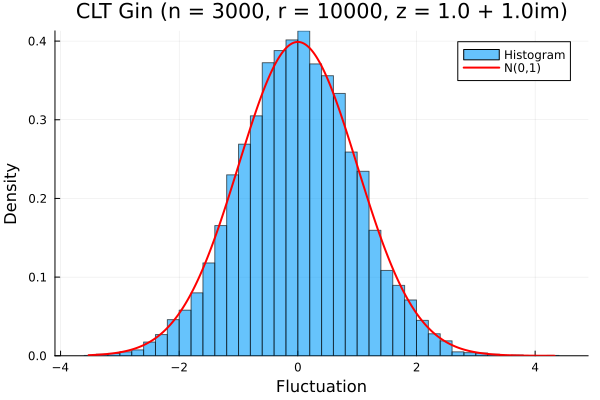}
\includegraphics[width=.45\textwidth]{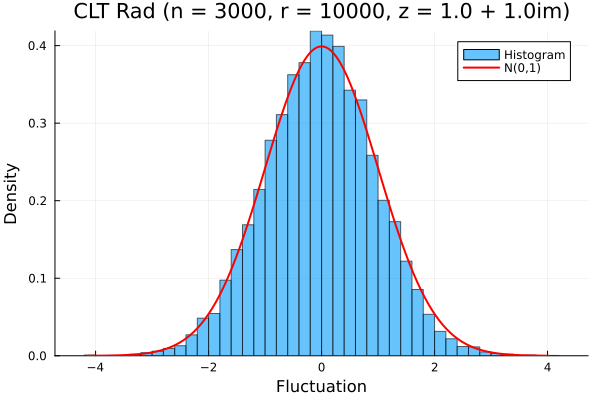}
\caption{\label{fi:LogPot:CLT}Numerical experiment : histogram of a sample of size $r$ of the left
  hand side of the CLT in Corollary \ref{co:LogPot:CLT}, for a large value of $n$, and
  three different values of $z$, together with the density of $\mathcal{N}(0,1)$. On the left, the entries of $A$ and $B$ are i.i.d.
  $\mathcal{N}_{\mathbb{C}}(0,1)$, while on the right they are i.i.d.\ uniform
  on $\{(\pm1,\pm1)\}/\sqrt{2}$. This illustrates the universality of the CLT
  beyond \ref{it:condensity} and \ref{it:condmom4}.}
\end{figure}

\begin{figure}[p]
  \centering
  \begin{tikzpicture}[scale=2.5]
    \draw[thick] (0,0) circle(1);

    \draw[->] (-1.2,0) -- (4,0) node[right] {};
    \draw[->] (0,-0.2) -- (0,1.4) node[above] {};

    \coordinate (N) at (0,1);                      %
    \coordinate (O) at (0,0);                      %
    \coordinate (P) at ({sqrt(2)/2}, {sqrt(2)/2}); %
    \coordinate (P') at ({sqrt(2)/2},0);           %
    \coordinate (P'') at (0,{sqrt(2)/2},0);        %
	\coordinate (Q) at ({1/(sqrt(2)-1)}, 0);       %

    \draw[dashed] (N) -- (Q);          %
    \draw[dashed] (P) -- (Q);      %
    \draw[dashed] (P) -- (P');         %
    \draw[dashed] (P) -- (P'');        %

    \fill (O) circle(0.04) node[below left] {$O$};
    \fill (N) circle(0.04) node[above right] {$N$};
    \fill (P') circle(0.04) node[above right] {$P'$};
    \fill (P'') circle(0.04) node[below right] {$P''$};
    \fill[blue] (P) circle(0.04) node[above right] {$P$};
    \fill[blue] (Q) circle(0.04) node[below] {$Q$};
  \end{tikzpicture}
  \caption{\label{fi:stereo}The point $Q\in\mathbb{R}$ is the image of
    $P\in\mathbb{S}^1\setminus\{N\}$ by the 1D stereographic projection $T$
    with respect of the north pole $N$. The image of $N$ is $\infty$. The
    Pythagoras Theorem for the triangle $ONQ$ gives $1+OQ^2=NQ^2$ since
    $ON=1$. The Thales Theorem for the alignments $QP'O$ and $QPN$ gives
    $OQ/OP'=NQ/NP$. Combining the two and eliminating $NQ$ gives
    $1+OQ^2=OQ^2 NP^2/OP'^2$, hence $OQ^2=OP'^2/(NP^2-OP'^2)$. The Pythagoras
    Theorem again, this time for the triangle $NPP''$, gives
    $NP^2=PP''^2+NP''^2=OP'^2+NP''^2$, therefore $OQ^2=OP'^2/NP''^2$, which is
    the 1D formula $T(x)=x_1/(1-x_2)$ for
    $(x_1,x_2)\in\mathbb{S}^1\setminus\{e_2\}$ with the Cartesian transcription
    $O=0$, $N=e_2$, $P=x$, $Q=T(x)$. Finally, the 2D formula
    $T(x)=(x_1,x_2)/(1-x_3)$ for
    $(x_1,x_2,x_3)\in\mathbb{S}^2\setminus\{e_3\}$ follows from the 1D formula
    by the coplanarity of $0$, $e_3$, $x$, and $T(x)$. See
    \cite{zbMATH07645442} for more in this spirit. More generally, in
    arbitrary dimension $d\geq1$, $T(x)=(x_1,\ldots,x_{d-1})/(1-x_d)$,
    $x\in\mathbb{S}^{d-1}\setminus\{e_d\}$.}
  \begin{tikzpicture}[scale=2.5]
    \draw[dotted, thick] (0,0) circle(1);

    \draw[dotted,->] (-1.2,0) -- (4,0) node[right] {};
    \draw[dotted,->] (0,-0.2) -- (0,1.4) node[above] {};

    \coordinate (N) at (0,1);                       %
    \coordinate (O) at (0,0);                       %
    \coordinate (P1) at ({sqrt(3)/2}, {1/2}); %
	\coordinate (Q1) at ({sqrt(3)}, 0);       %
    \coordinate (P2) at ({1/2}, {sqrt(3)/2});       %
	\coordinate (Q2) at ({1/(2-sqrt(3))}, 0);       %

    \draw[dotted] (N) -- (Q1);        %
    \draw[dotted] (N) -- (Q2);        %
    \draw[solid, very thick] (P1) -- (P2); %
    \draw[solid, very thick] (Q1) -- (Q2); %

    \fill (O) circle(0.04) node[below left] {$O$};
    \fill (N) circle(0.04) node[above right] {$N$};    
    \fill (P1) circle(0.04) node[above right] {$P_1$};
    \fill (Q1) circle(0.04) node[below] {$Q_1$};
    \fill (P2) circle(0.04) node[above right] {$P_2$};
    \fill (Q2) circle(0.04) node[below] {$Q_2$};    
  \end{tikzpicture}    
  \caption{The closer the chordal line segment $[P_1,P_2]$ is to the pole $N$, the greater the length of its stereographic image $[Q_1,Q_2]$. The metric splitting \eqref{eq:dists} reads $Q_1Q_2=\frac{1}{2} P_1P_2 /(\frac{1}{2} P_1N\times \frac{1}{2}P_2N)$.
  \label{fi:dists}}
\end{figure}

\begin{figure}[p]
    \centering
    \includegraphics[width=0.33\linewidth]{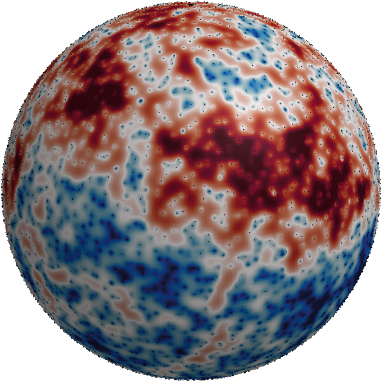}
    \caption{Numerical experiment : sample of the Green field ${(L_n(g_p))}_{p\in\mathbb{S}^2}$ with $n=5000$.}
    \label{fig:placeholder}
\end{figure}

\end{document}